\documentclass[11pt, leqno]{amsart}
\usepackage[utf8]{inputenc}
\usepackage{geometry}
\usepackage{etoolbox}
\usepackage{lipsum, bm}
\usepackage[english]{babel}
\usepackage{csquotes}
\usepackage{mathrsfs, enumerate}
\usepackage{mathtools, tikz, pgfplots, subcaption, mathdots}
\usepackage{pythonhighlight}
\usepackage{xcolor, multicol}
\usepackage{yfonts, float}
\usetikzlibrary{patterns}
\usepackage{thmtools}
\usepackage[makeroom]{cancel}

\usepackage{tocvsec2}
\tikzset{cross/.style={cross out, draw=black, minimum size=2*(#1-\pgflinewidth), inner sep=0pt, outer sep=0pt},
	cross/.default={2pt}}
\usetikzlibrary{shapes.misc}

\usepackage[breaklinks=true]{hyperref}
\usepackage{cleveref}

\newcommand{\height}[0]{\mathrm{ht}}
\newcommand{\wt}[0]{\mathrm{wt}}
\newcommand{\ad}[0]{\mathrm{ad}}

\usepackage{amsmath, amsthm, amssymb}

\theoremstyle{plain}
\newtheorem{theorem}{Theorem}[section]
\newtheorem{lemma}[theorem]{Lemma}
\newtheorem{prop}[theorem]{Proposition}

\newtheorem{thmx}{Theorem}[section]

\theoremstyle{definition}
\newtheorem{definition}[theorem]{Definition}

\newtheorem{example}[theorem]{Example}

\newtheorem{question}{Question}
\newtheorem*{p-psp}{parabolic-PSP}
\newtheorem{corollary-C}[theorem]{Corollary to Theorem C}

\newtheorem{remark}[theorem]{Remark}

\newtheorem*{claim}{Claim}
\numberwithin{equation}{section}

\newtheoremstyle{problem}{5pt}{5pt}{}{}{\normalfont}{\textbf{:}}{.5em}{}
\theoremstyle{problem}

\newtheorem{note}{\textbf{Note}}

\begin{document}
	\title[Characters of modules over negative rank-2 Borcherds--Kac--Moody Lie algebras]{Characters of modules over negative rank-2 Borcherds--Kac--Moody Lie algebras}

\author{Souvik Pal, Supravat Sarkar and G. Krishna Teja}
    
	\begin{abstract}
		Let $\mathfrak{g}=\mathfrak{g}(A)$ be the Borcherds--Kac--Moody Lie algebra (BKM LA) for a BKM Cartan matrix $A$ that is filled by  negative integers. 
        Fix a Cartan subalgebra $\mathfrak{h}$ of $\mathrm{g}$ and the classical cone of dominant integral weights $P^+\subset \mathfrak{h}^*$.
The non-integrable simple highest weight $\mathfrak{g}$-modules $L(\mu)$'s widely studied were those by Naito ([{\it Trans. Amer. Soc.}, 1995]), for $\mu$'s dot-linked to $P^+$-translates of sums $- \sum_{j\in J}\alpha_j$ of mutually orthogonal and imaginary simple roots $\alpha_j$'s.

    Recently, we computed weights of all highest weight $\mathfrak{g}$-modules $V$'s (over all BKM LA's), and character of $L(\rho)$ for Weyl vector $\rho$.
    These needed a family of ``integrable''  $L(\mu)$'s for $\mu$'s inside our novel signed-dominant-integral cone $P^{\pm}$ (which generalizes $P^+$).
Pairings $\mu(\alpha_i^{\vee})\leq 0$ for $\mu\in P^{\pm}$ are multiples of $\frac{A_{ii}}{2}$ for all $i$.
   Nevertheless, $L(\mu)$ contains  ``Chevalley--Serre relations'' $f_i^{\frac{2}{A_{ii}}{\mu(\alpha_i^{\vee})}+1}L(\mu)_{\mu}=0$, which seem to be previously unstudied and even in Naito's works.
    
    This paper initiates in rank-2, the study of module structures and maximal vectors (or Verma embeddings) in Verma covers $M(\mu)$'s of $L(\mu)$'s for $\mu\in P^{\pm}$.
    Our goal in this is to explore in weight spaces of those Vermas, the strictness, or else a uniform equality, of lower bounds by Kac and Kazhdan ([{\it Adv. Math.}, 1979]) for  count of linearly independent maximal vectors. 
    We obtain presentations and characters of all $V$'s when Kac--Kazhdan equation has unique solution in the interior of root-cone.
    This builds on results of Kac and Kazhdan in crucial unique solution case.
\end{abstract}
	\maketitle
	\settocdepth{section}
	\tableofcontents
	\allowdisplaybreaks
    \section{Introduction, and Statements and Discussions of Main Results} \label{Section 1}
Let the sets of complex numbers, real numbers, non-negative reals, integers, non-negative integers and natural numbers be denoted by $\mathbb{C}\supset  \mathbb{R} \supset \mathbb{R}_{\geq 0}\supset\mathbb{Z}\supset \mathbb{Z}_{\geq 0}\supset \mathbb{N}$.
Throughout the paper, the Lie algebras and their modules we study are over $\mathbb{C}$.
 
 Let $\mathfrak{g}=\mathfrak{g}(A)$ be the Borcherds--Kac--Moody Lie algebra (BKM LA), for any fixed BKM-Cartan (BKMC) matrix $A$ whose rows and columns are indexed by $\mathcal{I}$.
$\mathfrak{g}$ has the Lie bracket $[.,.]$ and universal enveloping algebra $U(\mathfrak{g})$. 
 For  $i\in \mathcal{I}$, recall $A_{ii}$ is either $2$ or $0$ or is any negative real number.
  Let $\mathfrak{h}$ be a fixed Cartan subalgebra in $\mathfrak{g}$, and relative to it, let $\Pi=\{\alpha_i\ |\ i\in \mathcal{I}\}$ and $\Pi^{\vee}= \{\alpha_i^{\vee}\ |\ i\in \mathcal{I}\}$ be the simple roots and simple co-roots.  
Let $\{e_i \text{ (raising)},\ f_i \text{ (lowering)},\ \alpha_i^{\vee}\ |\ i\in \mathcal{I}\}$ be the Chevalley generators of $\mathfrak{g}$.
 Let $\mathfrak{g}=\mathfrak{n}^+\oplus \mathfrak{h}\oplus \mathfrak{n}^-$ be the triangular decomposition of $\mathfrak{g}$.
Let $\preceq$ be the usual {\it partial order} on $\mathfrak{h}^*$.
See \cite{Wakimoto} and \cite[Chapter 11]{Kac book} for meticulous introduction to BKM LAs.
  
 R. Borcherds (\cite{Borcherds J. Alg, Borcherds Liesuper}) introduced BKM LAs generalizing his Monster vertex operator algebras (modules over affine KM LAs studied in \cite{Lepowsky E8, Lepowsky Vertex operators}); the latter play a central role in string theory and conformal field theory, quantum chromodynamics and quantum gravity (\cite{Witten}).
This was for settling the Conway and Norton's Moonshine conjecture(s) \cite{Conway--Norton, Thompson 2}.
BKM LAs are crucial settings of  contragredient LAs studied by Kac--Kazhdan \cite{Kac--Kazhdan}.

Given $\lambda\in \mathfrak{h}^*$, let $M(\lambda)\twoheadrightarrow L(\lambda)$ be the Verma and simple highest weight $\mathfrak{g}$-modules respectively, both with highest weight (which we abbreviate as h.w.) $\lambda$.
For a highest weight module $V\twoheadleftarrow M(\lambda)$ and element $\mu\in \mathfrak{h}^*$, let $V_{\mu}:=\big\{ v\in V \ \big|\ h\cdot v= \mu(h)v\ \text{for all } h\in \mathfrak{h} \big\}$ be the $\mu$-weight-space of $V$.

Recently,  \cite{T-P} introduced {\it signed-dominant-integral cone} $P^{\pm}$ \big(Definition \ref{Defn integrability of lambda}(a)\big) for following: 
 a)~Study of full Chevalley--Serre relations in all $L(\mu)$'s, by {\it maximal vectors} arising in \eqref{Eqn 1-dim maxl vect} below. 
 b)~To explicitly compute weight-set $\wt V$ of every highest weight $\mathfrak{g}$-module (h.w.m.) $V$.
 This extended in one stroke and completed the qualitative picture of $\wt V$'s that was sought-for and built across a series of works by Khare, Dhillon and of the first and third authors \cite{Khare_JA,  Khare_Trans, Dhillon_arXiv, Khare_Ad, MDWF, WFHWMRS}.\\
 c) Step b) required working {\it along full $P^{\pm}$-directions}, with maximal vectors (killed by $\mathfrak{n}^+$) in \eqref{Eqn 1-dim maxl vect} that are lost during $M(\lambda) \twoheadrightarrow V$.
 Those lost vectors generalize {\it integrability}, or invariance of $V$'s under subgroups of Weyl group of $\mathfrak{g}$, and are studied under the concept of {\it Holes in $V$'s} in \cite{T-P}.

 \begin{definition}\label{Defn integrability of lambda}
    (a) $P^{\pm}\ :=\  \Big\{ \lambda\in \mathfrak{h}^*\ \Big| \ \lambda(\alpha_i^{\vee})\in \frac{A_{ii}}{2}\mathbb{Z}_{\geq 0}\ \text{for all } i\in \mathcal{I} \Big\} $.  (b) The {\it classical dominant-integral weight-cone} over contragredient $\mathfrak{g}$ (\cite{Kac--Kazhdan}), yielding foundational and well-studied integrable simple h.w.m.s, is $P^+:= \big\{ \lambda\in\mathfrak{h}^*\ \big|\  \lambda(\Pi^{\vee})\subset \mathbb{R}_{\geq 0} \text{ and } \lambda\big(\{\alpha_i^{\vee}\ |\ A_{ii}=2 \}\big)\subset \mathbb{Z}_{\geq 0} \big\}$.
    And $\alpha_i$'s with $A_{ii}=2$ are called {\it real simple roots}, and the remaining are {\it imaginary simple} roots.
    \end{definition}
    \begin{note}
 $P^{\pm}$ generalizes $P^+$, notably to accommodate crucial Weyl vector $\rho \notin P^+$ as $\rho(\alpha_i^{\vee})=\frac{A_{ii}}{2}$ $\forall\ i$.
 \cite[Example 4.5]{T-P} shows neither $P^+\not\subset P^{\pm}$ nor $P^{\pm}\not\subset P^+$ in general.
This paper focus on the side where all simple roots are  imaginary.
In such negative settings of all $A_{ii}\leq 0$, generically $P^+\cap P^{\pm}=\{0\}\subset\mathfrak{h}^*$.
The naturality of $P^{\pm}$-cone is seen by following ``complete $\mathfrak{sl}_2$-theory'' picture:
 \end{note}
 \begin{lemma}[{\cite[Lemma 2.1]{T-P}}]\label{Lemma rank-1 BKM rep. theory}
Fix any contragredient $\mathfrak{g}(A)$ and $i\in \mathcal{I}$ (note $A_{ii}\in \mathbb{C}$), and $\lambda\in \mathfrak{h}^*$.
Then $f_i^n M(\lambda)_{\lambda}= M(\lambda)_{\lambda-n \alpha_i}$ are maximal $\iff\ n= 
\begin{cases}
\frac{2}{A_{ii}}\lambda(\alpha_i^{\vee})+1,\ &\text{ if }A_{ii}\neq 0,\\    
\text{any natural number} &\text{ if }A_{ii}\ = \ \lambda(\alpha_i^{\vee})\ = \ 0.
\end{cases} $
\end{lemma}
\noindent
{\bf Example.}  $f_1^5 M(-4)_{-4}$ over $\mathfrak{g}\big([-2]_{1\times  1}\big)$ are maximal.
So $\dim L(-4)=5< \infty$; despite $2\alpha_1\notin P^+$.
\begin{remark}\label{Remark all 1-dim maxl vect}
Throughout, we assume maximal vectors to be non-zero.
Our ``cumulative $(\mathfrak{sl}_2\oplus \cdots \oplus \mathfrak{sl}_2)$-theory'' reveals  maximal vectors in 1-dim. weight spaces in $M(\lambda)$ by  following monomials:
\begin{equation}\label{Eqn 1-dim maxl vect}
f_{r_1}^{\frac{2}{A_{r_1r_1}}\lambda(\alpha_{r_1}^{\vee})+1}\cdots f_{r_k}^{\frac{2}{A_{r_kr_k}}\lambda(\alpha_{r_k}^{\vee})+1}\times f_{r_{k+1}}^{n_{k+1}}\cdots f_{r_{l}}^{n_l}\times f_{r_{l+1}}^{\lambda(\alpha_{r_{l+1}}^{\vee})+1}\cdots f_{r_s}^{\lambda(\alpha_{r_s}^{\vee})+1} \cdot m_{\lambda}\quad\text{where}
\end{equation}
i) $f_{r_1},\ldots, f_{r_s}$ commute; 
ii) $A_{r_1, r_1}, \ldots, A_{r_k, r_k}<0,\ \  A_{r_{k+1}, r_{k+1}}= \cdots = A_{r_l, r_l}=0,\ \ A_{r_{l+1}, r_{l+1}}=\cdots= A_{r_{s}, r_s}=2$;  iii) $\lambda(\alpha_{r_1}^{\vee}),\ldots,  \lambda(\alpha_{r_k}^{\vee}), \lambda(\alpha_{r_{l+1}}^{\vee}),\ldots, \lambda(\alpha_{r_s}^{\vee}) $ are suitably scaled integers by Lemma \ref{Lemma rank-1 BKM rep. theory}, and $n_{k+1}, \ldots, n_{l}\ \in \mathbb{Z}_{\geq 0}$ are arbitrary; and iv) $m_{\lambda}\neq 0$ is a fixed highest weight vector in $M(\lambda)_{\lambda}$.
\end{remark}\smallskip
\noindent
{\bf Goals.} Our paper focuses on  aspects P1)--P2) along $P^{\pm}$-platform, for multiple applications that we discuss in this section, including the following: Fix $\lambda\in P^{\pm}$. \smallskip\\
P1) Structure of Verma modules $M(\lambda)$ and construction of their maximal vectors. (See Theorem~\ref{Theorem maximal vectors}.)
P2) Character of $L(\lambda)$, and of all quotients $V \twoheadrightarrow M(\lambda)$ in some cases. (See Theorem \ref{Theorem composition series}.)\smallskip\\
Throughout the paper, we study above for $\lambda\in P^{\pm}$, over rank-2 settings of $\mathfrak{g}(A)$ for BKMC matrices
\begin{equation}\label{Eqn defn cartan matrix A(a,b,c,d)}
  A(b,a,c,d)\ :=\ 
\begin{bmatrix}
   -b & -a\\
   -c & -d
\end{bmatrix},\   \text{for }\ a,...,d\in \mathbb{N}.\quad 
\left(\text{Fundamental type is }A(2):= \begin{bmatrix}
    -2 & -1\\
    -1 & -2 
\end{bmatrix}.\right)
\end{equation}
$A(b,a,c,d)$ is simplest and important setting (in view of Remarks \ref{Remark rank 2 settings} and \ref{Remark rank-2 advantages}) to initiate the study of $L(\lambda)$ for $\lambda\in P^{\pm}$.
 These simples are previously unstudied to the best of our knowledge and of experts M. Wakimoto, A. Khare and S. Viswanath.
 We begin by the first problem of our focus.
\begin{question}\label{Question Lmax simple?}
    Consider $L(\lambda)= \frac{M(\lambda)}{\big\langle f_1^{\lambda(\alpha_1^{\vee})+1} M(\lambda)_{\lambda},\ f_2^{\lambda(\alpha_2^{\vee})+1}M(\lambda)_{\lambda} \big\rangle}$ over $\mathfrak{sl}_3(\mathbb{C})$ for $\lambda\in P^+$.
    This quotienting imparts integrability, Weyl group symmetry, finite-ness, etc to $L(\lambda)$.
    Similarly, ``finite-ness with non-integrability'' is seen for $L(\lambda)$ with $\lambda\in P^{\pm}$ in the special case over rank-1 BKM $\mathfrak{g}(A)$'s.\smallskip\\
    Consider $\mathfrak{g}\big(A(2)\big)$ and $\lambda\in P^{\pm}$ with $\lambda(\alpha_i^{\vee})= -(M_i-1) \in  \mathbb{Z}_{\leq 0}$ for all $i$.
    ($M_1=M_2=1 \iff \lambda=0$.)\\
   Killing all $f_i^{- (-1)(M_i-1)+1}M(\lambda)_{\lambda}$ from $M(\lambda)$ leads to loss of weight-lines $\lambda- (M_i+\mathbb{Z}_{\geq 0})\alpha_i$.\smallskip\\
 (a) Now what cumulative effects or properties (say, similar to those above in type $A_2$) are casted\\
 \hspace*{0.5cm} on  $L^{\max}(\lambda):= \frac{M(\lambda)}{\big\langle f_1^{M_1} M(\lambda)_{\lambda},\ f_2^{M_2}M(\lambda)_{\lambda} \big\rangle}$ by the simultaneous killing of the two maximal vectors?\\
 (b) Is $L^{\max}(\lambda)$ in the above case simple?\quad
 (c) If $\lambda=\rho$ (so $M_1=M_2=2$), is $L^{\max}(\rho)$ simple?
\end{question}

 \cite[\S 3.2]{T-P} initiated computing characters of  $L(\lambda)$'s for $\lambda\in P^{\pm}$.
Notably \cite[Theorem D]{T-P} recorded characters of $L(\rho)$ and also of all h.w.m.s $V$'s, for negative BKMC matrix $A=A(n)$ \big(similar to $A(2)$\big) which is type-$A_n$ Cartan matrix but with $-2$'s replacing $2$'s.
Remark \ref{Remark on presentation of new simples} addresses questions (a)--(c); its part 2) provides first motivation to goal P1).
Our first motivation to P2) is as follows.
\begin{remark}
Kac--Kazhdan \cite{Kac--Kazhdan} quoted the study of $L(-\rho)$ and  simplicity of $M(-\rho)$ over contragredient $\mathfrak{g}(A)$'s to be important.
\cite{Kac--Kazhdan} conjectured on characters of $L(-\rho)$ using its free imaginary root-directions; see \cite{Hayashi} for its proof in some cases.
However, none of $L(\rho)\twoheadleftarrow M(\rho)$ \big(elements P1)--P2)\big) for h.w. $+\rho$, were previously studied to the best of our knowledge, even in  pioneering works \cite{Naito 1}--\cite{Naito BGG 2} of S. Naito which extended Kazhdan--Lusztig theory to Borcherds setting. 
In our notations, for h.w.s $\mu$'s dot-linked to weights of the maximal vectors in \eqref{Eqn 1-dim maxl vect}, Naito determined characters of (non-integrable) $L(\mu)$'s \underline{only if $\lambda(\alpha_{r_1}^{\vee}) =\cdots = \lambda(\alpha_{r_k}^{\vee})=0$}.
When $\mu\in P^+$ as well, only such simple roots orthogonal to $\lambda$ play  a role in Weyl--Kac--Borcherds (WKB) character formulas; see \cite{Kac book, Wakimoto}.
With this limitation of classical techniques, our approach is based on solving Kac--Kazhdan equation \eqref{Eqn dot action by any positive root} below.
It might be interesting to extend Naito's results for $\lambda$'s with at least one of $\lambda(\alpha_{r_1}^{\vee}),\ldots, \lambda(\alpha_{r_k}^{\vee})$ negative as in Lemma \ref{Lemma rank-1 BKM rep. theory}.
\end{remark}
 
\begin{remark}\label{Remark on presentation of new simples}
1) \cite[Theorem D]{T-P} solved Question \ref{Question Lmax simple?}(b) positively only for $\lambda=\rho$, over $\mathfrak{g}\big(A(n)\big)$.
2) On the other hand, \cite[Corollary 3.8 to Theorem C]{T-P} over $A(b,a,c=a,d)$, negatively answers that problem for $\lambda\in P^{\pm}\setminus \{\rho\}$.
That was done when $M_1,M_2\geq 2$ and when $\lambda-\alpha_1-\alpha_2$ satisfies {\it Kac--Kazhdan equation} \eqref{Eqn dot action by any positive root} below, by computing the maximal vector 
\begin{equation}\label{Eqn maxl vector (1,1)}
2af_1f_2m_{\lambda} +b(M_2-1)[f_2,f_1]m_{\lambda}\  \in\ \ M(\lambda)_{\lambda-\alpha_1-\alpha_2}\setminus  \ U(\mathfrak{g})\cdot \{f_1^{M_1} M(\lambda)_{\lambda},\ f_2^{M_2}M(\lambda)_{\lambda}\}.
\end{equation} 
3) At the same time, \cite[ Lemma 6.2; Lemma 6.3 or Proposition 3.7 (a); and  Lemma 6.6 Plots (C) and (D)]{T-P} show the following properties to be common to $A(2)$ and type $A_2$ cases:
Monomial type maximal vectors $f_2^nf_1^{M_1}M(\lambda)_{\lambda}$ in $M(\lambda)$'s by dot-actions of ``reflections about imaginary simple roots''; finite sized-ness of those ``dot-orbits''; and finitely many exponentials (involving sometimes fewer weights than whole dot-orbits) in numerators in characters of $L(\lambda)$.
\end{remark}
Our second motivation to study maximal vectors is by a classical result of Kac and Kazhdan in Proposition \ref{Prop KK}\eqref{maxl vect count inequality} on locating and counting maximal vectors; see Question \ref{Question strict inequality?}. 
The procedure involved is exploring solutions to generalized {\it strong dot-linkage condition}  \eqref{Eqn dot action by any positive root} below by \cite{Kac--Kazhdan}.
Namely, if $M(\lambda)_{\mu}$ has a maximal vector (i.e. $\mu$ is a {\it maximal weight}), then there exist positive roots $\beta_1,\ldots, \beta_k$ for $\mu$ such that every partial sum in $\mu=\lambda- \sum_{t=1}^k n_t\beta_t$ satisfies the following condition.
\begin{align}\label{Eqn dot action by any positive root}
\underset{(\text{as referred to in \cite[Theorem 4.2]{Malikov}})}{\textbf{Kac--Kazhdan equation :}} \qquad  \bigg( \lambda+\rho-\sum_{t=1}^i n_t\beta_t\ ,\ 2\beta_{i+1}\bigg) = n_{i+1}(\beta_{i+1}, \beta_{i+1})\ \ \forall\ i. 
\end{align}
For our purpose, we recalled \eqref{Eqn dot action by any positive root} for symmetrizable $\mathfrak{g}$ with its invariant form $(.,.)$.
In finite type, celebrated Harish-chandra--BGG theories show the maximal weights when $\lambda\in P^+$ to be the dot-conjugates of $\lambda$ the under Weyl group.
Quantitatively, span of all the maximal vectors in space $M(\lambda)_{\mu}$ can have dimensions more than unity, as follows; unlike in finite type (\cite[Theorem 4.2]{Hump BGG}) by Noetherianess of $U(\mathfrak{g})$.
Let $\Delta$ and $\Delta^+$ be the set of roots and positive roots in $\mathfrak{g}$ respectively.
\begin{prop}[{\cite[Proposition 4.1]{Kac--Kazhdan}}]\label{Prop KK}
    Fix a symmetrizable contragredient $\mathfrak{g}$ and $\lambda\in \mathfrak{h}^*$.
    Fix a solution $\mu=\lambda-\beta_1$ to \eqref{Eqn dot action by any positive root} for $\beta_1\in \Delta^+$.
Recall, $\mu$ is a minimal solution if no $\mu'\precneqq \mu$ satisfies \eqref{Eqn dot action by any positive root}.  
In $M(\lambda)_{\lambda-\beta_1}$, let $r$ be the number of linearly independent maximal vectors.
   \[
\mu \ \text{ is a minimal solution to \eqref{Eqn dot action by any positive root}} \quad \implies \quad r\ = \ \sum_{t\in \mathbb{N}} \dim \mathfrak{g}_{\big(\frac{1}{t}\beta_1\big)}.
  \vspace*{-4mm} \] 
    \begin{equation}\label{maxl vect count inequality}
      \text{In general,\ \ when }\beta_1 \text{ is not isotropic:}\qquad   r\ \ \geq\  \ \sum_{t\in \mathbb{N}} \dim \mathfrak{g}_{\big(\frac{1}{t}\beta_1\big)}.
    \end{equation}
    
\end{prop}
\noindent 
It was showed by analyzing the determinant of Shapovalov form on $M(\lambda)$.
\cite{Kac--Kazhdan} asked for construction of maximal vectors even in solution spaces $\mu=\lambda-\beta_1$ for $\beta_1\in \Delta^+$; it is our $3^{\text{rd}}$ motivation.

The problem of explicitly constructing maximal vectors is generally seen to be hard.
The seminal work \cite[Section 3]{Malikov} (in rank-2 $\widehat{A_1}$-type) constructed them in $M(\lambda)$ in affine cases, by enlarging $M(\lambda)$ to accommodate complex powers of $f_i$'s.
Maximal vectors continue to attract attention till date (\cite{Sale}). 
Hereafter, unless otherwise specified, we work over $\mathfrak{g}\big(A(b,a,c,d)\big)$; so $\mathcal{I}=\{1,2\}$.
\begin{question}\label{Question strict inequality?}
In Proposition \ref{Prop KK}, is there a pair of h.w. $\lambda\in \mathfrak{h}^*$ and non-minimal solution $\mu=\lambda-\beta_1$ to \eqref{Eqn dot action by any positive root} for $\beta_1\in \Delta^+$ (with $r$ as therein), for which strict inequality in \eqref{maxl vect count inequality} arises? That is,
    \[
    r\ \ >\ \ \sum_{t\in \mathbb{N}} \dim \mathfrak{g}_{\big(\frac{1}{t}\beta_1\big)}?
    \]
   \end{question}
   \begin{remark}\label{Remark rank-2 advantages}
1) To prune for positive answers to this in infinite types, it might be reasonable to start in rank-2 and with imaginary $\beta_1\in \Delta^+$.
The points below Lemma 6.1 in \cite{T-P} noted in some rank-2 cases (H) and (R) therein, that the presence of a real or a Heisenberg simple root leads to infinitely many number of solutions.
In contrast, \cite[Lemma 6.2]{T-P} showed for $\lambda\in P^{\pm}$ over $\mathfrak{g}\big(A(b,a,c,d)\big)$, that solutions are always finitely many.
For this simplicity and by point 2) below, we concentrated on  setting \eqref{Eqn defn cartan matrix A(a,b,c,d)} for $\lambda\in P^{\pm}$; for goals P1) and P2) and for resolving Question \ref{Question strict inequality?}. \smallskip\\ 
2) In case \eqref{Eqn defn cartan matrix A(a,b,c,d)}, any sum of $\alpha_1$ and $\alpha_2$  over positive integers is a root, as needed in Question \ref{Question strict inequality?}.\smallskip\\
3) The {\it Casimir operator's} action yields tractable necessity condition of maximal weights $\mu$ in $M(\lambda)$,
\begin{center}
namely, \ the {\it norm equality }\qquad  $(\lambda+\rho \ , \ \lambda+\rho)\  = \  (\mu+\rho\ , \ \mu+\rho )$.
\end{center}
In setting \eqref{Eqn defn cartan matrix A(a,b,c,d)}, it is equivalent to \eqref{Eqn dot action by any positive root}, and simplifies (by the proof of \cite[Lemma 6.1]{T-P}) as:
\begin{equation}  \label{Eqn norm equality condition for general mu= (X,Y) in Case (N)}
    b X^2\ + \ \frac{ad}{c}Y^2\ - \  bM_1 X\ - \ \frac{ad}{c}M_2Y \ +\ 2aXY\ = \ 0.
            \end{equation}
            \end{remark}
            The non-coincidence of solutions of norm-equality and  of \eqref{Eqn dot action by any positive root} was showed for $\lambda=\rho\in P^{\pm}$  over $\mathfrak{g}\big(A(n)\big)$, by \cite[Theorem D, and Observations 7.3 and 7.5]{T-P}; which was not known previously. 

\cite[Proposition 3.7(c)]{T-P} considered for warm-up, $\lambda\in P^{\pm}$ and \eqref{Eqn dot action by any positive root} with solutions  $\mu=\lambda-2\alpha_1-2\alpha_2$ or $\mu= \lambda-n\alpha_1-\alpha_2$ and (or $\mu= \lambda-\alpha_1-n\alpha_2$) for $n\in \mathbb{N}$.
The $r$-values for those choices of $\mu$'s were found to be 2 and 1 respectively; which showed equality in \eqref{maxl vect count inequality}.
Following it, \cite[Lemma 6.6 Figs. (A)--(D)]{T-P} characterized cases where $(2,2)$ is minimal or is non-minimal.
We now work along next natural lines of solutions $\mu$'s which involve each of $\alpha_1$ and $\alpha_2$ at least twice.
\begin{thmx}\label{Theorem maximal vectors}
 Let $\mathfrak{g}=\mathfrak{g}\big(A(b,a,c,d)\big)$ as in Setting \eqref{Eqn defn cartan matrix A(a,b,c,d)}.
   Fix any $\lambda\in P^{\pm}$ with powers $M_i$, where $\lambda(\alpha_i^{\vee})= \frac{A_{ii}}{2}(M_i-1)$ for $i=1,2$ (by Lemma \ref{Lemma rank-1 BKM rep. theory}).
   Fix a solution $\mu= \lambda-X\alpha_1-Y\alpha_2$ to \eqref{Eqn norm equality condition for general mu= (X,Y) in Case (N)}.
    \begin{itemize}
\item[(A1)] If $(X,Y)=(2,n)$ or $(n,2)$, the count of linearly independent maximal vectors in $M(\lambda)_{\mu}$ is 
    \begin{center}
    $\dim \mathfrak{g}_{2\alpha_1+n\alpha_2}+\dim \mathfrak{g}_{\alpha_1+\frac{n}{2}\alpha_2}\ \ = \ \ \bigg\lceil \displaystyle \frac{n+1}{2} \bigg\rceil \ \ = \ \ \dim \big[\mathrm{Hom}_{\mathfrak{g}}\big(M(\lambda- 2\alpha_1-n\alpha_2),\ M(\lambda) \big)\big]$.
    \end{center}
    \item[(A2)] If $(X,Y)=(3,3)$, $M(\lambda)_{\mu}$ has $\dim \mathfrak{g}_{\alpha_1+\alpha_2}$ $+\dim \mathfrak{g}_{3\alpha_1+3\alpha_2}=4$ many linearly independent maximal vectors.
    (See Example \ref{Note non-minimal mu's r-equality}.)
    \end{itemize}
\end{thmx}

Above $\lceil \cdot \rceil: \mathbb{R} \rightarrow \mathbb{Z}$ is the lowest integer function.
We prove Theorem \ref{Theorem maximal vectors}(A1) in Section \ref{Section proof of (A1)}, by performing raising operator actions explicitly.
We do so on images of {\it Poincar\'{e}--Birkoff--Witt monomials} on {\it Lyndon root basis vectors} in $M(\lambda)$ (by Remark \ref{remark Lyndon words}). 
We then use a graphical interpretation of resulting linear equations (for maximality of sums of those ``monomials''), to build their solutions.
The proof of (A2) in Section \ref{Section proof of (A2)} also proceeds in a similar way.
We note that the result in the theorem uniformly holds true over all $\mathfrak{g}\big(A(b,a,c,d)\big)$ for all $a,b,c,d \in \mathbb{N}$.
\begin{example}\label{Note non-minimal mu's r-equality}
When $(M_1, M_2, b,a,c,d)= (1,7,1,1,1,2)$, \eqref{Eqn norm equality condition for general mu= (X,Y) in Case (N)} has 3 interior solutions: $(3,1),\ (3,3),\\ (1,6)$.
 So $(3,3)$ is not a minimal solution therein.
 Nevertheless, (A2) shows equality in \eqref{maxl vect count inequality}.
\end{example}
We go back to original problem of characters of $L(\lambda)$ for $\lambda\in P^{\pm}$. 
Initial cases of Verma's, where they are close to their simples and when those simples have simpler descriptions, are when \eqref{Eqn dot action by any positive root} has no  or exactly one solution in whole $\lambda$-shifted root-cone $\lambda- \mathbb{N}\Pi = \lambda- (\mathbb{N}\{\alpha_1\}\oplus\mathbb{N}\{\alpha_2\}) =\wt M(\lambda)$ $\setminus\{\lambda- n\alpha_1,\ \lambda-m\alpha_2\ |\ n,m\geq 0\}$.
In former cases in them $M(\lambda)$ is simple, and we recall a crucial result for the latter case:
\begin{lemma}[{\cite[Lemma 3.1]{Kac--Kazhdan}}]\label{Lemma KK unique sol}
Fix any $\lambda\in \mathfrak{h}^*$.
Suppose \eqref{Eqn dot action by any positive root} has unique solution $\mu = \lambda- \beta_1$ (for $\beta_1\in \Delta^+$) strictly below $\lambda$.
Every submodule of $M(\lambda)$ is isomorphic to $M(\lambda-\beta_1)\oplus \cdots \oplus M(\lambda-\beta_1)$.
\end{lemma}
\noindent
It might be interesting to find examples of $\lambda$'s in Lemma \ref{Lemma KK unique sol}, which was seemingly unexplored.

Now let $\lambda\in P^{\pm}$ (which is ``integrable''in the sense of completeness by Lemma \ref{Lemma rank-1 BKM rep. theory} and Remark \ref{Remark all 1-dim maxl vect}), with associated powers $M_1 \text{ and } M_2$. Then:\smallskip\\
{\bf Note.} \eqref{Eqn dot action by any positive root} has at least three solutions $\lambda$, $\lambda-M_1\alpha_1$ and $\lambda-M_2\alpha_2$ when $\lambda \in P^{\pm}$.\smallskip\\
When \eqref{Eqn dot action by any positive root} has no solution other than the above three, $L^{\max}(\lambda)$ is simple; its proof proceeds similar to that of simplicity of $M(\lambda')$ when \eqref{Eqn dot action by any positive root} for $\lambda'$ has no non-trivial solution.
On the very next layer, we see the case of \eqref{Eqn dot action by any positive root} (for $\lambda\in P^{\pm}$ presently) with unique solution from interior $\lambda- \mathbb{N}\Pi$; namely, unique solution $(X,Y)$ to \eqref{Eqn norm equality condition for general mu= (X,Y) in Case (N)} with $X,Y>0$.
We now resolve:
\begin{question}\label{Question unique interior sol}
    In spirit of Lemma \ref{Lemma KK unique sol}, for $\lambda\in P^{\pm}$, when \eqref{Eqn dot action by any positive root} has a unique solution in the interior $\lambda- \mathbb{N}\Pi$, what is the structure of $M(\lambda)$?
    And what is the character of $L(\lambda)$?
\end{question}
\begin{remark}\label{Observation minimal unique sol in interior}
    In the above problem, $\mu= \lambda-\beta_1$ is not minimal, iff either $M_1\alpha_1\precneqq \beta_1$ or $M_2\alpha_2\precneqq \beta_1$.
    In $A(2)$ case, \cite[proposition 3.10]{T-P} conducted a number-theoretic characterization of h.w.s $\lambda\in P^{\pm}$ with unique solution in $\mathbb{N}\Pi$ to \eqref{Eqn norm equality condition for general mu= (X,Y) in Case (N)}; also in view of the shortcoming noted below Lemma \ref{Lemma KK unique sol}.
    It found over $A(2)$ that the minimal solutions are $\big(\frac{2}{3}M_1, \frac{2}{3}M_2 \big)$, which occur if and only if $M_1=M_2$.
    And non-minimal solutions therein have the form $(M_1,\ M_2-M_1)$ or $(M_1-M_2, \ M_2)$.
    More generally, $(M_1, n)$ is a non-minimal solution to \eqref{Eqn norm equality condition for general mu= (X,Y) in Case (N)} if and only if $n=M_2-\frac{2c}{d}M_1\in \mathbb{N}$.
\end{remark}
Our next result extends the phenomenon in Lemma \ref{Lemma KK unique sol}, resolving Question \ref{Question unique interior sol}.
It also strengthens the analysis built in \cite[Observation 3.9 and Proposition 3.10, and the discussions surrounding them]{T-P}.
Moreover, we obtain --- in spirit of WKB character type formulas of all $V\twoheadleftarrow M(\rho)$ in \cite[Theorem D]{T-P} --- character formulas for all h.w.m.s $L(\lambda) \twoheadleftarrow V\twoheadleftarrow M(\lambda)$, in the cases we considered.
\begin{thmx}\label{Theorem composition series}
Fix $\mathfrak{g}\big(A(b,a,c,d)\big)$,  $a,..,d\in \mathbb{N}$, and $\lambda\in P^{\pm}$ with powers $M_1,M_2$. 
Suppose $n= M_2-\frac{2c}{d}M_1\in \mathbb{N}$; it holds for $A(2), A(b, a, c ,2c),...$.
And suppose $(M_1,n)$ is the unique solution to \eqref{Eqn norm equality condition for general mu= (X,Y) in Case (N)} in the interior $\mathbb{N}\times \mathbb{N}$.
Then $r= \dim \mathrm{Hom}_{\mathfrak{g}}\big[ M\big(\lambda-M_1\alpha_1- n\alpha_2 \big), \ M(\lambda)\big]\geq 1$. 
Moreover:\begin{itemize}
\item[(a)] $U(\mathfrak{n}^-) f_2^{M_2} M(\lambda)_{\lambda} \underset{\mathfrak{g}\text{-mod}}{\simeq} L(\lambda-M_2\alpha_2)$ and $U(\mathfrak{n}^-)f_2^n f_1^{M_1}M(\lambda)_{\lambda}\underset{\mathfrak{g}\text{-mod}}{\simeq }L(\lambda-M_1\alpha_1-n\alpha_2)$. \\ 
$0 \rightarrow  U(\mathfrak{n}^-) f_2^n f_1^{M_1} M(\lambda)_{\lambda}\rightarrow M(\lambda-M_1\alpha_1)= U(\mathfrak{n}^-) f_1^{M_1} M(\lambda)_{\lambda}\rightarrow  L(\lambda-M_1\alpha_1)  \rightarrow 0$ is an exact sequence of $\mathfrak{g}$-modules. 
\item[(b)] $M(\lambda)$ has a unique submodule $N (\lambda) \underset{ \mathfrak{g}\text{-mod}}{\simeq} L\big(\lambda  -M_1\alpha_1- n\alpha_2\big)^{\oplus \ r-1}
\underset{ \mathfrak{g}\text{-mod}}{\simeq}\ M\big(\lambda  -M_1\alpha_1- n\alpha_2\big)^{\oplus \ r-1}$, \quad with \ \  $L(\lambda)\ = \ \frac{M(\lambda)}{\big\langle f_1^{M_1}M(\lambda)_{\lambda}\big\rangle \oplus \big\langle f_2^{M_2}M(\lambda)_{\lambda}\big\rangle \oplus N(\lambda) }$.
\item[(c)] There are totally $6r$ many non-isomorphic h.w.m.s with top weight $\lambda$.
Furthermore, for an exact sequence $0\rightarrow N_V \rightarrow M(\lambda)\rightarrow V\rightarrow 0$, we have the exactness of\\ $N_V\cap \Big(N(\lambda)\oplus \big\langle f_2^n f_1^{M_1} M(\lambda)_{\lambda}\big\rangle \oplus  \big\langle f_2^{M_2}M(\lambda)_{
\lambda}\big\rangle\Big)\ \xhookrightarrow{} \ N_V\    \twoheadrightarrow \ \frac{N_V \cap \big\langle f_1^{M_1}M(\lambda)_{\lambda}\big\rangle}{N_V\cap \big\langle f_2^n f_1^{M_1} M(\lambda)_{\lambda}\big\rangle}$.
\item[(d)] The WKB type character formula of $L(\lambda)$ is 
\[
\mathrm{char} L(\lambda) \quad = \quad \frac{e^{\lambda}-e^{\lambda-M_1\alpha_2}-e^{\lambda-M_2\alpha_2}-(r-1)e^{\lambda-M_1\alpha_1-n\alpha_2}}{\prod\limits_{\alpha \in \Delta^+}
(1 -e^{-\alpha})^{\dim(\mathfrak{g}_{\alpha})}}.
\]
More generally, the (distinct) numerators of the WKB type character formulas of all the $6r$ many non-isomorphic quotients of $M(\lambda)$ are: 
\[
e^{\lambda}-le^{\lambda-M_1\alpha_2}-ie^{\lambda-M_2\alpha_2}-(j+k)e^{\lambda-M_1\alpha_1-n\alpha_2}\ \ \text{ for }l,i,j\in \{0,1\},\ l\leq j, \  \text{ and }k\in \{0,\ldots, r-1\}.
\]
\end{itemize}
\end{thmx}
\begin{remark}
 $V$ with a given character numerator by Theorem \ref{Theorem composition series}(d), can be traced-back in  obvious way.
Similarly, the expected Jordan--H\"{o}lder series for $V$ can be easily written. 
\end{remark}
\begin{remark}\label{Remark r in unique solution case}
 (1) In each unique minimal solution in Remark \ref{Observation minimal unique sol in interior} (which the cases in Theorem \ref{Theorem maximal vectors} do not cover), we have equality in \eqref{maxl vect count inequality}.\\
  (2) \cite[Proposition 3.10]{T-P} recorded all the (infinitely many) solutions $(M_1=1, M_2-1\in \mathbb{N})$ for $M_2=4,6,7,9,13,15,16,...$, over $A(2)$.
  And in all those cases we have equality in \eqref{maxl vect count inequality} with $r=1$.
  Furthermore, we see the same equality with $r=2$ for minimal $(M_1=2, M_2-2=2)$ in the cases of \cite[Lemma 6.6(A)]{T-P}, by \cite[Proposition 3.7(c)]{T-P} or Theorem \ref{Theorem composition series}.\\
 (3) We know $r$ value for solutions $(2,n)$ by Theorem \ref{Theorem composition series}, but their minimality is prohibited when $n>2$ by \cite[Proposition 3.10(c)]{T-P}.
\end{remark}      
As Proposition \ref{Prop KK}(a) is not applicable in the cases in Theorem \ref{Theorem composition series}, in order to completely know character formulas here, we ask following questions:\begin{question}\label{Question num of maxl vect or socle simples geq 1?}
   (1)~Can we determine $r$ explicitly in Theorem \ref{Theorem composition series}, given that the solution that is strictly before $(M_1,n)$ is (easier and smaller) $(M_1,0)$?
   And under what conditions in those cases, $r$ equals the lower bound in \eqref{maxl vect count inequality}?
   The next question in BGG Category $\mathcal{O}$ could be of independent interest.
      \\
   (2)~In finite type $\dim \big[\mathrm{Hom}_{\mathfrak{g}}\big(L(\nu),\ M(\lambda)\big)]\leq 1$ (\cite[Theorem 4.2]{Hump BGG}).
   Consider a composition series $\cdots \longrightarrow M^i\cdots \longrightarrow M^0=M(\lambda)\twoheadrightarrow L(\lambda)$ for any $\lambda\in \mathfrak{h}^*$.
   Here, can a simple $L(\mu)$ occur at least twice in successive positions, i.e., $\frac{M^{i}}{M^{i+1}}\simeq \frac{M^{i-1}}{M^{i}}\simeq L(\mu)$?
    If that happens, is $\frac{M^{i-1}}{M^{i+1}}\simeq L(\mu)\oplus L(\mu)$?
   Particularly, can the socle of (non-Verma) h.w.m. $V$ be direct sum of two or more copies of $L(\mu)$?
 \end{question}
 \begin{remark}\label{Remark rank 2 settings}
We conclude the present discussions by noting that lower rank settings are prominent in the literature for studying: maximal vectors in Vermas \cite{Malikov}, extending Kazhdan--Lusztig theory to Kac--Moody setting \cite{Wallach--Rocha}, identities using Weyl--Kac characters \cite{Feingold} and word maps in negative settings of BKM LAs \cite[Section 5]{Bracket width}.
Next, our motivations to study solutions to norm equality and Kac--Kazhdan equation, include building the analysis of Weyl semigroups $W_{\mathcal{H}}$ to develop (even conjectural) character formulas for higher order Verma modules introduced and studied in \cite{Teja--Khare, T-P}.
\end{remark}
\begin{remark}\label{remark Lyndon words}
In negative setting \eqref{Eqn defn cartan matrix A(a,b,c,d)} --- or of negative $A$'s with any rank and $A_{i,j}<0$ for all $i,j$ --- the negative part $\mathfrak{n}^-$ is a free Lie algebra (e.g. \cite{Elizabeth}).
The Witt's necklace formula yields their root-multiplicities; and such settings of $A$'s continue to attract works in Lie theory and algebraic combinatorics (e.g. \cite{Shushma}).
Hall words, or Lyndon words on $f_i$'s yield basis for root spaces (\cite{Lyndon}).
Here, Weyl group is trivial, and even non-trivial integrable $L(\lambda)$'s are infinite-dimensional.
\end{remark}

\subsection*{Acknowledgments}
 We thank A. Khare and R. Venkatesh, for encouraging us to  study highest weight modules over BKM LAs.
We immensely thank B. Sury K. Hariram and P. Karmakar for fruitful discussions and for their help in Sage programming.
%We thank A. Prasad and S. Viswanath for their helpful discussions during Sage-days program at IMSc and for suggesting references. 
This work was mostly done when the first author was an NBHM Postdoc (fellowship Ref. No. 0204/9/2024/R\& D-II/2965) at IISc Bangalore, and when third author was an NBHM Postdoc at ISI Bangalore (Ref. No. 0204/16(8)/2022/R\&D-II/11979).
The third author thanks Prof. Munna Naik, and Dr. Kalachand for the support by DST inspire grant by Govt. of India (Reg. no. IFA-23MA191).

\section{Proof of Theorem \ref{Theorem maximal vectors}(A1): Maximal vectors with weight $(2,n)$}\label{Section proof of (A1)}
\noindent
We do computations in the proofs of (\ref{Theorem maximal vectors}1) and (\ref{Theorem maximal vectors}2) uniformly over all $A(b,a,c,d)$ for all $a,b,c,d\in \mathbb{N}$.
\begin{proof}[{\bf \textnormal Proof of Theorem \ref{Theorem maximal vectors}(A1)}]
We fix $\mathfrak{g}(A)$, $\lambda$, $M_1,M_2$ as in the theorem.
We assume $(2,n)$ to satisfy \eqref{Eqn norm equality condition for general mu= (X,Y) in Case (N)}.
We write in steps 1--4 below, the conditions and our algorithm for computing maximal vectors in $M(\lambda)_{\lambda-2\alpha_1-n\alpha_2}$ uniformly for all $n\in \mathbb{N}$.
\smallskip \\
\underline{Step 1}. {\it Basis for ``$(2,n)$-weight space'' in Verma $M(\lambda)$} : 
Let $\ad_{f_2}^{\circ k}:\mathfrak{g}\rightarrow \mathfrak{g}$ be the $k$-fold composition of the {\it adjoint bracket map} $\ad_{f_2}(\cdot)=[f_2, \cdot ]$ w.r.t. $f_2$.
We use the following ordered root-basis for $\mathfrak{n}^-$, where the order below increases from right to left:
\[
 \ \ \ldots \ ;  \big[\underbrace{\ad_{f_2}^{\circ k}(f_1),\ \ad_{f_2}^{\circ s}(f_1)}_{(k,s) \text{ s.t. } k<s \text{ and }k+s\leq n} \big] ;\ \  \ldots ;\ \   \big[f_1,\ [f_2,\ f_1]\big] ;\quad  \Big[ \underbrace{f_2, \ldots, \big[f_2, [f_2}_{n}, f_1] \big] \cdots \Big];\ \  \ldots ; \ \   [f_2, \ f_1];\ \   f_1; \ \   f_2
\]
Above, the subcollection of words  $\big[\ad_{f_2}^{\circ k}(f_1),\ \ad_{f_2}^{\circ s}(f_1)\big]$ from triangular region $\{(k, s)\ |\ k<s \text{ and }k+s\leq n\}$ are ordered among themselves with respect to the (induced) lexicographic/dictionary order on that region in $\mathbb{Z}_{\geq 0}\times \mathbb{Z}_{\geq 0}$. 
Our computations do not involve Lie words other than the above.

We work with following listed $(n+1)\times (n+1)$-many PBW monomial basis vectors (where-in $k+s\leq n$) using the above Lie words, from $M(\lambda-2\alpha_1-n\alpha_2)$.
\begin{equation}\label{Eqn (2,n) PWB monomial basis}
f_1^2f_2^n m_{\lambda},\qquad  \underset{k\geq s\geq 0}{\ad_{f_2}^{\circ k}(f_1)\cdot \ad_{f_2}^{\circ s}(f_1)}f_2^{n-k-s}m_{\lambda},\qquad \Big[\underset{0\leq k<s}{\ad_{f_2}^{\circ k}(f_1),\  \ad_{f_2}^{\circ s}(f_1)} \Big]f_2^{n-k-s}m_{\lambda}. 
\end{equation}
Recall $\big\{\big[\ad_{f_2}^{\circ k}(f_1),\  \ad_{f_2}^{\circ s}(f_1) \big]\ \big|\ 0\leq k<s \big\}$ constitutes the Lyndon root basis for space $\mathfrak{g}_{-2\alpha_1-(k+s)\alpha_2}$.
We assume for some tuple $(d_{k,s})_{\substack{0\leq k,s\leq n \\ k+s\leq n }}\in \mathbb{C}^{\frac{(n+1)(n+2)}{2}}$, that the below vector is maximal:
\[
x\ \ = \ \ \sum\limits_{0\leq  s\leq k\leq n} d_{k,s} \ad_{f_2}^{\circ k}(f_1)\cdot \ad_{f_2}^{\circ s}(f_1)f_2^{n-k-s}m_{\lambda}\ +\  \sum\limits_{0\leq k<s \leq n}d_{k,s}\big[\ad_{f_2}^{\circ k}(f_1),\  \ad_{f_2}^{\circ s}(f_1) \big]f_2^{n-k-s}m_{\lambda}.\]
\medskip\\
\underline{Step 2}. {\it Raising operator actions on basis vectors in $(2,n)$-space } : 
Let $\delta(i,j)$ denote the {\it Dirac delta} function.
Note that $e_1$ takes $(2, n)$ weight-space of $M(\lambda)$ into its $(1, n)$ weight-space.
Similarly, $e_2$ takes $(2, n)$ weight-space of $M(\lambda)$ into its $(2, n-1)$ weight-space.
Under each of these operations, we expand the image of each basis vector from $(2, n)$ space in terms of (ordered) basis elements from the  respective lower target weight-spaces, in the following computations:

\noindent
To determine $d_{k,s}$ for maximality of $x$, we write the actions of $e_i$s on the basis vectors in \eqref{Eqn (2,n) PWB monomial basis}. 
\begin{itemize}
    \item[(1)] $f_1^2f_2^n m_{\lambda}$\ \   $\xlongrightarrow{e_1}$\ \ $\big(-b(M_1-1)+2na+b\big)f_1f_2^nm_{\lambda}$\medskip
    \item[(2)] $\ad_{f_2}^{\circ k}(f_1)\ad_{f_2}^{\circ s}(f_1)f_2^{n-k-s} m_{\lambda}$\ \  ($k>s\geq 0$)  \  \    $\xlongrightarrow{e_1}$\ \ $\delta(s,0)\Big(-\frac{b}{2}(M_1-1)+(n-k)a\Big)\ad_{f_2}^{\circ k}(f_1)f_2^{n-k} m_{\lambda}$\medskip \\
    \hspace*{6.5cm}$-\ \delta(s,1)\big(a\big)\ad_{f_2}^{\circ k}(f_1)f_2^{n-k}m_{\lambda}\ - \ \delta(k,1)(a)f_2f_1f_2^{n-1}m_{\lambda} \medskip \\
    \hspace*{1.5cm}= \ \delta(s,0)\Big(-\frac{b}{2}(M_1-1)+(n-k)a\Big)\underset{k\geq 1}{\ad_{f_2}^{\circ k}(f_1)f_2^{n-k}} m_{\lambda}\ -\ \underset{n>k\geq 2>s=1}{\delta(s,1)(a)\ad_{f_2}^{\circ k}(f_1)f_2^{n-k}}m_{\lambda}$ \medskip \\
    \hspace*{2cm}  $- \ \delta(k,1)(a)[f_2,f_1]f_2^{n-1}m_{\lambda}\ - \ \delta(k,1)(a)f_1f_2^{n}m_{\lambda}$ \medskip
    \item[(3)] $[f_2,f_1][f_2,f_1]f_2^{n-2}m_{\lambda}$\ \ $\xlongrightarrow{e_1}$\ \ $(-2a)[f_2,f_1]f_2^{n-1}m_{\lambda}\ -\ a\big[f_2, [f_2, f_1]\big]f_2^{n-2}m_{\lambda}$  \medskip
 \item[(4)] $\ad_{f_2}^{\circ k}(f_1) \ad_{f_2}^{\circ s}(f_1)f_2^{n-k-s}m_{\lambda}$ \ and\ $\big[\ad_{f_2}^{\circ k}(f_1),\ \ad_{f_2}^{\circ s}(f_1)\big]f_2^{n-k-s}m_{\lambda}$ \ \  $(k,s\geq 2)$ \ \ $\xlongrightarrow{e_1}$ \ \ 0  \medskip
    \item[(5)] $\Big[[f_2,f_1],\ \ad_{f_2}^{\circ s}(f_1)\Big]f_2^{n-s-1}m_{\lambda}$\ \ ($n>s\geq 2$)\ \ $\xlongrightarrow{e_1}$\ \ $ (-a)\ad_{f_2}^{\circ s+1}(f_1)f_2^{n-s-1}m_{\lambda}$ \medskip
      \item[(6)] $\Big[f_1,\ \ad_{f_2}^{\circ s}(f_1)\Big]f_2^{n-s}m_{\lambda}$\ \ ($s\geq 1$)\ \ $\xlongrightarrow{e_1}$\ \ $\Big(\big[s+\delta(s,1)\big]a+b\Big)\ad_{f_2}^{\circ s}(f_1)f_2^{n-s}m_{\lambda}\ +\ 0$
      \bigskip
    \item[($1'$)] $f_1^2f_2^n m_{\lambda}$\ \   $\xlongrightarrow{e_2}$\ \ $n\frac{d}{2} \big(-(M_2-1)+(n-1)\big)f_1^2f_2^{n-1}m_{\lambda}$ \medskip
    \item[($2'$)] $\ad_{f_2}^{\circ k}(f_1) \ad_{f_2}^{\circ s}(f_1)f_2^{n-k-s}m_{\lambda}$\ \  $(k> s\geq 0)$ \ \ $\xlongrightarrow{e_2}$ \ \ $k\Big( c+\frac{k-1}{2}d\Big)\ad_{f_2}^{\circ k-1}(f_1) \ad_{f_2}^{\circ s}(f_1)f_2^{n-k-s}m_{\lambda}$ \medskip \\
    \hspace*{3.5cm} + \   $s\Big( c+\frac{s-1}{2}d\Big)\ad_{f_2}^{\circ k}(f_1) \ad_{f_2}^{\circ s-1}(f_1)f_2^{n-k-s}m_{\lambda}$ \medskip \\
    \hspace*{3.5cm} +\  $\big(n-k-s\big)\frac{d}{2} \Big(-(M_2-1)+ (n-k-s-1)\Big)\ad_{f_2}^{\circ k}(f_1) \ad_{f_2}^{\circ s}(f_1)f_2^{n-k-s-1}m_{\lambda}$ \medskip \medskip 
       \item[($3'$)] $\ad_{f_2}^{\circ s}(f_1) \ad_{f_2}^{\circ s}(f_1)f_2^{n-2s}m_{\lambda}$\ \ ($s\geq 1$) \ \   $\xlongrightarrow{e_2}$ \ \ $s\Big( c+\frac{s-1}{2}d\Big)\Big[\ad_{f_2}^{\circ s-1}(f_1) ,\ \ad_{f_2}^{\circ s}(f_1)\Big]f_2^{n-2s}m_{\lambda}$ \medskip \\
    \hspace*{3.5cm} + \   $2s\Big( c+\frac{s-1}{2}d\Big)\ad_{f_2}^{\circ s}(f_1) \ad_{f_2}^{\circ s-1}(f_1)f_2^{n-2s}m_{\lambda}$ \medskip \\
    \hspace*{3.5cm} +\  $\big(n-2s\big)\frac{d}{2} \big(-(M_2-1)+ (n-2s-1)\big)\ad_{f_2}^{\circ s}(f_1) \ad_{f_2}^{\circ s}(f_1)f_2^{n-2s-1}m_{\lambda}$ \medskip
     \item[($4'$)] $\Big[\ad_{f_2}^{\circ k}(f_1),\ \ad_{f_2}^{\circ s}(f_1)\Big]f_2^{n-k-s}m_{\lambda}$\ \ ($0\leq k\leq s-1$)\ \   $\xlongrightarrow{e_2}$ \ \  $k\Big(c+\frac{k-1}{2}d \Big)\Big[\ad_{f_2}^{\circ k-1}(f_1),\ \ad_{f_2}^{\circ s}(f_1)\Big]f_2^{n-k-s}m_{\lambda}$ \medskip \\
     \hspace*{4.5cm} +\  $\Big(1-\delta(k, s-1)\Big) s\Big(c+\frac{s-1}{2}d \Big)\Big[\ad_{f_2}^{\circ k}(f_1),\ \ad_{f_2}^{\circ s-1}(f_1)\Big]f_2^{n-k-s}m_{\lambda}$ \medskip \\
     \hspace*{4.5cm}+\ $(n-k-s)\frac{d}{2}\Big( -(M_2-1)+(n-k-s-1)\Big)\Big[\ad_{f_2}^{\circ k}(f_1),\ \ad_{f_2}^{\circ s}(f_1)\Big]f_2^{n-k-s-1}m_{\lambda}$
     %\item[($5'$)] $\Big[\ad_{f_2}^{\circ s-1}(f_1),\ \ad_{f_2}^{\circ s}(f_1)\Big]f_2^{n-2s-1}m_{\lambda}$\ \ ($s\geq 1$)\ \   $\xlongrightarrow{e_2}$ \ \  $(s-1)\Big(a+\frac{s-2}{2}b \Big)\Big[\ad_{f_2}^{\circ s-2}(f_1),\ \ad_{f_2}^{\circ s}(f_1)\Big]f_2^{n-k-s}m_{\lambda}$ \ + \ 0 \\ \hspace*{4.5cm}+\ $(n-2s-1)\Big(a+\frac{n-2s-2}{2}b\Big)\Big[\ad_{f_2}^{\circ s-1}(f_1),\ \ad_{f_2}^{\circ s}(f_1)\Big]f_2^{n-k-s-2}m_{\lambda}$
 \end{itemize}\bigskip
We consider the vanishings in $e_1x=0$ of  \ \ \ $f_1f_2^n m_{\lambda},\ \ \ [f_2, f_1]f_2^{n-1}m_{\lambda},\ \ \  \big[f_2,[f_2,f_1]\big]f_2^{n-2}m_{\lambda} \ (\text{for }n>2),\ \ \  \ad_{f_2}^{\circ k}(f_1)f_2^{n-k}m_{\lambda} \ (\text{for }n-1\geq k>2)$,\ \ \ $\ad_{f_2}^{\circ n}(f_1)m_{\lambda}$ (for $n\geq 2$). \smallskip \\ 
Next we consider the vanishings in $e_2x=0$ of $f_1^2f_2^{n-1}m_{\lambda},\ \ \  \ad_{f_2}^{\circ k-1}(f_1)\ad_{f_2}^{\circ s}(f_1)f_2^{n-k-s}m_{\lambda}\ (\text{for }k-s\geq 3),\ \ \  \ad_{f_2}^{\circ s+1}(f_1)\ad_{f_2}^{\circ s}(f_1)f_2^{n-2s-2}m_{\lambda} \ (\text{for }s\geq 0),\ \ \  \ad_{f_2}^{\circ s}(f_1)\ad_{f_2}^{\circ s}(f_1)f_2^{n-2s-1}m_{\lambda}\ (\text{for }s\geq 1)$,\ \ \ 
 $\big[\ad_{f_2}^{\circ s-1}(f_1),\ \ad_{f_2}^{\circ s}(f_1)\big]f_2^{n-2s}m_{\lambda} \ (\text{for }s\geq 1)$,\ \ \  
 $\big[\ad_{f_2}^{\circ k-1}(f_1),\ \ad_{f_2}^{\circ s}(f_1)\big]f_2^{n-k-s}m_{\lambda} \ (\text{for }s-k\geq 1)$.\smallskip \\
 These vanishings lead to the following two sets of (5+6) equations, where-in we define for convenience sake for fixed $M_1,M_2,b,a,c,d\in \mathbb{N}$ above, \quad  $A(k):=  k\Big( c+\frac{k-1}{2}d\Big)$,\quad   $B(k):= \Big(-\frac{b}{2}(M_1-1)+(n-k)a\Big)$,\quad  $C(k):= (n-k)\frac{d}{2}\big(n-k-M_2\big)$.
 \begin{align}
    \big(-b(M_1-1)+b+2an\big)\ d_{0,0}\ - \ a\ d_{1,0}\quad = &\ 0 \label{Eqn length 2 M1}\\
    B(2)\ d_{1,0}-2a \ d_{1,1}\ +\   (2a+b)\ d_{0, 1}\quad = &\ 0\label{Eqn length 4 at k=1}\\
B(2)\ d_{2,0}\ -\ a\  d_{2,1}\ -\ a\ d_{1,1}\ +\   (2a+b)d_{0,2}\quad = &\ 0\label{Eqn length 4 at k=2} \\
B(k)\ d_{k,0}\ -\ a\ d_{k,1}\ -\ a\ d_{1,k-1}\ +\   (ka+b)\ d_{0,k}\quad = &\ 0\label{Eqn length 4 at k} \\ 
 B(n)\ d_{n,0}\ - \ a \ d_{1,n-1}\ + \ (na+b) \ d_{0,n} \quad = & \ 0 \label{Eqn length 3 at n}\\
C(0)\ d_{0,0}\ + \ c d_{1,0}\quad = &\ 0\label{Eqn length 2 M2} \\
 A(k)\ d_{k,s}\ +\ A(s+1)\ d_{k-1,s+1}\ +\ C(k+s-1) \ d_{k-1,s}\quad = &\ 0\label{Eqn s- (k-1) leq -2}\\
 A(s+2)\ d_{s+2,s}  \ +  \  2A(s+1) \ d_{s+1,s+1} 
 \ + \  C(2s+1)\ d_{s+1,s}\quad = &\ 0 \label{Eqn above diag}\\
 A(s+1)\ d_{s+1,s}\ + \ C(2s)\ d_{s,s}\quad = &\ 0\label{Eqn length 2} \\
 A(s)\ d_{s,s} \ 
 + \ A(s+1)\ d_{s-1,s+1}  \ + \ C(2s-1)\ d_{s-1,s}\quad = &\ 0 \label{Eqn below diag}\\
  A(k)\ d_{k,s} \ 
 + \ A(s+1)\ d_{k-1,s+1} \ +  \ C(k+s-1)\ d_{k-1,s}\label{Eqn s-(k-1) geq 2}\quad = &\ 0 
     \end{align}            
     \underline{Step 3}. {\it Cluster(s) of variables occurring in each of the equations above are as follows}:
     We interpret the above equations graphically.
     Namely, ignoring the coefficients (involving arbitrarily fixed $M_1,M_2,b,a,c,d\in \mathbb{N}$ above), we see the below pictures (A)--(G) respectively for the plots of variables (connected) in each equation in the following families:
     (This in particular categorizes similar set of equations among the above.) \ \ \ 
     (A) \  \eqref{Eqn length 2 M2} and \eqref{Eqn length 2} \ \Big(for each $0\leq s\leq \lfloor \frac{n}{2}\rfloor-1$\Big), \  the dropping of \eqref{Eqn length 2 M1} from (A) which will be explained later. \ \ \ (B) \ \eqref{Eqn length 4 at k=1}. \ \ \ (C)~\eqref{Eqn length 4 at k=2} and \eqref{Eqn length 4 at k} \ (for each $2\leq k\leq n-1$). \ (D) \ \eqref{Eqn length 3 at n} \ (for each $n\geq 2$). \ \ \ (E) \ \eqref{Eqn s- (k-1) leq -2} and \eqref{Eqn s-(k-1) geq 2} \ \big(for every $(k-1,s)$ not on the three lines with $| k-s|\leq 1$\big). \ \ \ (F) \ \eqref{Eqn above diag} \ \big(for each $0\leq s\leq \lfloor \frac{n}{2}\rfloor-1$\big). \ \ \ (G) \ \eqref{Eqn below diag} \ \big(for each $0\leq s\leq \lfloor \frac{n}{2}\rfloor$\big).
     \begin{center}
					\begin{figure}[H]
						\begin{subfigure}[t]{0.2\textwidth}
							\begin{tikzpicture}[scale=0.5] 
							\draw[step=1, gray, ultra thin] (-0.2,-0.2) grid (6,6) ;
							\draw [fill=black] (0, 0) circle (5pt) ;
                            	\draw [fill=black] (1, 0) circle (5pt) ;
                            \draw [fill=black] (1, 1) circle (5pt) ;
    \draw [fill=black] (2, 1) circle (5pt) ;
   % \draw [fill=black] (2, 2) circle (5pt) ;     \draw [fill=black] (3, 2) circle (5pt) ;
    \draw [fill=black] (3, 3) circle (5pt) node[anchor=south] {{\color{black}\tiny{($s,s$)}}};
    \draw [fill=black] (4, 3) circle (5pt) ;
   % \draw [fill=black] (3, 4) circle (5pt) ;
    \draw[line width = 0.7mm]  (0,0) -- (1,0);
 \draw[line width = 0.7mm]  (1,1) -- (2,1);
  %\draw[line width = 0.7mm]  (2,2) -- (3,2);
  \draw[line width = 0.7mm]  (3,3) -- (4,3);
  \draw[dashed]  (3,6) -- (6,3);
   \draw [fill=black] (3.3, 5.7)  node[anchor=west] {{\color{black}\tiny{$s+k=n$}}};
    \draw [fill=black] (4.2, 5.25)  node[anchor=west] {{\color{black}\tiny{line}}};
            \draw (1.7, -0.2) node[anchor=west]{{\color{black}{ $\boldsymbol{\cdots}$ }}}; 
             \draw (3.7, -0.2) node[anchor=west]{{\color{black}{ $\boldsymbol{\cdots}$ }}}; 
     \draw (0.4, 4.2) node[anchor=east]{{\color{black}{ $\boldsymbol{\vdots}$ }}}; 
   \draw (0.4, 2.2) node[anchor=east]{{\color{black}{ $\boldsymbol{\vdots}$ }}}; 
  \draw (3.2, 2.2) node[anchor=east]{{\color{black}{ $\boldsymbol{\iddots}$ }}};
  \draw (5.2, 4) node[anchor=east]{{\color{black}{ $\boldsymbol{\iddots}$ }}}; 
	\end{tikzpicture}\caption{}
						\end{subfigure} \quad 
                        \begin{subfigure}[t]{0.2\textwidth} 
\begin{tikzpicture}[scale=0.4] 
							\draw[step=1, gray, ultra thin] (-0.2,-0.2) grid (4.2,4.2) ;
							\draw [fill=black] (1, 0) circle (5pt) ;
    \draw [fill=black] (1, 1) circle (5pt) node[anchor=west] {{\color{black}\tiny{($1,1$)}}} ;
    \draw [fill=black] (0, 1) circle (5pt) ;
                             \draw[dashed]  (4,0) -- (0,4);
     \draw [fill=black] (0.3, 3.7)  node[anchor=west] {{\color{black}\tiny{$s+k=n$ line}}};
    \draw[line width = 0.7mm]  (1,0) -- (1,1) -- (0,1);
                            \end{tikzpicture}
                            \caption{\hspace*{1cm}}
                            \end{subfigure}  \quad 
						\begin{subfigure}[t]{0.2\textwidth} 
\begin{tikzpicture}[scale=0.5] 
\draw[step=1, gray, ultra thin] (-0.2,-0.2) grid (8,8) ;
\draw [fill=black] (0, 2) circle (5pt) ;
\draw [fill=black] (1, 1) circle (5pt) ;
    \draw [fill=black] (0, 3) circle (5pt) ;
    \draw [fill=black] (1, 2) circle (5pt) ;
   % \draw [fill=black] (2, 2) circle (5pt) ;     \draw [fill=black] (3, 2) circle (5pt) ;
%   \draw [fill=black] (0, 5) circle (5pt) ;
\draw [fill=black] (2,0) circle (5pt) ;
\draw [fill=black] (2, 1) circle (5pt) ;
\draw [fill=black] (3,0) circle (5pt) ;
\draw [fill=black] (3,1) circle (5pt) ;
    \draw [fill=black] (0, 5) circle (5pt) node[anchor=west] {{\color{black}\tiny{($0,k$)}}};
    \draw [fill=black] (5,0) circle (5pt) node[anchor=north] {{\color{black}\tiny{($k,0$)}}};
    \draw [fill=black] (5,1) circle (5pt) node[anchor=south] {{\color{black}\tiny{($k,1$)}}};
     \draw [fill=black] (7,0) circle (5pt) node[anchor=north] {{\color{black}\tiny{($n-1,0$)}}};
    \draw [fill=black] (7,1) circle (5pt) node[anchor=south] {{\color{black}\tiny{($n-1,1$)}}};
     \draw [fill=black] (1, 4) circle (5pt) node[anchor=west] {{\color{black}\tiny{($1,k-1$)}}};
     \draw [fill=black] (0, 7) circle (5pt) node[anchor=west] {{\color{black}\tiny{($0,n-1$)}}};
     \draw [fill=black] (1, 6) circle (5pt) node[anchor=west] {{\color{black}\tiny{($1,n-2$)}}};
    \draw[line width = 0.7mm]  (0,2) -- (1,1);
    \draw[line width = 0.7mm]  (0,3) -- (1,2);
    \draw[line width = 0.7mm]  (0,5) -- (1,4);
     \draw[line width = 0.7mm]  (0,7) -- (1,6);
  \draw[dashed]  (8,0) -- (0,8);
   \draw[line width = 0.7mm]  (2,0) -- (2,1);
    \draw[line width = 0.7mm]  (3,0) -- (3,1);
        \draw[line width = 0.7mm]  (5,0) -- (5,1);
     \draw[line width = 0.7mm]  (7,0) -- (7,1);
  \draw[dashed]  (8,0) -- (0,8);
     \draw [fill=black] (4.4, 3.5)  node[anchor=west] {{\color{black}\tiny{$s+k=n$ line}}};
              \draw (3, 0.5) node[anchor=west]{{\color{black}{ $\boldsymbol{\cdots}$ }}};
              \draw (5, 0.5) node[anchor=west]{{\color{black}{ $\boldsymbol{\cdots}$ }}};
     \draw (1.2, 3.8) node[anchor=east]{{\color{black}{ $\boldsymbol{\vdots}$ }}}; 
     \draw (1.2, 6) node[anchor=east]{{\color{black}{ $\boldsymbol{\vdots}$ }}}; 
     	\draw[gray, line width = 0.1mm]  (0,2) -- (-0.3, 2) -- (-0.3, -0.3) -- (2, -0.3) -- (2,0);
        \draw[gray, line width = 0.1mm]  (0,3) -- (-0.6, 3) -- (-0.6, -0.6) -- (3, -0.6) -- (3,0);
         \draw[gray, line width = 0.1mm]  (0,5) -- (-1.2, 5) -- (-1.2, -1.2) -- (5, -1.2) -- (5,0);
          \draw[densely dotted, thick] (-1.2, -1.2) -- (-0.6, -0.6) ;
          \draw[gray, line width = 0.1mm]  (0,7) -- (-1.7, 7) -- (-1.7, -1.7) -- (7, -1.7) -- (7,0);
          \draw[densely dotted, thick] (-1.7, -1.7) -- (-1.2, -1.2) ;
	\end{tikzpicture}\caption{$n>2$}
						\end{subfigure}~\qquad \qquad\ \      
                        \begin{subfigure}[t]{0.22\textwidth} 
\begin{tikzpicture}[scale=0.5] 
							\draw[step=1, gray, ultra thin] (-0.2,-0.2) grid (4.2,4.2) ;
							\draw [fill=black] (4, 0) circle (5pt) ;
    \draw [fill=black] (1, 3) circle (5pt) node[anchor=west] {{\color{black}\tiny{($1,n-1$)}}} ;
    \draw [fill=black] (0, 4) circle (5pt) ;
                             \draw[dashed]  (4,0) -- (0,4);
     \draw [fill=black] (1.3, 2.5)  node[anchor=west] {{\color{black}\tiny{$s+k=n$}}};
     \draw [fill=black] (2, 2)  node[anchor=west] {{\color{black}\tiny{line}}} ;
    \draw[line width = 0.7mm]  (1,3) -- (0,4);
                            \end{tikzpicture}
                            \caption{}
                            \end{subfigure} 
                        \bigskip
                        \begin{subfigure}[t]{0.22\textwidth} 
\begin{tikzpicture}[scale=0.5] 
\draw[step=1, gray, ultra thin] (-0.2,-0.2) grid (9.2,9.2) ;
\draw [fill=black] (2, 0) circle (5pt) ;
\draw [fill=black] (2, 1) circle (5pt) ;
\draw [fill=black] (3, 0) circle (5pt) ;
\draw [fill=black] (9, 0) circle (5pt) ;
\draw [fill=black] (8, 1) circle (5pt) ;
\draw [fill=black] (8, 0) circle (5pt) ;
\draw [fill=black] (5,4) circle (5pt) ;
\draw [fill=black] (5,3) circle (5pt) ;
\draw [fill=black] (6,3) circle (5pt) ;
\draw [fill=black] (0,2) circle (5pt) ;
\draw [fill=black] (1, 2) circle (5pt) ;
\draw [fill=black] (0, 3) circle (5pt) ;
\draw [fill=black] (0,9) circle (5pt) ;
\draw [fill=black] (1,8) circle (5pt) ;
\draw [fill=black] (0, 8) circle (5pt) ;
\draw [fill=black] (4,5) circle (5pt) ;
\draw [fill=black] (3,5) circle (5pt) ;
\draw [fill=black] (3,6) circle (5pt) ;
     \draw (2.5, 0.3) node[anchor=west]{{\color{black}{ $\boldsymbol{\cdots}$ }}};
        \draw (5, 0.3) node[anchor=west]{{\color{black}{ $\boldsymbol{\cdots}$ }}};
         \draw (4.3, 1) node[anchor=east]{{\color{black}{ $\boldsymbol{\iddots}$ }}};
     \draw (4.5, 2.5) node[anchor=east]{{\color{black}{ $\boldsymbol{\iddots}$ }}};
     \draw (5.7, 2) node[anchor=east]{{\color{black}{ $\boldsymbol{\iddots}$ }}};  
         \draw (7.8, 2) node[anchor=east]{{\color{black}{ $\boldsymbol{\ddots}$ }}};
    \draw [fill=black] (4, 1) circle (5pt)  ;
    \draw [fill=black] (5, 1) circle (5pt) node[anchor=west] {{\color{black}\tiny{($k,s$)}}} ;
    \draw [fill=black] (4, 2) circle (5pt) ;
        \draw[dashed]  (9,0) -- (0,9);
     \draw [fill=black] (2.,7.5)  node[anchor=west] {{\color{black}\tiny{$s+k=n$ line}}};
      \draw[dashed]  (0,0) -- (9,9);
     \draw [fill=black] (6,6)  node[anchor=west] {{\color{black}\tiny{$k=s$ line}}};
    \draw[line width = 0.7mm]  (3,0) -- (2,0) -- (2,1);
      \draw[line width = 0.7mm]  (9,0) -- (8,0) -- (8,1);
          \draw[line width = 0.7mm]  (0,9) -- (0,8) -- (1,8);
              \draw[line width = 0.7mm]  (0,3) -- (0,2) -- (1,2);
                 \draw[line width = 0.7mm]  (4,2) -- (4,1) -- (5,1);
       \draw[line width = 0.7mm]  (5,4) -- (5,3) -- (6,3);

        \draw (1,3.8) node[anchor=east]{{\color{black}{ $\boldsymbol{\vdots}$ }}};
      \draw (1,6) node[anchor=east]{{\color{black}{ $\boldsymbol{\vdots}$ }}};
        \draw (0.3,5) node[anchor=west]{{\color{black}{ $\boldsymbol{\vdots}$ }}};
         \draw (1.5,3.3) node[anchor=east]{{\color{black}{ $\boldsymbol{\iddots}$ }}};
     \draw (3, 5) node[anchor=east]{{\color{black}{ $\boldsymbol{\iddots}$ }}};
     \draw (3, 3.5) node[anchor=east]{{\color{black}{ $\boldsymbol{\iddots}$ }}};  
         \draw (3,6.8) node[anchor=east]{{\color{black}{ $\boldsymbol{\ddots}$ }}};
    \draw [fill=black] (1,4) circle (5pt)  ;
    \draw [fill=black] (1,5) circle (5pt)  ;
    \draw [fill=black] (2,4) circle (5pt) node[anchor=west] {{\color{black}\tiny{($k,s$)}}} ;
 %    \draw[line width = 0.7mm]  (3,0) -- (2,0) -- (2,1);
      \draw[line width = 0.7mm]  (0,9) -- (0,8) -- (1,8);
          %\draw[line width = 0.7mm]  (0,9) -- (0,8) -- (1,8);
              \draw[line width = 0.7mm]  (0,3) -- (0,2) -- (1,2);
                 \draw[line width = 0.7mm]  (2,4) -- (1,4) -- (1,5);
       \draw[line width = 0.7mm]  (4,5) -- (3,5) -- (3,6);
       \draw[dotted, thick]  (1.7, 0) -- (1.7,1.2) -- (5,4.4) -- (9, 0.2);
          \draw[dotted, thick]  (0, 1.7) -- (1.2,1.7) -- (4.4,5) -- (0.2, 9);
       \draw[line width = 0.7mm]  (4,5) -- (3,5) -- (3,6);
    \end{tikzpicture}
                \caption{$|k-1 -s|\geq 1$}
         \end{subfigure} ~ \qquad\qquad 
         \begin{subfigure}[t]{0.22\textwidth}
							\begin{tikzpicture}[scale=0.5] 
							\draw[step=1, gray, ultra thin] (-0.2,-0.2) grid (6,6) ;
							\draw [fill=black] (1, 0) circle (5pt) ;
                            	\draw [fill=black] (1, 1) circle (5pt) ;
                            \draw [fill=black] (2,0) circle (5pt) ;
                            \draw [fill=black] (4, 4) circle (5pt)  node[anchor=east] {{\color{black}\tiny{$(\lfloor n/2\rfloor, \lfloor n/2\rfloor)$}}};
                            	\draw [fill=black] (4,3) circle (5pt) ;
                            \draw [fill=black] (5,3) circle (5pt) ;
       \draw[line width = 0.7mm]  (1,1) -- (1,0) -- (2,0);
    \draw[line width = 0.7mm]  (4,4) -- (4,3) -- (5,3);
  %\draw[line width = 0.7mm]  (2,2) -- (3,2);
   \draw[dashed]  (3,6) -- (6,3);
      \draw[dashed]  (0,0) -- (4,4);
         \draw[dashed]  (2,6) -- (6,2);
   \draw [fill=black] (0, 5.7)  node[anchor=west] {{\color{black}\tiny{$s+k=n$}}};
     \draw [fill=black] (1.2, 5.2)  node[anchor=west] {{\color{black}\tiny{(even)}}};
        \draw [fill=black] (3, 5.7)  node[anchor=west] {{\color{black}\tiny{$n=s+k$}}};
     \draw [fill=black] (3.2, 5.2)  node[anchor=west] {{\color{black}\tiny{(odd)}}};
        \draw [fill=black] (2.8, 2.8)  node[anchor=east] {{\color{black}\tiny{line $s=k$}}};
            %\draw (1.7, -0.2) node[anchor=west]{{\color{black}{ $\boldsymbol{\cdots}$ }}}; 
           %  \draw (3.7, -0.2) node[anchor=west]{{\color{black}{ $\boldsymbol{\cdots}$ }}}; 
 %    \draw (0.4, 4.2) node[anchor=east]{{\color{black}{ $\boldsymbol{\vdots}$ }}}; 
  % \draw (0.4, 2.2) node[anchor=east]{{\color{black}{ $\boldsymbol{\vdots}$ }}}; 
  \draw (3.7, 2) node[anchor=east]{{\color{black}{ $\boldsymbol{\iddots}$ }}};
  %\draw (5.2, 4) node[anchor=east]{{\color{black}{ $\boldsymbol{\iddots}$ }}}; 
	\end{tikzpicture}\caption{}
						\end{subfigure} ~
                        \qquad
         \begin{subfigure}[t]{0.22\textwidth}
							\begin{tikzpicture}[scale=0.5] 
							\draw[step=1, gray, ultra thin] (-0.2,-0.2) grid (6,6) ;
							\draw [fill=black] (0,1) circle (5pt) ;
                            	\draw [fill=black] (1, 1) circle (5pt) ;
                            \draw [fill=black] (0,2) circle (5pt) ;
                            \draw [fill=black] (4, 4) circle (5pt)  node[anchor=north] {{\color{black}\tiny{\qquad\qquad $(\lfloor n/2\rfloor, \lfloor n/2\rfloor)$}}};
                            	\draw [fill=black] (3,4) circle (5pt) ;
                            \draw [fill=black] (3,5) circle (5pt) ;
       \draw[line width = 0.7mm]  (1,1) -- (0,1) -- (0,2);
    \draw[line width = 0.7mm]  (4,4) -- (3,4) -- (3,5);
  %\draw[line width = 0.7mm]  (2,2) -- (3,2);
   \draw[dashed]  (3,6) -- (6,3);
      \draw[dashed]  (0,0) -- (4,4);
         \draw[dashed]  (2,6) -- (6,2);
   \draw [fill=black] (0, 5.7)  node[anchor=west] {{\color{black}\tiny{$s+k=n$}}};
     \draw [fill=black] (1.2, 5.2)  node[anchor=west] {{\color{black}\tiny{(even)}}};
        \draw [fill=black] (3, 5.7)  node[anchor=west] {{\color{black}\tiny{$n=s+k$}}};
     \draw [fill=black] (3.2, 5.2)  node[anchor=west] {{\color{black}\tiny{(odd)}}};
        \draw [fill=black] (1.8, 1.8)  node[anchor=west] {{\color{black}\tiny{line $s=k$}}};
            %\draw (1.7, -0.2) node[anchor=west]{{\color{black}{ $\boldsymbol{\cdots}$ }}}; 
           %  \draw (3.7, -0.2) node[anchor=west]{{\color{black}{ $\boldsymbol{\cdots}$ }}}; 
 %    \draw (0.4, 4.2) node[anchor=east]{{\color{black}{ $\boldsymbol{\vdots}$ }}}; 
  % \draw (0.4, 2.2) node[anchor=east]{{\color{black}{ $\boldsymbol{\vdots}$ }}}; 
  \draw (0.8, 3) node[anchor=west]{{\color{black}{ $\boldsymbol{\iddots}$ }}};
  %\draw (5.2, 4) node[anchor=east]{{\color{black}{ $\boldsymbol{\iddots}$ }}}; 
	\end{tikzpicture}\caption{}
						\end{subfigure}
         \end{figure}
                        \end{center}
     We build solution-tuple $\big(d_{k,s}\big)_{\substack{k,s\geq 0  \\ k+s\leq n }}$ for each fixed $n\in \mathbb{N}$, to the above $\frac{(n+1)(n+2)}{2}$ many equations.
Namely, we assume that tuple to satisfy all the conditions, and then use some of the conditions to determine the values of those $d_{k,s}$ for each $(k,s)$.
   We note to observe the count in the penultimate line, that the total number of equations in families \eqref{Eqn length 2 M1}--\eqref{Eqn s-(k-1) geq 2} is: $\underset{\text{Fig (A)}}{\lfloor \frac{n+1}{2}\rfloor} + \underset{\eqref{Eqn length 2 M1}}{1} + \underset{\text{Fig (B) }\&\text{ (D)}}{2} +  \underset{\text{Fig (C)}}{n-2} + \underset{\text{Count of deg. 2 nodes in Figs (E)}-\text{(G)}}{\Big(\frac{(n+1)(n+2)}{2}-\lfloor \frac{n+1}{2}\rfloor-n-1\Big)}$. 
   And the desired tuples have $\frac{(n+1)(n+2)}{2}$ many components, which are symbols/un-knowns $d_{k,s}$.
        
          We begin by observing that \eqref{Eqn length 2 M1} coincides with \eqref{Eqn length 2 M2}, since $(2,n)$ satisfying \eqref{Eqn norm equality condition for general mu= (X,Y) in Case (N)} implies $-b(M_1-1)+b+2an = \frac{-a}{c} \frac{d}{2}n(n-M_2)= \frac{-a}{c}C(0)$.
          Henceforth we drop \eqref{Eqn length 2 M1}.
    So by the counts in the above paragraph, we have at least one (non-zero) maximal vector.
    And the below step shows the number of linearly independent maximal vectors to be at least two.
\medskip\\
\underline{Step 4} {\it \eqref{Eqn length 2 M1}--\eqref{Eqn s-(k-1) geq 2} are indeed $\frac{(n+1)(n+2)}{2}-2$ many equations} : 
For this, we will observe below that we have redundancies between the following set of four equations arising at : 1) $s=1$ in family (A), 2) $k=2$ in family (C), 3) $n=2$ in (D) and 4) the one in (B).
That is, we show that \eqref{Eqn length 4 at k=2} at $k=2$, is already a combination of the others in the previous line. 

     The first case here is when $n>2$.
     We use the hook and edge equations in pictures (F), (G) and (A)  respectively for variables $d_{2,0} , d_{0,2}, d_{2,1}$ -- involving $d_{1,1}$ -- and replace them in \eqref{Eqn length 4 at k=2} by $d_{1,0} , d_{1,1} , d_{0,1}$.
    \big(We could do so as assumed $(d_{k,s})$ to satisfy all the conditions, in particular the conditions in hook and edge equations.\big) 
    Then, we will have the following.
     \[B(2)\Big[\frac{-C(1)}{A(2)}\Big]d_{1,0}  \ +\  \left[ B(2)\frac{-2A(1)}{A(2)}-a- a\frac{-C(2)}{A(2)}- \frac{A(1)}{A(2)} (2a+b)\right]d_{1,1}  \ + \  (2a+b)\Big[\frac{-C(1)}{A(2)}\Big]d_{0,1} \ = \ 0 .\]
     Above the coefficient of $A(2)\times d_{1,1}$ is $ -2c \Big( \frac{-b}{2}(M_1-1)+ (n-2)a\Big)  - a(2c+d) + \frac{ad}{2}(n-2)(n-M_2-2)- c(2a+b) \ = \  bc(M_1-1) - 2(n-\cancel{2})ac - \cancel{4ac} -ad -bc +\frac{ad}{2}n^2-\frac{ad}{2}nM_2 + ad (M_2-2n+2)$.
     Recall from $(2,n)$ satisfying the norm condition \eqref{Eqn norm equality condition for general mu= (X,Y) in Case (N)}, we have 
     \begin{equation}\label{Eqn M_1 into M_2 by norm equality for (2,n)}
    \frac{ad}{2} n(n-M_2)\  =\  bc(M_1-1)-bc -2acn.
     \end{equation}
     Thereby, the coef. of $d_{1,1}$ by the previous two sentences  finally is $\frac{1}{A(2)}\big[ ad n(n-M_2) + ad(M_2-2n+1)\big] =\frac{2}{A(2)} \frac{ad}{2}(n-1)(n-M_2-1) =2\frac{a}{A(2)}C(1)$.
     Thus the equation we started by is $-\frac{C(1)}{A(2)}\times \Big[ B(2)d_{1,0}-2ad_{1,1}+ (2a+b) d_{0,1} \Big] = 0 $, which is: i) either triviality if $C(1)=0$, i.e. if $n=M_2+1$; or ii) equivalent to \eqref{Eqn length 4 at k=1} (which is a valid equation even at $n=2$, in connection to next paragraph).

     Suppose $n=2$. 
     We do not have \eqref{Eqn length 4 at k=2},  instead we have \eqref{Eqn length 3 at n}, and note that $C(2)=0$.
     We do not see $d_{2,1}$-th component in our $(d_{k,s})$. 
    In connection to this, we look back at our above simplification of equation \eqref{Eqn length 4 at k=2} --  occurring on the ultimate line $s+k=n=2$ (for the variables $d_{k,s}$) -- that $d_{2,1}$ therein brought-in only $d_{1,1}$ with coefficient $C(2)$.
    And $C(2)=0$ if $n=2$.
    So after that simplification, \eqref{Eqn length 4 at k=2} is equivalent to \eqref{Eqn length 3 at n}.
     Hence by the proof in case $n>2$, \eqref{Eqn length 3 at n} (at $n=2$) is equivalent to \eqref{Eqn length 4 at k=1}, completing the proof of this step.
\medskip\\
\underline{Step 5}. {\it Algorithm for building  solution tuples, by setting-up free variables} : 
We fix any $n\in \mathbb{N}$, and variables $(d_{k,s})_{\substack{0\leq k,s\leq n \\ s+k=n }}$ as in equations \eqref{Eqn length 4 at k=1}--\eqref{Eqn s-(k-1) geq 2}.
We begin by observing among those unknowns, that the variables $d_{k,k}$ for integers $0\leq k \leq \lfloor \frac{n}{2} \rfloor$, are totally $\Big\lceil \frac{n+1}{2} \Big\rceil$ many.
This number equals the lower bound \big(in $(2,n)$ case\big) in the theorem.
Our strategy in the rest of the proof is as follows.
\begin{note}\label{Note redundancies}
  1) In the below proof-lines, we first arbitrarily fix values for each of $d_{j,j}$'s (one at a time) for $0\leq j\leq \big\lfloor \frac{n}{2}\big\rfloor$.
    We iteratively determine the values of each $d_{k,s}$ by writing it implicitly as a linear function of $\big(d_{0,0},\ldots, d_{\lfloor \frac{n}{2}\rfloor}\big)$, and by using not all but a sub-collection of equations \eqref{Eqn length 4 at k=1}--\eqref{Eqn s-(k-1) geq 2} (we will indicate at every layer of construction below the unused equations).
    So, each choice of $\big(d_{0,0},\ldots, d_{\lfloor \frac{n}{2}\rfloor}\big)$ leads to a possible maximal vector; which certainly satisfies the used equations among \eqref{Eqn length 4 at k=1}--\eqref{Eqn s-(k-1) geq 2}. 
Assuming this works, now for the problem of whether the tuple of $d_{k,s}$'s built in this way, satisfies all \eqref{Eqn length 4 at k=1}--\eqref{Eqn s-(k-1) geq 2}, see the next point.\\
     $2$) Conversely, on every tuple $\big(d_{k,s}'\big)_{0\leq k,s\leq n\\ k+s\leq n}$ that yields a maximal vector, the used set of equations in the previous point do apply.
     As a result, $d_{k,s}'$ must assume the same values as computed in above point using $\big(d_{0,0},\ldots, d_{\lfloor \frac{n}{2}\rfloor}\big)$.
    Thus, number of linearly independent set of maximal vector tuples, is at most $\big\lceil \frac{n+1}{2} \rceil$.
    But this is the lower bound in \eqref{maxl vect count inequality}, and hence we have the desired equality in \eqref{maxl vect count inequality} in the settings of the theorem. So tuple $(d_{k,s})$ we build using $\big(d_{0,0},\ldots, d_{\lfloor \frac{n}{2}\rfloor, \lfloor \frac{n}{2} \rfloor }\big)$ below, will automatically satisfy the unused equations in 1).
\end{note}
We begin by fixing a value for $d_{0,0}$.
The value of $d_{1,0}$ is then uniquely determined by \eqref{Eqn length 2 M2} \big(equivalently by \eqref{Eqn length 2 M1}\big).
Visibly in Figs (A)--(G), no other equation involves solely \big($d_{0,0},\ d_{1,0}$\big). \smallskip\\
\begin{note} 1) In our proof below, we successively determine each $d_{k,s}$ (in terms of diagonals $d_{i,i}$) using one of the following sets of conditions: i)~either a single hook equation in Figs. (B) or (E) or (F) or (G); or ii) single edge equation in Fig. (A); or iii) the equation in Fig. (C) on $(d_{0,k}, d_{1, k-1},d_{k,0}, d_{k,1})$ together with the hook equation in Fig. (E) on $(d_{0,k}, d_{1, k-1}, d_{0, k-1})$ for $2\leq k$; or 
iv) the equation in Fig. (D) on $(d_{0,k}, d_{1, k-1},d_{k,0})$ together with the hook equation in Fig. (E) on $(d_{0,k}, d_{1, k-1}, d_{0, k-1})$. \\
 2) In such combinations of equations, reader might easily see as in the first iteration prior to this note, and similarly at every iteration below, that coefficients $A(k)$ or $A(s)$ or $-a$ or $(ka+b)$ (that we reciprocate) of $d_{k,s}$ are always non-zero.
 And this allows us to determine $d_{k,s}$ uniquely (in view of Note \ref{Note redundancies}); and hereafter, we omit writing non-vanishing of the coefficients in the previous line. 
\end{note}
 Continuing with the proof, we arbitrarily fix a value for the second free variable $d_{1,1}$; note $n\geq 2$.
 The value of $d_{0,1}$ is then uniquely determined -- in terms of afore-found $d_{1,0}$ (and of afore-fixed $d_{0,0}$) --  by \eqref{Eqn length 4 at k=1}.
 And again clearly from Figs. (A)--(G) there are no further equations involving these three variables.
Using the hook-shaped equations in Figs. (F) and (G) involving $d_{1,1}$, we derive the values of $d_{2,0}$ and $d_{0,2}$ in terms of the (fixed) pairs ($d_{1,0}$ and $d_{1,1}$) and resp. ($d_{0,1}$ and $d_{1,1}$).
The equations we used to compute each $d_{k,s}$ for $k+s\leq 2$ in the above iterations are indeed all the equations that arise when $n=2$, and thus they are automatically satisfied by sub-tuple $(d_{k,s})_{0\leq k+s\leq 2}$  we built until now \big(using freely chosen or initial variables $d_{0,0}$ and $d_{1,1}$\big).
Thus the solution space to \eqref{Eqn length 2 M1}--\eqref{Eqn s-(k-1) geq 2} when $n=2$ is 2-dimensional, which is in agreement with $\dim \mathfrak{g}_{2\alpha_1+2\alpha_2}+\dim \mathfrak{g}_{\alpha_1+\alpha_2}= 1+1=2$.

We explain using the below pictures on how to extend the above ideas when $n>2$: (1) We successively compute each $d_{k,s}$ denoted by $\bullet$ in figures at each step below, with using all the known $d_{i,j}$'s denoted by $\circ$ in previous steps and layers in those pictures.
(2) Each iteration begins by fixing a free variable $d_{i, i}$, which we denote by $\tiny{\fbox{i}}$ for $i\in \big\{1, \ \ldots,\  \lfloor \frac{n}{2}\rfloor\big\}$ in those pictures.

We begin by visualizing the part of the proof in $n=2$ case at the beginning of Step 5. 
\begin{center}
    \begin{figure}[H]
    \begin{subfigure}[t]{0.2\textwidth} 
\begin{tikzpicture}[scale=0.5] 
\draw[step=1, gray, ultra thin] (-0.2,-0.2) grid (2.7,2.7) ;
\draw[line width = 0.3mm]  (-0.25,-0.25) -- (-0.25,0.25) -- (0.25,0.25) -- (0.25, -0.25) -- (-0.25, -0.25) ;
\draw (0,0) node[anchor=center]{{\color{black} \bf\tiny 1}} ;
\draw[->, thick] (0,0) -- (0.7,0) ;
\draw[line width = 0.3mm] (0,0) -- (1,0) ;
\draw [fill = black] (1,0) circle (5pt)
 ;    
 \draw[dashed] (2,0)-- (0,2) ; 
 \end{tikzpicture}
   \caption*{\hspace*{-2cm} 1a) For $d_{1,0}$}\label{fig small}
    \end{subfigure}
                            \quad 
\begin{subfigure}[t]{0.2\textwidth} 
\begin{tikzpicture}[scale=0.5] 
\draw[step=1, gray, ultra thin] (-0.2,-0.2) grid (2.7,2.7) ;
\draw[line width = 0.3mm]  (-0.25,-0.25) -- (-0.25,0.25) -- (0.25,0.25) -- (0.25, -0.25) -- (-0.25, -0.25) ;
\draw (0,0) node[anchor=center]{{\color{black} \bf\tiny 1}} ;
\draw[line width = 0.3mm]  (1-0.25,1-0.25) -- (1-0.25,1+0.25) -- (1+0.25,1+0.25) -- (1+0.25, 1-0.25) -- (1-0.25, 1-0.25) ;
\draw (1,1) node[anchor=center]{{\color{black} \bf\tiny 2}} ;
\draw[->, thick] (1,1) -- (0.3,1) ;
\draw[->, thick] (1,0) -- (0.2,0.8) ;
\draw[->, thick] (1,0) -- (0,1) ;
\draw[line width = 0.3mm] (1,0) -- (1,0.7) ;
%\draw[line width = 0.3mm] (0,0) -- (1,0) ;
\draw[line width = 0.4mm] (0,1) -- (1,1) ;
\draw [thick] (1,0) circle (6pt)
 ;
 \draw [fill = black] (0,1) circle (5pt)
 ;
 \draw[dashed] (2,0) -- (0,2) ;
 \end{tikzpicture}
\caption*{\hspace*{-2cm} 1b) For $d_{1,1}$}
\end{subfigure}\quad 
\begin{subfigure}[t]{0.2\textwidth} 
\begin{tikzpicture}[scale=0.5] 
\draw[step=1, gray, ultra thin] (-0.2,-0.2) grid (2.7,2.7) ;
\draw[line width = 0.3mm]  (-0.25,-0.25) -- (-0.25,0.25) -- (0.25,0.25) -- (0.25, -0.25) -- (-0.25, -0.25) ;
\draw (0,0) node[anchor=center]{{\color{black} \bf\tiny 1}} ;
\draw[line width = 0.3mm]  (1-0.25,1-0.25) -- (1-0.25,1+0.25) -- (1.25,1+0.25) -- (1.25, 1-0.25) -- (1-0.25, 1-0.25) ;
\draw (1,1) node[anchor=center]{{\color{black} \bf\tiny 2}} ;
\draw[->, thick] (1,0) -- (1.7,0) ;
\draw[->, thick] (1,1) -- (1.7,0.3) ;
\draw[line width = 0.3mm] (1,1) -- (2,0) ;
\draw[line width = 0.3mm] (1,0) -- (2,0) ;
\draw [thick] (1,0) circle (6pt)
 ;
 \draw [thick] (0,1) circle (6pt)
 ;
\draw[dashed] (2,0) -- (0,2);
\draw[fill = black] (2,0) circle (5pt) ; 
\draw[line width = 0.4mm] (1,1) -- (1,0) ; 
 \end{tikzpicture}
\caption*{\hspace*{-2cm} 1c) For $d_{2,0}$}
 \end{subfigure}\quad 
 \begin{subfigure}[t]{0.2\textwidth} 
\begin{tikzpicture}[scale=0.5] 
\draw[step=1, gray, ultra thin] (-0.2,-0.2) grid (2.7,2.7) ;
\draw[line width = 0.3mm]  (-0.25,-0.25) -- (-0.25,0.25) -- (0.25,0.25) -- (0.25, -0.25) -- (-0.25, -0.25) ;
\draw (0,0) node[anchor=center]{{\color{black} \bf\tiny 1}} ;
\draw[line width = 0.3mm]  (1-0.25,1-0.25) -- (1-0.25,1+0.25) -- (1.25,1+0.25) -- (1.25, 1-0.25) -- (1-0.25, 1-0.25) ;
\draw (1,1) node[anchor=center]{{\color{black} \bf\tiny 2}} ;
\draw[->, thick] (0,1) -- (0,1.7) ;
\draw[->, thick] (1,1) -- (0.3, 1.7) ;
\draw[line width = 0.3mm] (1,1) -- (0,2) ;
\draw[line width = 0.3mm] (0,1) -- (0,2) ;
\draw [thick] (1,0) circle (6pt)
 ;
 \draw [thick] (2,0) circle (6pt)
 ;
 \draw [thick] (0,1) circle (6pt)
 ;
\draw[dashed] (2,0) -- (0,2);
\draw[fill = black] (0,2) circle (5pt) ; 
\draw[line width = 0.4mm] (1,1) -- (0,1) ; 
 \end{tikzpicture}
\caption*{\hspace*{-2cm} 1d) For $d_{0,2}$}
 \end{subfigure}
    \end{figure}
\end{center}
For $n\geq 3$ -- where-in equations \eqref{Eqn length 4 at k=2} and \eqref{Eqn length 4 at k} start appearing, unlike in $n=2$ sub-case --
we compute as shown in the two plots below, the values (using given $d_{0,0}, d_{0,1}$) of the following: (1)~$d_{2,1}$ by \eqref{Eqn length 2} at $s=1$ \big(recall we already verified \eqref{Eqn length 4 at k=2} to hold on $d_{0,2}, d_{1,1}, d_{2,1}, d_{2,0}$ in Step 4\big). (2) And then $d_{3,0}$ by hook equation \eqref{Eqn s- (k-1) leq -2} centered at $(2,0)$.
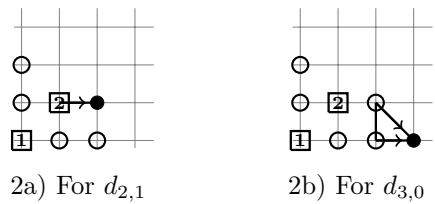
\begin{center}
    \begin{figure}[H]
    \begin{subfigure}[t]{0.2\textwidth} 
\begin{tikzpicture}[scale=0.5] 
\draw[step=1, gray, ultra thin] (-0.2,-0.2) grid (3.5,3.5) ;
\draw[line width = 0.3mm]  (-0.25,-0.25) -- (-0.25,0.25) -- (0.25,0.25) -- (0.25, -0.25) -- (-0.25, -0.25) ;
\draw (0,0) node[anchor=center]{{\color{black} \bf\tiny 1}} ;
\draw[line width = 0.3mm]  (1-0.25,1-0.25) -- (1-0.25,1+0.25) -- (1.25,1+0.25) -- (1.25, 1-0.25) -- (1-0.25, 1-0.25) ;
\draw (1,1) node[anchor=center]{{\color{black} \bf\tiny 2}} ;
%\draw[->, thick] (0,1) -- (0,1.7) ;
\draw[->, thick] (1,1) -- (1.7, 1) ;
\draw[line width = 0.3mm] (1,1) -- (2,1) ;
%\draw[line width = 0.3mm] (0,1) -- (0,2) ;
\draw [thick] (1,0) circle (6pt)
 ;
 \draw [thick] (2,0) circle (6pt)
 ;
 \draw [thick] (0,1) circle (6pt)
 ;
  \draw [thick] (0,2) circle (6pt)
 ;
%\draw[dashed] (2,0) -- (0,2);
\draw[fill = black] (2,1) circle (5pt) ; 
%\draw[line width = 0.4mm] (1,1) -- (0,1) ; 
 \end{tikzpicture}
\caption*{\hspace*{-1.5cm} 2a) For $d_{2,1}$}
 \end{subfigure}\quad 
   \begin{subfigure}[t]{0.2\textwidth} 
\begin{tikzpicture}[scale=0.5] 
\draw[step=1, gray, ultra thin] (-0.2,-0.2) grid (3.5,3.5) ;
\draw[line width = 0.3mm]  (-0.25,-0.25) -- (-0.25,0.25) -- (0.25,0.25) -- (0.25, -0.25) -- (-0.25, -0.25) ;
\draw (0,0) node[anchor=center]{{\color{black} \bf\tiny 1}} ;
\draw[line width = 0.3mm]  (1-0.25,1-0.25) -- (1-0.25,1+0.25) -- (1.25,1+0.25) -- (1.25, 1-0.25) -- (1-0.25, 1-0.25) ;
\draw (1,1) node[anchor=center]{{\color{black} \bf\tiny 2}} ;
%\draw[->, thick] (0,1) -- (0,1.7) ;
\draw[->, thick] (2,0) -- (2.7,0) ;
\draw[->, thick] (2,1) -- (2.7,0.3) ;
\draw[line width = 0.3mm] (2,0) -- (3,0) ;
\draw[line width = 0.3mm] (2,0) -- (2,1) ;
\draw[line width = 0.3mm] (2,1) -- (3,0) ;
%\draw[line width = 0.3mm] (0,1) -- (0,2) ;
\draw [thick] (1,0) circle (6pt)
 ;
 \draw [thick] (2,0) circle (6pt)
 ;
 \draw [thick] (0,1) circle (6pt)
 ;
  \draw [thick] (0,2) circle (6pt)
 ;
  \draw [thick] (2,1) circle (6pt)
 ;
%\draw[dashed] (2,0) -- (0,2);
\draw[fill = black] (3,0) circle (5pt) ; 
%\draw[line width = 0.4mm] (1,1) -- (0,1) ; 
 \end{tikzpicture}
\caption*{\hspace*{-1.5cm} 2b) For $d_{3,0}$}
 \end{subfigure}
    \end{figure}
\end{center}
Note that no equations other than the ones used until now apply on the variables found in the ultimate figure 2b) above.
And in saying so, we remind the reader of Step 4 on \eqref{Eqn length 4 at k=2} when $n\geq 3$, and that the components found up to Fig. 2b) are unaffected by \eqref{Eqn length 4 at k=2}. 

Suppose $n=3$.
We are required to compute left-over variables $d_{(1,2)}$ and $d_{(0,3)}$.
Observe there are exactly two equations that we did not use yet, which are as follows: \eqref{Eqn length 3 at n} at $n=3$ and \eqref{Eqn s- (k-1) leq -2} for $(k-1, s)= (0,2)$.
More precisely,
\[
 - \ a \ d_{1,2}\ + \ (3a+b) \ d_{0,3}\ \ = \ \ - B(3)\ d_{3,0} \qquad \text{and}  \qquad  A(1)\ d_{1,2}\ +\ A(3)\ d_{0,3}\ \ =\ \  - C(2) \ d_{0,2}.
\]
These form consistent system of two equations in variables $d_{1,2}$ and $d_{0,3}$ because $A(1)=c>0$ and $A(3)=3(c+d)>0$.
So $d_{1,2}$ and $d_{0,3}$ can be uniquely found (in terms of their predecessors).

Henceforth, we assume $n\geq 4$, and continue building solution-tuple from Fig. 2b) stage.
Here we cannot compute either of $d_{1,2}$ and $d_{0,3}$, in an analogous way to that in the previous paragraph.
So the need comes to choose the next free variable.
We fix it to be $d_{2,2}$, and then using it we compute uniquely along points (1)--(5) below, all of the surfacing $d_{k,s}$.
The plots 3a)--3g) portray the procedure in those points, with 4a) showing the cumulative picture of 3a)--3g). \\   
(1) We first compute $d_{3,1}$ by hook equation \eqref{Eqn above diag}. \quad (2) And then $d_{4,0}$ by hook equation \eqref{Eqn s- (k-1) leq -2}.\\
(3) Next, we compute $d_{0,3}$ and $d_{1,2}$ for $n\geq 4$ here, following the idea of using two consistent system of equations as in the proof for the case $n=3$ (in the penultimate paragraph above).
Namely, by using \eqref{Eqn length 4 at k} at $k=3$ and \eqref{Eqn s- (k-1) leq -2} at $(k-1,s)= (0,2)$, we solve the following pair of equations.
\[
 - \ a \ d_{1,2}\ + \ (3a+b) \ d_{0,3}\ \ = \ \ - B(3)\ d_{3,0} \ +\ a\ d_{3,1}  \qquad \text{and}  \qquad  A(1)\ d_{1,2}\ +\ A(3)\ d_{0,3}\ \ =\ \  - C(2) \ d_{0,2}.
\]
(4) We apply two hook equations \eqref{Eqn below diag} and then \eqref{Eqn s- (k-1) leq -2} centered at $(k-1,s)\ = (1,2), \ (0,3)$ respectively, and determine $d_{1,3}$ and then $d_{0,4}$.
This completes the construction when $n=4$.\\
(5)  Suppose $n\geq 5$.
We continue the construction as follows.
$d_{3,2}$ is written using an edge equation in Fig. (A). 
Next $d_{4,1}\text{ and then } d_{5,0}$ are given by hook equations in Fig. (E) centered at $d_{3,1},\ d_{4,0}$.
\smallskip\\
The ideas/methodology in (1)--(5) are shown successively in plots 3a)--3e) below:
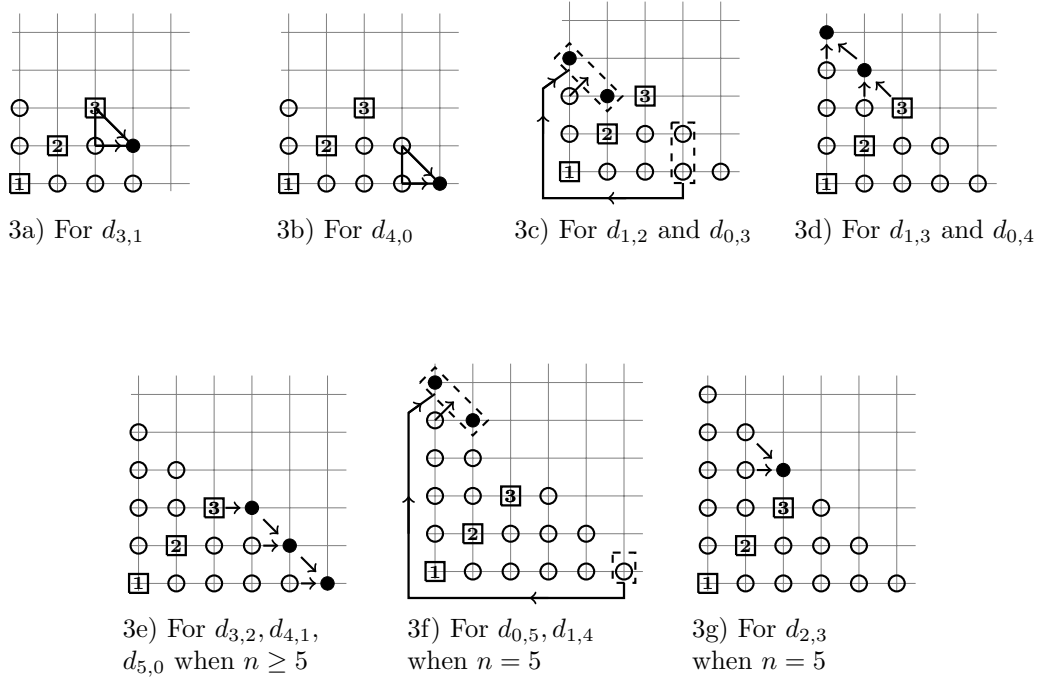
\begin{center}
\begin{figure}[H]
  \begin{subfigure}[t]{0.2\textwidth} 
\begin{tikzpicture}[scale=0.5] 
\draw[step=1, gray, ultra thin] (-0.2,-0.2) grid (4.5,4.5) ;
\draw[line width = 0.3mm]  (-0.25,-0.25) -- (-0.25,0.25) -- (0.25,0.25) -- (0.25, -0.25) -- (-0.25, -0.25) ;
\draw (0,0) node[anchor=center]{{\color{black} \bf\tiny 1}} ;
\draw[line width = 0.3mm]  (1-0.25,1-0.25) -- (1-0.25,1+0.25) -- (1.25,1+0.25) -- (1.25, 1-0.25) -- (1-0.25, 1-0.25) ;
\draw[line width = 0.3mm]  (2-0.25,2-0.25) -- (2-0.25,2+0.25) -- (2+0.25,2+0.25) -- (2+0.25, 2-0.25) -- (2-0.25, 2-0.25) ;
%\draw (1,2) node[anchor=center]{{\color{black} \bf\tiny 3}} ;
\draw (1,1) node[anchor=center]{{\color{black} \bf\tiny 2}} ;
\draw (2,2) node[anchor=center]{{\color{black} \bf\tiny 3}} ;
%\draw[->, thick] (0,1) -- (0,1.7) ;
\draw[->, thick] (2,1) -- (2.7,1) ;
\draw[->, thick] (2,2) -- (2.8, 1.2) ;
\draw[line width = 0.3mm] (2,1) -- (3,1) ;
\draw[line width = 0.3mm] (2,1) -- (2,2) ;
\draw[line width = 0.3mm] (2,2) -- (3,1) ;
%\draw[line width = 0.3mm] (0,1) -- (0,2) ;
\draw [thick] (1,0) circle (6pt)
 ;
 \draw [thick] (2,0) circle (6pt)
 ;
 \draw [thick] (0,1) circle (6pt)
 ;
  \draw [thick] (0,2) circle (6pt)
 ;
  \draw [thick] (2,1) circle (6pt)
 ;
 \draw [thick] (3,0) circle (6pt)
 ;
%\draw[dashed] (2,0) -- (0,2);
\draw[fill = black] (3,1) circle (5pt) ; 
%\draw[line width = 0.4mm] (1,1) -- (0,1) ; 
 \end{tikzpicture}
\caption*{\hspace*{-1.5cm} 3a) For $d_{3,1}$}
 \end{subfigure}\ \
 \begin{subfigure}[t]{0.2\textwidth} 
\begin{tikzpicture}[scale=0.5] 
\draw[step=1, gray, ultra thin] (-0.2,-0.2) grid (4.5,4.5) ;
\draw[line width = 0.3mm]  (-0.25,-0.25) -- (-0.25,0.25) -- (0.25,0.25) -- (0.25, -0.25) -- (-0.25, -0.25) ;
\draw (0,0) node[anchor=center]{{\color{black} \bf\tiny 1}} ;
\draw[line width = 0.3mm]  (1-0.25,1-0.25) -- (1-0.25,1+0.25) -- (1.25,1+0.25) -- (1.25, 1-0.25) -- (1-0.25, 1-0.25) ;
\draw[line width = 0.3mm]  (2-0.25,2-0.25) -- (2-0.25,2+0.25) -- (2+0.25,2+0.25) -- (2+0.25, 2-0.25) -- (2-0.25, 2-0.25) ;
%\draw (1,2) node[anchor=center]{{\color{black} \bf\tiny 3}} ;
\draw (1,1) node[anchor=center]{{\color{black} \bf\tiny 2}} ;
\draw (2,2) node[anchor=center]{{\color{black} \bf\tiny 3}} ;
%\draw[->, thick] (0,1) -- (0,1.7) ;
\draw[->, thick] (3,0) -- (3.7,0) ;
\draw[->, thick] (3,1) -- (3.8, 0.2) ;
\draw[line width = 0.3mm] (3,0) -- (4,0) ;
\draw[line width = 0.3mm] (3,1) -- (4,0) ;
\draw[line width = 0.3mm] (3,0) -- (3,1) ;
%\draw[line width = 0.3mm] (0,1) -- (0,2) ;
\draw [thick] (1,0) circle (6pt)
 ;
 \draw [thick] (2,0) circle (6pt)
 ;
 \draw [thick] (0,1) circle (6pt)
 ;
  \draw [thick] (0,2) circle (6pt)
 ;
  \draw [thick] (2,1) circle (6pt)
 ;
 \draw [thick] (3,0) circle (6pt)
 ;
  \draw [thick] (3,1) circle (6pt)
 ;
%\draw[dashed] (2,0) -- (0,2);
\draw[fill = black] (4,0) circle (5pt) ; 
%\draw[line width = 0.4mm] (1,1) -- (0,1) ; 
 \end{tikzpicture}
\caption*{\hspace*{-1.5cm} 3b) For $d_{4,0}$}
 \end{subfigure}\ \
  \begin{subfigure}[t]{0.2\textwidth} 
\begin{tikzpicture}[scale=0.5] 
\draw[step=1, gray, ultra thin] (-0.2,-0.2) grid (4.5,4.5) ;
\draw[line width = 0.3mm]  (-0.25,-0.25) -- (-0.25,0.25) -- (0.25,0.25) -- (0.25, -0.25) -- (-0.25, -0.25) ;
\draw (0,0) node[anchor=center]{{\color{black} \bf\tiny 1}} ;
\draw[line width = 0.3mm]  (1-0.25,1-0.25) -- (1-0.25,1+0.25) -- (1.25,1+0.25) -- (1.25, 1-0.25) -- (1-0.25, 1-0.25) ;
\draw[line width = 0.3mm]  (2-0.25,2-0.25) -- (2-0.25,2+0.25) -- (2+0.25,2+0.25) -- (2+0.25, 2-0.25) -- (2-0.25, 2-0.25) ;
\draw (1,1) node[anchor=center]{{\color{black} \bf\tiny 2}} ;
\draw (2,2) node[anchor=center]{{\color{black} \bf\tiny 3}} ;
%\draw[->, thick] (0,1) -- (0,1.7) ;
%\draw[->, thick] (3,0.7) -- (3, 0) ;
%\draw[->, thick] (1,2) -- (0.3,2.7) ;
%\draw[line width = 0.3mm] (1,2) -- (0,3) ;
%\draw[line width = 0.3mm] (0,2) -- (1,2) ;
%\draw[line width = 0.3mm] (3,0) -- (3,1) ;
%\draw[line width = 0.3mm] (0,1) -- (0,2) ;
\draw [thick] (1,0) circle (6pt)
 ;
 \draw [thick] (2,0) circle (6pt)
 ;
 \draw [thick] (0,1) circle (6pt)
 ;
  \draw [thick] (0,2) circle (6pt)
 ;
  \draw [thick] (2,1) circle (6pt)
 ;
 \draw [thick] (3,0) circle (6pt)
 ;
 \draw [fill=black] (0,3) circle (5pt)
 ;
  \draw [fill=black] (1,2) circle (5pt)
 ;
  \draw [thick] (4,0) circle (6pt)
 ;
%\draw[dashed] (2,0) -- (0,2);
\draw[thick] (3,1) circle (6pt) ; 
\draw[line width=0.3mm] (0,2.7) -- (-0.7,2.2) -- (-0.7, -0.7) -- (3,-0.7) -- (3,-0.3) ;
%\draw[line width = 0.4mm] (1,1) -- (0,1) ; 
\draw[->, thick] (-0.7,1) -- (-0.7, 1.5) ;
\draw[->, thick] (-0.7,2.2) -- (-0.35, 2.45) ;
\draw[->, thick] (0,2) -- (0.5, 2.5) ;
\draw[->, thick] (1.5,-0.7) -- (1,-0.7) ;
\draw[dashed, thick] (2.7, 1.3) -- (3.3, 1.3) -- (3.3, -0.3) -- (2.7, -0.3) -- (2.7, 1.3) ;
\draw[dashed, thick] (-0.4, 3) -- (0, 3.4) -- (1.4, 2) -- (1, 1.6) -- (-0.4, 3) ;
 \end{tikzpicture}
\caption*{\hspace*{-1cm} 3c) For $d_{1,2}$ and $d_{0,3}$}
 \end{subfigure}~  
\begin{subfigure}[t]{0.19\textwidth} 
\begin{tikzpicture}[scale=0.5] 
\draw[step=1, gray, ultra thin] (-0.2,-0.2) grid (4.5,4.5) ;
\draw[line width = 0.3mm]  (-0.25,-0.25) -- (-0.25,0.25) -- (0.25,0.25) -- (0.25, -0.25) -- (-0.25, -0.25) ;
\draw (0,0) node[anchor=center]{{\color{black} \bf\tiny 1}} ;
\draw[line width = 0.3mm]  (1-0.25,1-0.25) -- (1-0.25,1+0.25) -- (1.25,1+0.25) -- (1.25, 1-0.25) -- (1-0.25, 1-0.25) ;
\draw[line width = 0.3mm]  (2-0.25,2-0.25) -- (2-0.25,2+0.25) -- (2+0.25,2+0.25) -- (2+0.25, 2-0.25) -- (2-0.25, 2-0.25) ;
\draw (2,2) node[anchor=center]{{\color{black} \bf\tiny 3}} ;
\draw (1,1) node[anchor=center]{{\color{black} \bf\tiny 2}} ;
\draw[thick] (1,2) circle (6pt) ;
%\draw[->, thick] (0,1) -- (0,1.7) ;
%\draw[->, thick] (2,1.3) -- (2, 1.7) ;
\draw[->, thick] (1.7,2.3) -- (1.3,2.7 ) ;
\draw[->, thick] (0.8,3.3) -- (0.3, 3.7) ;
\draw[->, thick] (0,3.3) -- (0,3.7) ;
\draw[->, thick] (1,2.3) -- (1,2.7) ;
%\draw[->, thick] (2.8,1.2) -- (2.3, 1.7) ;
%\draw[line width = 0.3mm] (0,3.2) -- (0,4) ;
%\draw[line width = 0.3mm] (1,2) -- (1,3) ;
%\draw[line width = 0.3mm] (2,2) -- (1,3) ;
%\draw[line width = 0.3mm] (3,1) -- (2,2) ;
%\draw[line width = 0.3mm] (3,1) -- (0,4) ;
%\draw[line width = 0.3mm] (3,0) -- (4,0) ;
%\draw[line width = 0.3mm] (3,0) -- (3,1) ;
%\draw[line width = 0.3mm] (4,0) -- (3,1) ;
%\draw[line width = 0.3mm] (0,1) -- (0,2) ;
\draw [thick] (1,0) circle (6pt)
 ;
 \draw [thick] (2,0) circle (6pt)
 ;
 \draw [thick] (0,1) circle (6pt)
 ;
  \draw [thick] (0,2) circle (6pt)
 ;
  \draw [thick] (2,1) circle (6pt)
 ;
 \draw [thick] (3,0) circle (6pt)
 ;
 \draw [thick] (3,1) circle (6pt)
 ;
 \draw [thick] (0,3) circle (6pt)
 ;
 \draw [thick] (4,0) circle (6pt)
 ;
%\draw[dashed] (2,0) -- (0,2);

\draw[fill = black] (1,3) circle (5pt) ; 
; 
\draw[fill = black] (0,4) circle (5pt) ; 
%\draw[line width = 0.4mm] (1,1) -- (0,1) ; 
 \end{tikzpicture}
\caption*{\hspace*{-0.5cm} 3d) For $d_{1,3}$ and $d_{0,4}$}
 \end{subfigure} \end{figure}
 \end{center}
 \begin{center}
\begin{figure}[H]
 \begin{subfigure}[t]{0.19\textwidth} 
\begin{tikzpicture}[scale=0.5] 
\draw[step=1, gray, ultra thin] (-0.2,-0.2) grid (5.5,5.5) ;
\draw[line width = 0.3mm]  (-0.25,-0.25) -- (-0.25,0.25) -- (0.25,0.25) -- (0.25, -0.25) -- (-0.25, -0.25) ;
\draw (0,0) node[anchor=center]{{\color{black} \bf\tiny 1}} ;
\draw[line width = 0.3mm]  (1-0.25,1-0.25) -- (1-0.25,1+0.25) -- (1.25,1+0.25) -- (1.25, 1-0.25) -- (1-0.25, 1-0.25) ;
\draw[line width = 0.3mm]  (2-0.25,2-0.25) -- (2-0.25,2+0.25) -- (2+0.25,2+0.25) -- (2+0.25, 2-0.25) -- (2-0.25, 2-0.25) ;
\draw (2,2) node[anchor=center]{{\color{black} \bf\tiny 3}} ;
\draw (1,1) node[anchor=center]{{\color{black} \bf\tiny 2}} ;
\draw (2,2) node[anchor=center]{{\color{black} \bf\tiny 3}} ;
%\draw[->, thick] (0,1) -- (0,1.7) ;
\draw[->, thick] (2.3,2) -- (2.7, 2) ;
\draw[->, thick] (3.3, 1) -- (3.7,1) ;
\draw[->, thick] (3.3,1.7) -- (3.7, 1.3) ;
\draw[->, thick] (4.3, 0.7) -- (4.7,0.3) ;
\draw[->, thick] (4.3,0) -- (4.7,0) ;
%\draw[->, thick] (3.3,0) -- (3.7, 0) ;
%\draw[line width = 0.3mm] (0,3.2) -- (0,4) ;
%\draw[line width = 0.3mm] (1,2) -- (1,3) ;
%\draw[line width = 0.3mm] (2,2) -- (1,3) ;
%\draw[line width = 0.3mm] (3,1) -- (2,2) ;
%\draw[line width = 0.3mm] (3,1) -- (0,4) ;
%\draw[line width = 0.3mm] (3,0) -- (4,0) ;
%\draw[line width = 0.3mm] (3,0) -- (3,1) ;
%\draw[line width = 0.3mm] (4,0) -- (3,1) ;
%\draw[line width = 0.3mm] (0,1) -- (0,2) ;
\draw [thick] (1,0) circle (6pt)
 ;
 \draw [thick] (2,0) circle (6pt)
 ;
 \draw [thick] (0,1) circle (6pt)
 ;
  \draw [thick] (0,2) circle (6pt)
 ;
  \draw [thick] (2,1) circle (6pt)
 ;
 \draw [thick] (3,0) circle (6pt)
 ;
 \draw [thick] (3,1) circle (6pt)
 ;
 \draw [thick] (0,3) circle (6pt)
 ;
 \draw [thick] (1,2) circle (6pt)
 ;
%\draw[dashed] (2,0) -- (0,2);
\draw [thick] (0,4) circle (6pt) ; 
\draw [thick] (1,3) circle (6pt) ; 
\draw  [thick] (4,0) circle (6pt) ;
\draw[fill=black] (3,2) circle (5pt) ; 
\draw[fill=black] (4,1) circle (5pt) ;
\draw[fill=black] (5,0) circle (5pt) ;
%\draw[line width = 0.4mm] (1,1) -- (0,1) ; 
%\draw (6.5,2) node[anchor=center] {$\cdots$} ; \draw[->, thick] (7.3,2)--(8.5,2) ; \draw (9.5,2) node[anchor=center] {$\cdots$} ;
 \end{tikzpicture}
\caption*{3e) For $d_{3,2}, d_{4,1}, $\\ $d_{5,0}$
when $n\geq 5$}
 \end{subfigure}\hspace*{0.5cm}
  \begin{subfigure}[t]{0.19\textwidth} 
\begin{tikzpicture}[scale=0.5] 
\draw[step=1, gray, ultra thin] (-0.2,-0.2) grid (5.5,5.5) ;
\draw[line width = 0.3mm]  (-0.25,-0.25) -- (-0.25,0.25) -- (0.25,0.25) -- (0.25, -0.25) -- (-0.25, -0.25) ;
\draw (0,0) node[anchor=center]{{\color{black} \bf\tiny 1}} ;
\draw[line width = 0.3mm]  (1-0.25,1-0.25) -- (1-0.25,1+0.25) -- (1.25,1+0.25) -- (1.25, 1-0.25) -- (1-0.25, 1-0.25) ;
\draw[line width = 0.3mm]  (2-0.25,2-0.25) -- (2-0.25,2+0.25) -- (2+0.25,2+0.25) -- (2+0.25, 2-0.25) -- (2-0.25, 2-0.25) ;
\draw (2,2) node[anchor=center]{{\color{black} \bf\tiny 3}} ;
\draw (1,1) node[anchor=center]{{\color{black} \bf\tiny 2}} ;
\draw (2,2) node[anchor=center]{{\color{black} \bf\tiny 3}} ;
%\draw[->, thick] (0,1) -- (0,1.7) ;
%\draw[->, thick] (4.3,0) -- (4.7,0) ;
%\draw[->, thick] (3.3,0) -- (3.7, 0) ;
%\draw[line width = 0.3mm] (0,3.2) -- (0,4) ;
%\draw[line width = 0.3mm] (1,2) -- (1,3) ;
%\draw[line width = 0.3mm] (2,2) -- (1,3) ;
%\draw[line width = 0.3mm] (3,1) -- (2,2) ;
%\draw[line width = 0.3mm] (3,1) -- (0,4) ;
%\draw[line width = 0.3mm] (3,0) -- (4,0) ;
%\draw[line width = 0.3mm] (3,0) -- (3,1) ;
%\draw[line width = 0.3mm] (4,0) -- (3,1) ;
%\draw[line width = 0.3mm] (0,1) -- (0,2) ;
\draw [thick] (1,0) circle (6pt)
 ;
 \draw [thick] (2,0) circle (6pt)
 ;
 \draw [thick] (0,1) circle (6pt)
 ;
  \draw [thick] (0,2) circle (6pt)
 ;
  \draw [thick] (2,1) circle (6pt)
 ;
 \draw [thick] (3,0) circle (6pt)
 ;
 \draw [thick] (3,1) circle (6pt)
 ;
 \draw [thick] (0,3) circle (6pt)
 ;
 \draw [thick] (1,2) circle (6pt)
 ;
%\draw[dashed] (2,0) -- (0,2);
\draw [thick] (0,4) circle (6pt) ; 
\draw [thick] (1,3) circle (6pt) ; 
\draw  [thick] (4,0) circle (6pt) ;
\draw [thick] (3,2) circle (6pt) ; 
\draw [thick] (4,1) circle (6pt) ;
\draw[thick] (5,0) circle (6pt) ;
\draw [fill=black] (0,5) circle (5pt) ;
\draw [fill=black] (1,4) circle (5pt) ; 
\draw[->, thick] (0,4) -- (0.5, 4.5) ;
\draw[line width=0.3mm] (0,4.7) -- (-0.7,4.2) -- (-0.7, -0.7) -- (5,-0.7) -- (5,-0.3) ;
\draw[dashed, thick] (4.7, 0.5) -- (5.3, 0.5) -- (5.3, -0.3) -- (4.7, -0.3) -- (4.7, 0.5) ;
\draw[dashed, thick] (-0.4, 5) -- (0, 5.4) -- (1.4, 4) -- (1, 3.6) -- (-0.4, 5) ;
\draw[->, thick] (-0.7,1) -- (-0.7, 2) ;
\draw[->, thick] (-0.7,4.2) -- (-0.35, 4.45) ;
\draw[->, thick] (3,-0.7) -- (2.5,-0.7) ;
%\draw[line width = 0.4mm] (1,1) -- (0,1) ; 
%\draw (6.5,2) node[anchor=center] {$\cdots$} ; \draw[->, thick] (7.3,2)--(8.5,2) ; \draw (9.5,2) node[anchor=center] {$\cdots$} ;
 \end{tikzpicture}
\caption*{3f) For $d_{0,5}, d_{1,4}$\\ when $n=5$}
 \end{subfigure}\hspace*{0.5cm}
   \begin{subfigure}[t]{0.19\textwidth} 
\begin{tikzpicture}[scale=0.5] 
\draw[step=1, gray, ultra thin] (-0.2,-0.2) grid (5.5,5.5) ;
\draw[line width = 0.3mm]  (-0.25,-0.25) -- (-0.25,0.25) -- (0.25,0.25) -- (0.25, -0.25) -- (-0.25, -0.25) ;
\draw (0,0) node[anchor=center]{{\color{black} \bf\tiny 1}} ;
\draw[line width = 0.3mm]  (1-0.25,1-0.25) -- (1-0.25,1+0.25) -- (1.25,1+0.25) -- (1.25, 1-0.25) -- (1-0.25, 1-0.25) ;
\draw[line width = 0.3mm]  (2-0.25,2-0.25) -- (2-0.25,2+0.25) -- (2+0.25,2+0.25) -- (2+0.25, 2-0.25) -- (2-0.25, 2-0.25) ;
\draw (2,2) node[anchor=center]{{\color{black} \bf\tiny 3}} ;
\draw (1,1) node[anchor=center]{{\color{black} \bf\tiny 2}} ;
\draw (2,2) node[anchor=center]{{\color{black} \bf\tiny 3}} ;
%\draw[->, thick] (0,1) -- (0,1.7) ;
%\draw[->, thick] (4.3,0) -- (4.7,0) ;
%\draw[->, thick] (3.3,0) -- (3.7, 0) ;
%\draw[line width = 0.3mm] (0,3.2) -- (0,4) ;
%\draw[line width = 0.3mm] (1,2) -- (1,3) ;
%\draw[line width = 0.3mm] (2,2) -- (1,3) ;
%\draw[line width = 0.3mm] (3,1) -- (2,2) ;
%\draw[line width = 0.3mm] (3,1) -- (0,4) ;
%\draw[line width = 0.3mm] (3,0) -- (4,0) ;
%\draw[line width = 0.3mm] (3,0) -- (3,1) ;
%\draw[line width = 0.3mm] (4,0) -- (3,1) ;
%\draw[line width = 0.3mm] (0,1) -- (0,2) ;
\draw [thick] (1,0) circle (6pt)
 ;
 \draw [thick] (2,0) circle (6pt)
 ;
 \draw [thick] (0,1) circle (6pt)
 ;
  \draw [thick] (0,2) circle (6pt)
 ;
  \draw [thick] (2,1) circle (6pt)
 ;
 \draw [thick] (3,0) circle (6pt)
 ;
 \draw [thick] (3,1) circle (6pt)
 ;
 \draw [thick] (0,3) circle (6pt)
 ;
 \draw [thick] (1,2) circle (6pt)
 ;
 \draw [thick] (0,5) circle (6pt)
 ;
 \draw [thick] (1,4) circle (6pt)
 ;
%\draw[dashed] (2,0) -- (0,2);
\draw [thick] (0,4) circle (6pt) ; 
\draw [thick] (1,3) circle (6pt) ; 
\draw  [thick] (4,0) circle (6pt) ;
\draw [thick] (3,2) circle (6pt) ; 
\draw [thick] (4,1) circle (6pt) ;
\draw[thick] (5,0) circle (6pt) ;
\draw [fill=black] (2,3) circle (5pt) ;
\draw[->, thick] (1.3,3) -- (1.7,3) ;
\draw[->, thick] (1.3,3.7) -- (1.7,3.3) ;
%\draw[line width = 0.4mm] (1,1) -- (0,1) ; 
%\draw (6.5,2) node[anchor=center] {$\cdots$} ; \draw[->, thick] (7.3,2)--(8.5,2) ; \draw (9.5,2) node[anchor=center] {$\cdots$} ;
 \end{tikzpicture}
\caption*{3g) For $d_{2,3}$\\ when $n=5$}
 \end{subfigure}\end{figure}\end{center}
 (6) Suppose $n=5$.
Simultaneously using \eqref{Eqn length 3 at n} \big(Fig. (D)\big) at $n=5$ and hook equation \eqref{Eqn s- (k-1) leq -2} at $(k-1,s)= (0,4)$, we compute uniquely -- as in point (3) above -- the variables $d_{0,5}$ and $d_{1,4}$.
This is shown in Fig. 3f).
And then $d_{2,3}$ is computed using hook equation centered at $d_{1,3}$, which is shown by Fig. 3g).\\
(7) When $n\geq 6$, the variables found in points (1)--(5) \big(that is, all the plotted variables in 3e)\big) are the only ones that can be found, using $d_{2,2}$ and those in the line-segments along $s+k\in \{0,1,2,3\}$ in previous layers.  
 \begin{center}
 \begin{figure}[H]
\hspace*{-3.5cm} \begin{subfigure}[t]{0.2\textwidth} 
\begin{tikzpicture}[scale=1] 
\draw[fill=gray!30] (-0.3, 2.7) -- (-0.3, -0.4) -- (3.4 , -0.4) -- (3.8, -0.4) -- (2, 1.6) -- ( 0.7, 1.6) -- (-0.3, 2.7) ;
\draw[line width = 0.3mm]  (-0.25,-0.25) -- (-0.25,0.25) -- (0.25,0.25) -- (0.25, -0.25) -- (-0.25, -0.25) ;
\draw[step=1, gray, ultra thin] (-0.2,-0.2) grid (5.5,5.5) ;
\draw (0,0) node[anchor=center]{{\color{black} \bf\tiny 1}} ;
\draw[line width = 0.3mm]  (1-0.25,1-0.25) -- (1-0.25,1+0.25) -- (1.25,1+0.25) -- (1.25, 1-0.25) -- (1-0.25, 1-0.25) ;
\draw[line width = 0.3mm]  (2-0.25,2-0.25) -- (2-0.25,2+0.25) -- (2+0.25,2+0.25) -- (2+0.25, 2-0.25) -- (2-0.25, 2-0.25) ;

\draw (2,2) node[anchor=center]{{\color{black} \bf\tiny 3}} ;
\draw (1,1) node[anchor=center]{{\color{black} \bf\tiny 2}} ;
\draw (2,2) node[anchor=center]{{\color{black} \bf\tiny 3}} ;
\draw [thick] (1,0) circle (6pt)
 ;
 \draw [thick] (2,0) circle (6pt)
 ;
 \draw [thick] (0,1) circle (6pt)
 ;
  \draw [thick] (0,2) circle (6pt)
 ;
  \draw [thick] (2,1) circle (6pt)
 ;
 \draw [thick] (3,0) circle (6pt)
 ;
 \draw [thick] (3,1) circle (6pt)
 ;
 \draw [thick] (0,3) circle (6pt)
 ;
 \draw [thick] (1,2) circle (6pt)
 ;
 % \draw [thick] (0,5) circle (6pt)
 ;
 % \draw [thick] (1,4) circle (6pt)
 ;
%\draw[dashed] (2,0) -- (0,2);
\draw [thick] (0,4) circle (6pt) ; 
\draw [thick] (1,3) circle (6pt) ; 
\draw [thick] (4,0) circle (6pt) ;
\draw [thick] (3,2) circle (6pt) ; 
\draw [thick] (4,1) circle (6pt) ;
\draw [thick] (5,0) circle (6pt) ;
%\draw[line width = 0.4mm] (1,1) -- (0,1) ; 
\draw[line width = 0.3mm]  (0.8, 2.2) -- (0,3) ;
\draw[line width = 0.3mm]  (3,0) -- (3,0.8) ;
\draw[-> , thick]  (3.3,0.7) -- (3.7,0.3) ;
\draw[line width = 0.3mm]   (4,0) -- (0,4) ;
%\draw[-> , thick]  (4,0) -- (3.4, 0.6) ;
\draw[->, thick] (1, 3) -- (0.4, 3.6) ;
\draw[->> , thick]  (3,-0.7) -- (2.5, -0.7) ;
\draw[->> , thick]  (3,0.6) -- (3, 0.3) ;
\draw[->> , thick]  (0.3,2.7) -- (0.7, 2.3) ;
\draw[line width = 0.3mm] (2,2) -- (3,2) -- (5,0) ;
\draw[->>>> , thick]  (2,2) -- (2.7, 2) ;
\draw[-> , thick]  (2+0.3,2-0.3) -- (2+0.6, 2-0.6) ;
\draw [->>>, thick] (1.7, 2.3) -- (1.3, 2.7) ;
\draw [->>, thick] (0.7, 3.3) -- (0.25, 3.75) ;
\draw [thick] (2.5,2.5) node [anchor=center] {\bf {\tiny (IV)}};
\draw [thick] (1.2,2.8) node [anchor=west] {\bf {\tiny (III)}};
\draw [thick] (0.5,3.5) node [anchor=west] {\bf {\tiny (III)}};
\draw[->>>>>, thick] (3,2) -- (3.7, 1.3) ;
\draw [thick] (3.5,1.5) node [anchor=west] {\bf {\tiny (V)}};
\draw[->>>>>, thick] (4,1) -- (4.7, 0.3) ;
\draw [thick] (4.5,0.5) node [anchor=west] {\bf {\tiny (V)}};
\draw[->, thick] (4.4,0.6) -- (4.7, 0.3) ;
\draw[line width=0.3mm] (0,3) -- (-0.7,3) -- (-0.7, -0.7) -- (3,-0.7) -- (3,0) ;
\draw [thick] (2.9,1.5) node[anchor=center] {\bf {\tiny (I) }};
\draw [thick] (3.9,0.5) node[anchor=center] {\bf {\tiny (I) }};
\draw [thick] (3,-0.7) node[anchor=west] {\bf {\tiny (II) }};
%\draw[line width=0.3mm] (5, -0.2) -- (5,-1) -- (-1,-1) -- (-1, 5) -- (0,5) -- (0.85, 4.15);
%\draw[->>>>>>, thick] (5,-0.2) -- (5, -1) ;
%\draw[->>>>>>, thick] (0.2,4.8) -- (0.7, 4.3) ;
%\draw [thick] (0.5,4.5) node[anchor=west] {\bf {\tiny (VI) }};
%\draw [thick] (5,-1) node[anchor=west] {\bf {\tiny (VI) }};
%\draw[->>>>>>>, thick] (1.1,3.9) -- (1.7, 3.3) ;
%\draw[line width=0.3mm] (1,4) -- (2,3);
%\draw [thick] (1.5,3.5) node[anchor=west] {\bf {\tiny (VII) }};
%\draw[thick] (2,3) circle (6pt);
 \end{tikzpicture}
\caption*{\text{4a) Cumulative construction after adding }$d_{2,2}$ \text{ when }$n\geq 6$}
 \end{subfigure}\hspace*{5cm} 
 \begin{subfigure}[t]{0.2\textwidth} 
\begin{tikzpicture}[scale=1] 
\draw[fill=gray!30] (-0.3, 2.7) -- (-0.3, -0.4) -- (3.4 , -0.4) -- (3.8, -0.4) -- (2, 1.6) -- ( 0.7, 1.6) -- (-0.3, 2.7) ;
\draw[line width = 0.3mm]  (-0.25,-0.25) -- (-0.25,0.25) -- (0.25,0.25) -- (0.25, -0.25) -- (-0.25, -0.25) ;
\draw[step=1, gray, ultra thin] (-0.2,-0.2) grid (5.5,5.5) ;
\draw (0,0) node[anchor=center]{{\color{black} \bf\tiny 1}} ;
\draw[line width = 0.3mm]  (1-0.25,1-0.25) -- (1-0.25,1+0.25) -- (1.25,1+0.25) -- (1.25, 1-0.25) -- (1-0.25, 1-0.25) ;
\draw[line width = 0.3mm]  (2-0.25,2-0.25) -- (2-0.25,2+0.25) -- (2+0.25,2+0.25) -- (2+0.25, 2-0.25) -- (2-0.25, 2-0.25) ;

\draw (2,2) node[anchor=center]{{\color{black} \bf\tiny 3}} ;
\draw (1,1) node[anchor=center]{{\color{black} \bf\tiny 2}} ;
\draw (2,2) node[anchor=center]{{\color{black} \bf\tiny 3}} ;
\draw [thick] (1,0) circle (6pt)
 ;
 \draw [thick] (2,0) circle (6pt)
 ;
 \draw [thick] (0,1) circle (6pt)
 ;
  \draw [thick] (0,2) circle (6pt)
 ;
  \draw [thick] (2,1) circle (6pt)
 ;
 \draw [thick] (3,0) circle (6pt)
 ;
 \draw [thick] (3,1) circle (6pt)
 ;
 \draw [thick] (0,3) circle (6pt)
 ;
 \draw [thick] (1,2) circle (6pt)
 ;
  \draw [thick] (0,5) circle (6pt)
 ;
  \draw [thick] (1,4) circle (6pt)
 ;
%\draw[dashed] (2,0) -- (0,2);
\draw [thick] (0,4) circle (6pt) ; 
\draw [thick] (1,3) circle (6pt) ; 
\draw [thick] (4,0) circle (6pt) ;
\draw [thick] (3,2) circle (6pt) ; 
\draw [thick] (4,1) circle (6pt) ;
\draw [thick] (5,0) circle (6pt) ;
%\draw[line width = 0.4mm] (1,1) -- (0,1) ; 
\draw[line width = 0.3mm]  (0.8, 2.2) -- (0,3) ;
\draw[line width = 0.3mm]  (3,0) -- (3,0.8) ;
\draw[-> , thick]  (3.3,0.7) -- (3.7,0.3) ;
\draw[line width = 0.3mm]   (4,0) -- (0,4) ;
%\draw[-> , thick]  (4,0) -- (3.4, 0.6) ;
\draw[->, thick] (1, 3) -- (0.4, 3.6) ;
\draw[->> , thick]  (3,-0.7) -- (2.5, -0.7) ;
\draw[->> , thick]  (3,0.6) -- (3, 0.3) ;
\draw[->> , thick]  (0.3,2.7) -- (0.7, 2.3) ;
\draw[line width = 0.3mm] (2,2) -- (3,2) -- (5,0) ;
\draw[->>>> , thick]  (2,2) -- (2.7, 2) ;
\draw[-> , thick]  (2+0.3,2-0.3) -- (2+0.6, 2-0.6) ;
\draw [->>>, thick] (1.7, 2.3) -- (1.3, 2.7) ;
\draw [->>, thick] (0.7, 3.3) -- (0.25, 3.75) ;
\draw [thick] (2.5,2.5) node [anchor=center] {\bf {\tiny (IV)}};
\draw [thick] (1.2,2.8) node [anchor=west] {\bf {\tiny (III)}};
\draw [thick] (0.5,3.5) node [anchor=west] {\bf {\tiny (III)}};
\draw[->>>>>, thick] (3,2) -- (3.7, 1.3) ;
\draw [thick] (3.5,1.5) node [anchor=west] {\bf {\tiny (V)}};
\draw[->>>>>, thick] (4,1) -- (4.7, 0.3) ;
\draw [thick] (4.5,0.5) node [anchor=west] {\bf {\tiny (V)}};
\draw[->, thick] (4.4,0.6) -- (4.7, 0.3) ;
\draw[line width=0.3mm] (0,3) -- (-0.7,3) -- (-0.7, -0.7) -- (3,-0.7) -- (3,0) ;
\draw [thick] (2.9,1.5) node[anchor=center] {\bf {\tiny (I) }};
\draw [thick] (3.9,0.5) node[anchor=center] {\bf {\tiny (I) }};
\draw [thick] (3,-0.7) node[anchor=west] {\bf {\tiny (II) }};
\draw[line width=0.3mm] (5, -0.2) -- (5,-1) -- (-1,-1) -- (-1, 5) -- (0,5) -- (0.85, 4.15);
\draw[->>>>>>, thick] (5,-0.2) -- (5, -1) ;
\draw[->>>>>>, thick] (0.2,4.8) -- (0.7, 4.3) ;
\draw [thick] (0.5,4.5) node[anchor=west] {\bf {\tiny (VI) }};
\draw [thick] (5,-1) node[anchor=west] {\bf {\tiny (VI) }};
\draw[->>>>>>>, thick] (1.1,3.9) -- (1.7, 3.3) ;
\draw[line width=0.3mm] (1,4) -- (2,3);
\draw [thick] (1.5,3.5) node[anchor=west] {\bf {\tiny (VII) }};
\draw[thick] (2,3) circle (6pt);
 \end{tikzpicture}
\caption*{\text{4b) Cumulative construction after adding }$d_{2,2}$ \text{ when }$n=5$}
 \end{subfigure}
    \end{figure}
\end{center}
In pictures 4a) and 4b), the shaded region contains the full sub-tuple of variables that we computed using free $d_{0,0} \text{ and }d_{1,1}$. 
The further progress we made over  shaded region, after adding (free) $d_{2,2}$ in both the cases $n\geq 6$ and $n =5$, is shown step-wise, by paths (I)--(V) and resp. (I)--(VII), in those figures.
Those directed paths  categorize/isolate used conditions from Figs. (A)--(G) that are similar.

When $n\geq 6$, to continue the construction, we need to choose the next free variable, which we choose to be $d_{3,3}$.
Then, the next set of variables that can be computed after those in picture 4a), with using $d_{3,3}$ and using those in the previous layers, are as follows:
In below pictures 5a) and 5b), the shaded region is the set of all found $d_{k,s}$'s in and up to 3e).
\begin{center}
\begin{figure}[H]
\hspace*{-5cm}  \begin{subfigure}[t]{0.19\textwidth} 
\begin{tikzpicture}[scale=0.7] 
\draw[fill=gray!30] (-0.3, 4.7) -- (-0.3, -0.4) -- (5.4 , -0.4) -- (5.8, -0.4) -- (3, 2.6) -- ( 1.7, 2.6) -- (-0.3, 4.7) ;
\draw[step=1, gray, ultra thin] (-0.2,-0.2) grid (6.5,6.5) ;
\draw[line width = 0.3mm]  (-0.25,-0.25) -- (-0.25,0.25) -- (0.25,0.25) -- (0.25, -0.25) -- (-0.25, -0.25) ;
\draw (0,0) node[anchor=center]{{\color{black} \bf\tiny 1}} ;
\draw[line width = 0.3mm]  (1-0.25,1-0.25) -- (1-0.25,1+0.25) -- (1.25,1+0.25) -- (1.25, 1-0.25) -- (1-0.25, 1-0.25) ;
\draw[line width = 0.3mm]  (2-0.25,2-0.25) -- (2-0.25,2+0.25) -- (2+0.25,2+0.25) -- (2+0.25, 2-0.25) -- (2-0.25, 2-0.25) ;
\draw[line width = 0.3mm]  (3-0.25,3-0.25) -- (3-0.25,3+0.25) -- (3+0.25,3+0.25) -- (3+0.25, 3-0.25) -- (3-0.25, 3-0.25) ;
\draw (2,2) node[anchor=center]{{\color{black} \bf\tiny 3}} ;
\draw (1,1) node[anchor=center]{{\color{black} \bf\tiny 2}} ;
\draw (2,2) node[anchor=center]{{\color{black} \bf\tiny 3}} ;
\draw (3,3) node[anchor=center]{{\color{black} \bf\tiny 4}} ;
%\draw[->, thick] (0,1) -- (0,1.7) ;
%\draw[->, thick] (4.3,0) -- (4.7,0) ;
%\draw[->, thick] (3.3,0) -- (3.7, 0) ;
%\draw[line width = 0.3mm] (0,3.2) -- (0,4) ;
%\draw[line width = 0.3mm] (1,2) -- (1,3) ;
%\draw[line width = 0.3mm] (2,2) -- (1,3) ;
%\draw[line width = 0.3mm] (3,1) -- (2,2) ;
%\draw[line width = 0.3mm] (3,1) -- (0,4) ;
%\draw[line width = 0.3mm] (3,0) -- (4,0) ;
%\draw[line width = 0.3mm] (3,0) -- (3,1) ;
%\draw[line width = 0.3mm] (4,0) -- (3,1) ;
%\draw[line width = 0.3mm] (0,1) -- (0,2) ;
\draw [thick] (1,0) circle (6pt)
 ;
 \draw [thick] (2,0) circle (6pt)
 ;
 \draw [thick] (0,1) circle (6pt)
 ;
  \draw [thick] (0,2) circle (6pt)
 ;
  \draw [thick] (2,1) circle (6pt)
 ;
 \draw [thick] (3,0) circle (6pt)
 ;
 \draw [thick] (3,1) circle (6pt)
 ;
 \draw [thick] (0,3) circle (6pt)
 ;
 \draw [thick] (1,2) circle (6pt)
 ;
 \draw [thick] (0,5) circle (6pt)
 ;
 \draw [thick] (1,4) circle (6pt)
 ;
  \draw [thick] (1,5) circle (6pt)
 ;
  \draw [thick] (5,1) circle (6pt)
 ;
  \draw [thick] (2,4) circle (6pt)
 ;
  \draw [thick] (4,2) circle (6pt)
 ;
  \draw [thick] (0,6) circle (6pt)
 ;
  \draw [thick] (6,0) circle (6pt)
 ;
%\draw[dashed] (2,0) -- (0,2);
\draw [thick] (0,4) circle (6pt) ; 
\draw [thick] (1,3) circle (6pt) ; 
\draw  [thick] (4,0) circle (6pt) ;
\draw [thick] (3,2) circle (6pt) ; 
\draw [thick] (4,1) circle (6pt) ;
\draw[thick] (5,0) circle (6pt) ;
\draw [thick] (2,3) circle (6pt) ;
%\draw[->, thick] (1.3,3) -- (1.7,3) ;
%\draw[->, thick] (1.3,3.7) -- (1.7,3.3) ;
%\draw[line width = 0.4mm] (1,1) -- (0,1) ; 
%\draw (6.5,2) node[anchor=center] {$\cdots$} ; \draw[->, thick] (7.3,2)--(8.5,2) ; \draw (9.5,2) node[anchor=center] {$\cdots$} ;
\draw [line width=0.3mm] (3.3, 2.7) -- (6,0); 
\draw [line width=0.3mm] (2.7, 3.3) -- (0,6); 
\draw[->, thick] (4,2) -- (4.5,1.5) ;
\draw[->, thick] (3,3) -- (3.5,2.5) ;
\draw[->, thick] (5,1) -- (5.7,0.3) ;
\draw[->>>>, thick] (2,4) -- (1.2,4.8) ;
\draw[->>>>, thick] (1,5) -- (0.2,5.8) ;
\draw (4, 2) node[anchor=west] {\bf {\tiny (I)}} ;
\draw (1.5, 4.5) node[anchor=west] {\bf {\tiny (IV)}} ;
\draw[dashed, thick] (4.7, 1.3) -- (5.3, 1.3) -- (5.3, -0.3) -- (4.7, -0.3) -- (4.7, 1.3) ;
\draw[dashed, thick] (-0.4, 5) -- (0, 5.4) -- (1.4, 4) -- (1, 3.6) -- (-0.4, 5) ;
\draw[line width=0.3mm] (0,4.7) -- (-0.7,4.2) -- (-0.7, -0.7) -- (5,-0.7) -- (5,-0.3) ;
\draw[->>, thick] (-0.7,3) -- (-0.7, 3.5) ;
\draw [thick] (5, -0.7) node[anchor=west] {\bf {\tiny (II)}} ;
\draw[->>, thick] (4.5, -0.7) -- (4, -0.7) ;
\draw[->>, thick] (0, 4) -- (0.55, 4.55) ;
\draw [thick] (-0.7, 3) node[anchor=east] {\bf {\tiny (II)}} ;
\draw[->>>, thick] (1, 4) -- (1.7, 3.3) ;
\draw [thick] (1.5, 3.5) node[anchor=west] {\bf {\tiny (III)}} ;
\draw [line width=0.3mm] (1,4) -- (2,3) ; 
 \end{tikzpicture}
\subcaption*{\hspace*{-0.5cm} 5a)~For~$(d_{4,2},d_{5,1}, d_{6,0}),\ (d_{0,5}, d_{1,4})$,  $(d_{2,3}), (d_{2,4}, d_{1,5}, d_{0,6})$ using $d_{3,3}$ if $n=6$}
 \end{subfigure}\hspace*{5cm}
   \begin{subfigure}[t]{0.19\textwidth} 
\begin{tikzpicture}[scale=0.8] 
\draw[fill=gray!30] (-0.3, 4.7) -- (-0.3, -0.4) -- (5.4 , -0.4) -- (5.8, -0.4) -- (3, 2.6) -- ( 1.7, 2.6) -- (-0.3, 4.7) ;
\draw[step=1, gray, ultra thin] (-0.2,-0.2) grid (7.5,6.8) ;
\draw[line width = 0.3mm]  (-0.25,-0.25) -- (-0.25,0.25) -- (0.25,0.25) -- (0.25, -0.25) -- (-0.25, -0.25) ;
\draw (0,0) node[anchor=center]{{\color{black} \bf\tiny 1}} ;
\draw[line width = 0.3mm]  (1-0.25,1-0.25) -- (1-0.25,1+0.25) -- (1.25,1+0.25) -- (1.25, 1-0.25) -- (1-0.25, 1-0.25) ;
\draw[line width = 0.3mm]  (2-0.25,2-0.25) -- (2-0.25,2+0.25) -- (2+0.25,2+0.25) -- (2+0.25, 2-0.25) -- (2-0.25, 2-0.25) ;
\draw[line width = 0.3mm]  (3-0.25,3-0.25) -- (3-0.25,3+0.25) -- (3+0.25,3+0.25) -- (3+0.25, 3-0.25) -- (3-0.25, 3-0.25) ;
\draw (2,2) node[anchor=center]{{\color{black} \bf\tiny 3}} ;
\draw (1,1) node[anchor=center]{{\color{black} \bf\tiny 2}} ;
\draw (2,2) node[anchor=center]{{\color{black} \bf\tiny 3}} ;
\draw (3,3) node[anchor=center]{{\color{black} \bf\tiny 4}} ;
%\draw[->, thick] (0,1) -- (0,1.7) ;
%\draw[->, thick] (4.3,0) -- (4.7,0) ;
%\draw[->, thick] (3.3,0) -- (3.7, 0) ;
%\draw[line width = 0.3mm] (0,3.2) -- (0,4) ;
%\draw[line width = 0.3mm] (1,2) -- (1,3) ;
%\draw[line width = 0.3mm] (2,2) -- (1,3) ;
%\draw[line width = 0.3mm] (3,1) -- (2,2) ;
%\draw[line width = 0.3mm] (3,1) -- (0,4) ;
%\draw[line width = 0.3mm] (3,0) -- (4,0) ;
%\draw[line width = 0.3mm] (3,0) -- (3,1) ;
%\draw[line width = 0.3mm] (4,0) -- (3,1) ;
%\draw[line width = 0.3mm] (0,1) -- (0,2) ;
\draw [thick] (1,0) circle (6pt)
 ;
 \draw [thick] (2,0) circle (6pt)
 ;
 \draw [thick] (0,1) circle (6pt)
 ;
  \draw [thick] (0,2) circle (6pt)
 ;
  \draw [thick] (2,1) circle (6pt)
 ;
 \draw [thick] (3,0) circle (6pt)
 ;
 \draw [thick] (3,1) circle (6pt)
 ;
 \draw [thick] (0,3) circle (6pt)
 ;
 \draw [thick] (1,2) circle (6pt)
 ;
 \draw [thick] (0,5) circle (6pt)
 ;
 \draw [thick] (1,4) circle (6pt)
 ;
  \draw [thick] (1,5) circle (6pt)
 ;
  \draw [thick] (5,1) circle (6pt)
 ;
  \draw [thick] (2,4) circle (6pt)
 ;
  \draw [thick] (4,2) circle (6pt)
 ;
  \draw [thick] (0,6) circle (6pt)
 ;
  \draw [thick] (6,0) circle (6pt)
 ;
  \draw [thick] (4,3) circle (6pt)
 ;
  \draw [thick] (5,2) circle (6pt)
 ;
  \draw [thick] (6,1) circle (6pt)
 ;
  \draw [thick] (7,0) circle (6pt)
 ;
%\draw[dashed] (2,0) -- (0,2);
\draw [thick] (0,4) circle (6pt) ; 
\draw [thick] (1,3) circle (6pt) ; 
\draw  [thick] (4,0) circle (6pt) ;
\draw [thick] (3,2) circle (6pt) ; 
\draw [thick] (4,1) circle (6pt) ;
\draw[thick] (5,0) circle (6pt) ;
\draw [thick] (2,3) circle (6pt) ;
%\draw[->, thick] (1.3,3) -- (1.7,3) ;
%\draw[->, thick] (1.3,3.7) -- (1.7,3.3) ;
%\draw[line width = 0.4mm] (1,1) -- (0,1) ; 
%\draw (6.5,2) node[anchor=center] {$\cdots$} ; \draw[->, thick] (7.3,2)--(8.5,2) ; \draw (9.5,2) node[anchor=center] {$\cdots$} ;
\draw [line width=0.3mm] (3.3, 2.7) -- (6,0); 
\draw [line width=0.3mm] (2.7, 3.3) -- (0,6); 
\draw[->, thick] (4,2) -- (4.5,1.5) ;
\draw[->, thick] (3,3) -- (3.5,2.5) ;
\draw[->, thick] (5,1) -- (5.7,0.3) ;
\draw[->>>>, thick] (2,4) -- (1.2,4.8) ;
\draw[->>>>, thick] (1,5) -- (0.2,5.8) ;
\draw (4, 2) node[anchor=west] {\bf {\tiny (I)}} ;
\draw (1.5, 4.5) node[anchor=west] {\bf {\tiny (IV)}} ;
\draw[dashed, thick] (4.7, 1.3) -- (5.3, 1.3) -- (5.3, -0.3) -- (4.7, -0.3) -- (4.7, 1.3) ;
\draw[dashed, thick] (-0.4, 5) -- (0, 5.4) -- (1.4, 4) -- (1, 3.6) -- (-0.4, 5) ;
\draw[line width=0.3mm] (0,4.7) -- (-0.7,4.2) -- (-0.7, -0.7) -- (5,-0.7) -- (5,-0.3) ;
\draw[->>, thick] (-0.7,3) -- (-0.7, 3.5) ;
\draw [thick] (5, -0.7) node[anchor=west] {\bf {\tiny (II)}} ;
\draw[->>, thick] (4.5, -0.7) -- (4, -0.7) ;
\draw[->>, thick] (0, 4) -- (0.55, 4.55) ;
\draw [thick] (-0.7, 3) node[anchor=east] {\bf {\tiny (II)}} ;
\draw[->>>, thick] (1, 4) -- (1.7, 3.3) ;
\draw [thick] (1.5, 3.5) node[anchor=west] {\bf {\tiny (III)}} ;
\draw [line width=0.3mm] (1,4) -- (2,3) ; 
\draw[->>>>>, thick] (3,3) -- (3.9, 3) ;
\draw (3.5, 3.5) node[anchor=center] {\bf {\tiny (V)}} ;
\draw[->>>>>>, thick] (4,3) -- (4.7,2.3) ;
\draw (5, 2.5) node[anchor=west] {\bf {\tiny (VI)}} ;
\draw[->>>>>>, thick] (4+1,3-1) -- (4.7+1,2.3-1) ;
\draw[->>>>>>, thick] (4+2,3-2) -- (4.7+2,2.3-2) ;
\draw [line width=0.3mm] (4,3) -- (7,0) ;
 \end{tikzpicture}
\subcaption*{\hspace*{-0.7cm} 5b)~For~$(d_{4,2},d_{5,1}, d_{6,0}),\ (d_{0,5}, d_{1,4})$,~$(d_{2,3})$, \\ $(d_{2,4}, d_{1,5}, d_{0,6})$,~$(d_{4,3}, d_{5,2}, d_{6,1}, d_{7,0})$\\ using $d_{3,3}$ if $n\geq 7$}
 \end{subfigure}
 \end{figure}
 \end{center}
These observations will pave way to inductively proceed further for general case of $n$ (which can be even or odd); recall we assumed $n\geq 4$.
Namely, suppose we have determined the values of all the computable $d_{k,s}$ from $(d_{0,0}, \ldots, d_{i,i})$ for some $i< \big\lfloor \frac{n}{2} \big\rfloor-1$.
Let those be denoted by the shaded region in the below polygonal figure, whose shape is motivated by above Figs. 4a)--5b) in cases $n\leq 7$.
\begin{center}
    \begin{figure}[H]
     \begin{subfigure}[t]{0.2\textwidth} 
\begin{tikzpicture}[scale=0.5]
\draw[fill=gray!30] (-0.3, 6+2.7) -- (-0.3, -0.4) -- (6+3.4 , -0.4) -- (6+3.8, -0.4) -- (3+2, 3+1.5) -- (3+0.7, 3+ 1.5) -- (-0.3, 6+2.7) ;
\draw[step=1, gray, ultra thin] (-0.2,-0.2) grid (10,10) ;
%\draw[fill=gray!30] (-0.3, 4+2.7) -- (-0.3, -0.4) -- (4+3.4 , -0.4) -- (4+3.8, -0.4) -- (2+2, 2+1.5) -- (2+0., 2+ 1.5) -- (-0.3, 4+2.7) ;
\draw[line width = 0.3mm]  (-0.25,-0.25) -- (-0.25,0.25) -- (0.25,0.25) -- (0.25, -0.25) -- (-0.25, -0.25) ;
\draw (0,0) node[anchor=center]{{\color{black} \bf\tiny 1}} ;
\draw[line width = 0.3mm]  (1-0.25,1-0.25) -- (1-0.25,1+0.25) -- (1.25,1+0.25) -- (1.25, 1-0.25) -- (1-0.25, 1-0.25) ;
\draw[line width = 0.3mm]  (2-0.25,2-0.25) -- (2-0.25,2+0.25) -- (2+0.25,2+0.25) -- (2+0.25, 2-0.25) -- (2-0.25, 2-0.25) ;
\draw[line width = 0.3mm]  (4-0.25,4-0.25) -- (4-0.25,4+0.25) -- (4+0.25,4+0.25) -- (4+0.25, 4-0.25) -- (4-0.25, 4-0.25) ;
 \draw (3,3.2) node[anchor=center]{{\color{black}{ $\boldsymbol{\iddots}$ }}};
 \draw (4,2.2) node[anchor=west]{{\color{black}{ $\boldsymbol{\iddots}$ }}};
 \draw (0.5,4.5) node[anchor=west]{{\color{black}{ $\boldsymbol{\iddots}$ }}};
\draw (1,1) node[anchor=center]{{\color{black} \bf\tiny 2}} ;
\draw (2,2) node[anchor=center]{{\color{black} \bf\tiny 3}} ;
\draw (4,4) node[anchor=center]{{\color{black} \bf\tiny i}} ;
\draw (5,5.1) node[anchor=center]{{\color{gray} \bf\tiny  i+1}} ;
%\draw[dashed, gray, thick] (4,5) node[anchor=center]{{\color{gray} \bf\tiny }} ;
\draw[line width = 0.3mm, gray]  (5-0.45,5-0.35) -- (5-0.45,5+0.45) -- (5+0.45,5+0.45) -- (5+0.45, 5-0.35) -- (5-0.45, 5-0.3) ;
\draw [thick] (1,0) circle (6pt)
 ;
 \draw [thick] (2,0) circle (6pt)
 ;
 \draw [thick] (0,1) circle (6pt)
 ;
  \draw [thick] (0,2) circle (6pt)
 ;
  \draw [thick] (2,1) circle (6pt)
 ;
 \draw [thick] (3,0) circle (6pt)
 ;
 \draw [thick] (3,1) circle (6pt)
 ;
 \draw [thick] (0,3) circle (6pt)
 ;
 \draw [thick] (1,2) circle (6pt)
 ;
%\draw[dashed] (2,0) -- (0,2);
\draw[thick] (0,4) circle (6pt) ; 
\draw[thick] (1,3) circle (6pt) ; 
\draw[thick]  (4,0) circle (6pt) ;
\draw[thick] (3,2) circle (6pt) ; 
\draw[thick] (4,1) circle (6pt) ;
\draw[thick] (5,0) circle (6pt) ;
\draw[thick] (9,0) circle (6pt) ;
\draw[thick] (8,1) circle (6pt) ;
\draw[thick] (7,2) circle (6pt) ;
\draw[thick] (6,3) circle (6pt) ;
\draw[thick] (5,4) circle (6pt) ;
%\draw[thick] (4,4) circle (6pt) ;
\draw[thick] (3,5) circle (6pt) ;
\draw[thick] (2,6) circle (6pt) ;
\draw[thick] (1,7) circle (6pt) ;
\draw[thick] (0,8) circle (6pt) ;
%\draw[line width = 0.4mm] (1,1) -- (0,1) ; 
%\draw[line width = 0.5mm]  (0.7, 2.3) -- (0,3) ;
%\draw[->, thick]  (0.7, 2.3) -- (0.4,2.6) ;
%\draw[line width = 0.5mm]  (3,0) -- (3,1) ;
%\draw[-> , thick]  (3,0) -- (3,0.6) ;
%\draw[line width = 0.5mm]   (4,0) -- (0,4) ;
%\draw[-> , thick]  (4,0) -- (3.4, 0.6) ;
%\draw[->, thick] (1, 3) -- (0.4, 3.6) ;
%\draw[line width = 0.5mm] (2,2) -- (3,2) -- (5,0) ;
%\draw[-> , thick]  (2,2) -- (2.6, 2) ;
%\draw[->, thick] (3,2) -- (3.4, 1.6) ;
%\draw[->, thick] (4.4,0.6) -- (4.7, 0.3) ;
%\draw[line width=0.5mm] (0,3) -- (-0.7,3) -- (-0.7, -0.7) -- (3,-0.7) -- (3,0) ;
\draw[thick] (6, 0) node[anchor = west]{$\boldsymbol{\cdots}$} ;
\draw[thick] (0,5) node[anchor = south]{$\boldsymbol{\vdots}$} ;
 \end{tikzpicture}
 \end{subfigure}
 \caption*{6a) Indices of all the computable variables using $(d_{0,0}, \ldots, d_{i,i})$  for $i<\big\lfloor \frac{n}{2}\big\rfloor-1$}
    \end{figure}
\end{center}
We add the next free variable at the sketched place $(i+1,i+1)$ in the above figure.
Analogous to in picture 5a) (i.e. using explanations similar to those therein), we can compute the next layers of variables along the paths (I)--(VI); as $i< \big\lfloor \frac{n}{2}\big\rfloor-1$.
This is shown in following picture 6b); and the presently computed ones include $d_{k, \ 2i+1-k}$'s.
This will prove that all the $d_{k,s}$'s in the solution-tuple we started with, can be iteratively computed.
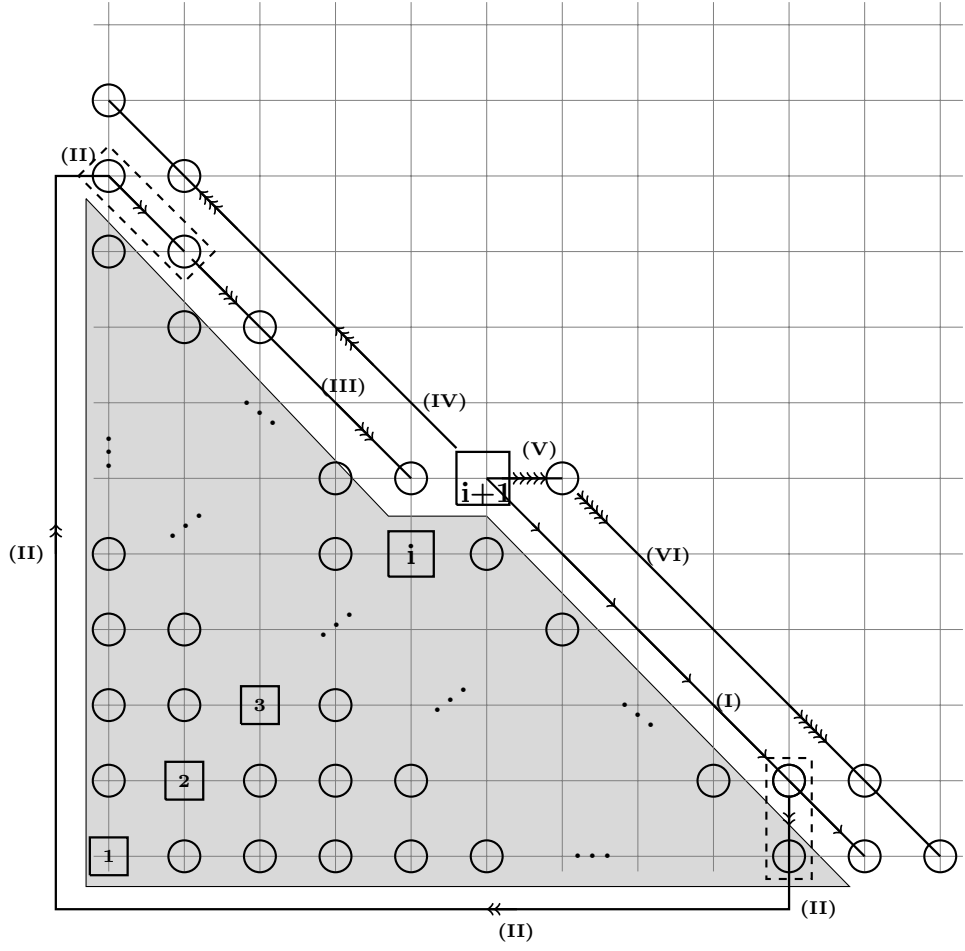
\begin{center}
\begin{figure}[H]
 \hspace*{-10cm} \begin{subfigure}[t]{0.2\textwidth} 
 \begin{tikzpicture}[scale=1]
\draw[fill=gray!30] (-0.3, 6+2.7) -- (-0.3, -0.4) -- (6+3.4 , -0.4) -- (6+3.8, -0.4) -- (3+2, 3+1.5) -- (3+0.7, 3+ 1.5) -- (-0.3, 6+2.7) ;
\draw[step=1, gray, ultra thin] (-0.2,-0.2) grid (11.3,11.3) ;
%\draw[fill=gray!30] (-0.3, 4+2.7) -- (-0.3, -0.4) -- (4+3.4 , -0.4) -- (4+3.8, -0.4) -- (2+2, 2+1.5) -- (2+0., 2+ 1.5) -- (-0.3, 4+2.7) ;
\draw[line width = 0.3mm]  (-0.25,-0.25) -- (-0.25,0.25) -- (0.25,0.25) -- (0.25, -0.25) -- (-0.25, -0.25) ;
\draw (0,0) node[anchor=center]{{\color{black} \bf\tiny 1}} ;
\draw[line width = 0.3mm]  (1-0.25,1-0.25) -- (1-0.25,1+0.25) -- (1.25,1+0.25) -- (1.25, 1-0.25) -- (1-0.25, 1-0.25) ;
\draw[line width = 0.3mm]  (2-0.25,2-0.25) -- (2-0.25,2+0.25) -- (2+0.25,2+0.25) -- (2+0.25, 2-0.25) -- (2-0.25, 2-0.25) ;
\draw[line width = 0.3mm]  (4+1-0.4,4+1-0.35) -- (4+1-0.4,4+1+0.35) -- (4+1+0.3,4+1+0.35) -- (4+1+0.3, 4+1-0.35) -- (4+1-0.4, 4+1-0.35) ;
\draw[line width = 0.3mm]  (4-0.3,5-1-0.3) -- (4-0.3,5-1+0.3) -- (4+0.3,5-1+0.3) -- (4+0.3, 5-1-0.3) -- (4-0.3, 5-1-0.3) ;
 \draw (3,3.2) node[anchor=center]{{\color{black}{ $\boldsymbol{\iddots}$ }}};
 \draw (4,2.2) node[anchor=west]{{\color{black}{ $\boldsymbol{\iddots}$ }}};
 \draw (0.5,4.5) node[anchor=west]{{\color{black}{ $\boldsymbol{\iddots}$ }}};
 \draw (7,2) node[anchor= center]{{\color{black}{ $\boldsymbol{\ddots}$ }}};
 \draw (2,6) node[anchor= center]{{\color{black}{ $\boldsymbol{\ddots}$ }}};
%\draw (1,2) node[anchor=center]{{\color{black} \bf\tiny 3}} ;
\draw (1,1) node[anchor=center]{{\color{black} \bf\tiny 2}} ;
\draw (2,2) node[anchor=center]{{\color{black} \bf\tiny 3}} ;
\draw (4,4) node[anchor=center]{{\color{black} \bf i}} ;
\draw (5,4.8) node[anchor=center]{{\color{black} \bf i+1}} ;
\draw [thick] (1,0) circle (6pt)
 ;
 \draw [thick] (2,0) circle (6pt)
 ;
 \draw [thick] (0,1) circle (6pt)
 ;
  \draw [thick] (0,2) circle (6pt)
 ;
  \draw [thick] (2,1) circle (6pt)
 ;
 \draw [thick] (3,0) circle (6pt)
 ;
 \draw [thick] (3,1) circle (6pt)
 ;
 \draw [thick] (0,3) circle (6pt)
 ;
 \draw [thick] (1,2) circle (6pt)
 ;
%\draw[dashed] (2,0) -- (0,2);
\draw[thick] (0,4) circle (6pt) ; 
\draw[thick] (1,3) circle (6pt) ; 
\draw[thick]  (4,0) circle (6pt) ;
\draw[thick] (3,2) circle (6pt) ; 
\draw[thick] (4,1) circle (6pt) ;
\draw[thick] (5,0) circle (6pt) ;
\draw[thick] (9,0) circle (6pt) ;
\draw[thick] (8,1) circle (6pt) ;
%\draw[thick] (7,2) circle (6pt) ;
\draw[thick] (6,3) circle (6pt) ;
\draw[thick] (5,4) circle (6pt) ;
\draw[thick] (3,4) circle (6pt) ;
\draw[thick] (3,5) circle (6pt) ;
%\draw[thick] (2,6) circle (6pt) ;
\draw[thick] (1,7) circle (6pt) ;
\draw[thick] (0,8) circle (6pt) ;
\draw[thick] (9,1) circle (6pt) ;
\draw[thick] (0,9) circle (6pt) ;
\draw[thick] (1,8) circle (6pt) ;
\draw[thick] (2,7) circle (6pt) ;
\draw[thick] (4,5) circle (6pt) ;
\draw[thick] (6,5) circle (6pt) ;
\draw[thick] (10,0) circle (6pt) ;
\draw[thick] (11,0) circle (6pt) ;
\draw[thick] (10,1) circle (6pt) ;
\draw[thick] (9,1) circle (6pt) ;
\draw[thick] (0,10) circle (6pt) ;
\draw[thick] (1,9) circle (6pt) ;
%\draw[line width = 0.4mm] (1,1) -- (0,1) ; 
\draw[line width = 0.3mm]  (1,8) -- (0,9) -- (-0.7, 9) -- (-0.7, -0.7) -- ( 9, -0.7) -- (9, 0.8) ;
\draw[line width=0.3mm] (1.1, 7.9) -- (4,5) ;
\draw[line width = 0.3mm] (10, 0) -- (5.6, 4.4) ;
\draw[line width = 0.3mm] (0,10) -- (4.6, 5.4) ;
\draw[line width = 0.3mm] (5,5) -- (6,5) ;
\draw[line width = 0.3mm] (6.2,4.8) -- (11, 0) ;
%\draw[->, thick]  (0.7, 2.3) -- (0.4,2.6) ;
%\draw[line width = 0.5mm]  (3,0) -- (3,1) ;
\draw[-> , thick]  (5,5) -- (5.7,4.3) ;
\draw[->> , thick]  (-0.7,4) -- (-0.7,4.4) ;
\draw[-> , thick]  (6,4) -- (6.7,3.3) ;
\draw[->, thick] (7,3) -- (7.7,2.3);
\draw[->, thick] (8,2) -- (8.7,1.3) ;
\draw[->, thick] (9,1) -- ( 9.7, 0.3) ; 
\draw[->>, thick] ( 5.4,  -0.7) -- (5, -0.7) ;
\draw[->>, thick] ( 9,  0.8) -- (9, 0.4) ;
%\draw[->, thick] (10, 0) -- ( 9.4,  0.6) ;
\draw[->> , thick]  (0,9) -- (0.5,8.5) ;
%\draw[-> , thick]  (1,9) -- (0.4,9.6) ;
%\draw[->, thick] (5,5) -- (5.6,5) ;
\draw[->>>>>>, thick] (6.2,4.8) -- (6.6,4.4) ;
%\draw[->, thick] (10, 1) -- (10.6,0.4 ) ;
\draw[->>>, thick]  (3,6) -- (3.5,5.5) ;
\draw[->>>, thick]  (1.2,7.8) -- (1.7, 7.3) ;
\draw[->>>>>, thick]  (5.2,5) -- (5.8, 5) ;
\draw[->>>>>>, thick]  (9, 2) -- (9.5, 1.5) ;
\draw[->>>>, thick]   (3.5,6.5) -- (3,7) ;
\draw[->>>>, thick]  (1.7, 8.3) -- (1.2, 8.8) ;
%\draw[line width = 0.5mm]   (4,0) -- (0,4) ;
%\draw[-> , thick]  (4,0) -- (3.4, 0.6) ;
%\draw[->, thick] (1, 3) -- (0.4, 3.6) ;
%\draw[line width = 0.5mm] (2,2) -- (3,2) -- (5,0) ;
%\draw[-> , thick]  (2,2) -- (2.6, 2) ;
%\draw[->, thick] (3,2) -- (3.4, 1.6) ;
%\draw[->, thick] (4.4,0.6) -- (4.7, 0.3) ;
%\draw[line width=0.5mm] (0,3) -- (-0.7,3) -- (-0.7, -0.7) -- (3,-0.7) -- (3,0) ;
\draw[thick] (6, 0) node[anchor = west]{$\boldsymbol{\cdots}$} ;
\draw[thick] (0,5) node[anchor = south]{$\boldsymbol{\vdots}$} ;
\draw[thick] (2.65,6.2) node[anchor = west]{\tiny \bf (III)} ;
\draw[thick] (-0.4, 9) node[anchor = south]{\tiny \bf (II)} ;
\draw[thick] (-0.7,4) node[anchor = east]{\tiny \bf (II)} ;
\draw[thick] (5 , -1) node[anchor = west]{\tiny \bf (II)} ;
%\draw[thick] (7.8,0.4) node[anchor = west]{\tiny \bf (I)} ;
\draw[thick] (9,-0.7) node[anchor = west]{\tiny \bf (II)} ;
\draw[thick] (8.2,2.3) node[anchor = north]{\tiny \bf (I)} ;
\draw[thick] (4,6) node[anchor = west]{\tiny \bf (IV)} ;
%\draw[thick] (0,10.8) node[anchor = north]{\tiny \bf (II)} ;
%\draw[thick] (7.2,4.1) node[anchor = east]{\tiny \bf (II)} ;
\draw[thick] (5.7,5.1) node[anchor = south]{\tiny \bf (V)} ;
\draw[thick] (7.4,3.7) node[anchor = south]{\tiny \bf (VI)} ;
%\draw[thick] (10.8,0) node[anchor = west]{\tiny \bf (III)} ;
\draw[dashed, thick] (4+4.7, 1.3) -- (4+5.3, 1.3) -- (4+5.3, -0.3) -- (4+4.7, -0.3) -- (4+4.7, 1.3) ;
\draw[dashed, thick] (-0.4, 4+5) -- (0, 4+5.4) -- (1.4, 4+4) -- (1, 4+3.6) -- (-0.4, 4+5) ;
 \end{tikzpicture}
 \end{subfigure}
 \caption*{6b) For $(d_{i+2,i}, \ldots, d_{2i+2, 0}),\ (d_{0, 2i+1}, d_{1, 2i}),\ (d_{2, 2i-1}, \ldots, d_{i, i+1}), \  (d_{i,i+2}, \ldots, d_{0, 2i+2}) , \  (d_{i+2, i+1})$, \\ \hspace*{0.5cm}
 $(d_{i+3, i}, \ldots, d_{2i+3, 0})$ using $d_{i+1, i+1}$ for $i+1< \big\lfloor \frac{n}{2} \big\rfloor$}
    \end{figure}
\end{center}
When $i+1< \big\lfloor \frac{n}{2} \big\rfloor $, the points that lie on 6 paths (I)--(VI) in Fig 6b),  are all the computable points $d_{k,s}$'s after adding $d_{i+1, i+1}$.
And in this case, when we draw a gray region that encloses all the points in Fig 6b), then its shape would resemble the gray region we started with in 6a).
This justifies the choice of the shape of the gray region in Fig. 6a).

Finally suppose $i+1= \big\lfloor \frac{n}{2}\big\rfloor$.
We branch-out into two subcases,  correspond to $n$ being even or odd.
When $n$ is even, $i+1=\frac{n}{2}$, in which case, we do not have to compute the points on line segments (V)--(VI), as their indices go out of the those of $d_{k,s}$ in desired tuple.
Moreover, when $i+1=\frac{n}{2}$, observe in Fig. 6b), that points on line segments (I)--(IV) together with those inside the shaded region constitute all variables in the desire tuple $(d_{k,s})$.
Finally we assume $n$ to be odd and $i+1=\frac{n-1}{2}$.
Following the progress in Fig. 6b), observe that we are only left to compute the variables $(d_{i+1, i+2}, d_{i, i+3}, \ldots, d_{0,n}$; whose indices fall along the line segment that lies above (IV) in Fig. 6b).
And this is done similar to the procedure in 3f), or to the procedure along path (VI) in 4b).
Namely, we first apply \eqref{Eqn length 3 at n} at $n$ together with \eqref{Eqn s- (k-1) leq -2} at $(k-1,s)=(0, n-1)$, to determine $(d_{0,n}, d_{1, n-1})$ using the system
\[
 - \ a \ d_{1,n-1}\ + \ (na+b) \ d_{0,n}\ \ = \ \ - B(n)\ d_{n,0}  \qquad \text{and}  \qquad  A(1)\ d_{1,n-1}\ +\ A(n)\ d_{0,n}\ \ =\ \  - C(n-1) \ d_{0,n-1}.
\]
Then, starting from $d_{1, n-1}$ and using equations \eqref{Eqn s- (k-1) leq -2} with their centers $(k-1,s)$ successively descending along line segment (IV) in Fig. 6b), we can iteratively compute $d_{2, n-2}, \ldots, d_{i, i+3}$.
This completes the proof of the theorem \big(in view of Note \ref{Note redundancies})\big).
\end{proof}

\section{Proof of Theorem \ref{Theorem maximal vectors}(A2): Maximal vectors with weight $(3,3)$} \label{Section proof of (A2)}
\begin{proof}[{{\it \bf Proof of Theorem \ref{Theorem maximal vectors}(A2)}}]
The proof of this part proceeds similar that of the previous part.
Our computations involve only the following ordered (increasing from left to right and top to bottom) right-normed Lie words from a suitable PBW root-basis of $\mathfrak{n}^-$.
\begin{align}
\begin{aligned}[t]  \big[[f_2, f_1],\ f_2 \big]; \ \ \Big[ \big[ [f_2, f_1],\ f_1\big],\ f_1\Big] ; \ \ \big[ [f_2, f_1],\  f_1\big] ;  \ \ [f_2, f_1] ; \ \ f_2 ; \ \ f_1 &\\
\Big[\big[ [f_2, f_1],\ f_1\big],\ [f_2, f_1] \Big]; \ \ \Big[\Big[ [f_2, f_1],\ f_1\big],\ f_1\Big],\  f_2 \Big]; \ \ \Big[\big[[f_2, f_1] ,\ f_1\big],\ f_2\Big];  \qquad  &\\
 \Big[ \big[ [f_2, f_1], \ f_2\big] ,\ [f_2,\ f_1] \Big];  \ \  \Big[ \Big[ \big[ [f_2, f_1], \ f_1\big] ,\ f_2 \Big] ,\ f_2 \Big];  \ \  \Big[ \big[ [f_2, f_1],\ f_2\big],\ f_2\Big]  ; \qquad \qquad & \\
  \Big[\big[ [f_2, f_1], \ f_2\big] ,\ \big[[f_2,\ f_1], f_1\big] \Big];  \ \  \bigg[ \Big[ \Big[ \big[ [f_2, f_1], \ f_1\big],\ f_1 \Big] ,\ f_2 \Big] ,\ f_2 \bigg]; \qquad \qquad \qquad  & \\
\quad \bigg[ \Big[\big[ [f_2, f_1], \ f_1\big] ,\ f_2 \Big], \ [f_2, \ f_1 ] \bigg]; \qquad \qquad \qquad \qquad  &
\end{aligned}
\end{align}
Their ordering is based on the construction of {\it standard monomials} by \cite{Hall} (see also \cite{3-words}), and their weights are ``at most $(3,3)$'' when we ignore the negativity of the roots involved.
In view of their lengths, we use simplified notation by noting only the sequence of their indices as follows:
\begin{align}\label{Liw word root basis symbol list}
\begin{aligned}[t]  21\ \Big|\ 2\ ; \qquad  2111\  ;\qquad  211 \ ;  \qquad 21  \ ;\qquad 2\  ;\qquad 1 &\\
211\ \Big|\ 21\ ; \qquad 2111\ \Big|\ 2\ ; \qquad 211\ \Big|\ 2\ ;  \qquad  &\\
21\ \Big|\ 2 \ \Big|\ 21 \ ; \qquad 211\ \Big|\ 2\ \Big|\ 2\ ; \qquad 21\ \Big|\ 2 \ \Big|\ 2 \ ; \qquad \qquad & \\
21\ \Big|\ 2 \ \Big|\ 211\ ; \qquad 2111\ \Big|\ 2\ \Big|\ 2\ ; \qquad \qquad\qquad  & \\
  211\ \Big|\ 2\ \Big|\ 21\ ; \qquad \qquad \qquad\qquad    &
\end{aligned}
\end{align}
We now list the PBW monomials on the above root-basis elements that yield weight vectors in $M(\lambda)_{\lambda- 3\alpha_1-3\alpha_2}$, $M(\lambda)_{\lambda-2\alpha_1-3\alpha_2}\text{ and } M(\lambda)_{\lambda- 3\alpha_1-2\alpha_2}$.
In this, the products of above Lie words in $U(\mathfrak{n}^-)$ are compactified to products of the corresponding index-sequences as follows.
For example, \[
\big[[f_2, f_1],\ f_2\big] \times [f_2, f_1]\times f_1 \ \   \rightsquigarrow\ \  21\ \Big|\ 2\ \ \star\ \ 21 \ \ \star\ \ 1\ ,\]
and the associated breakage under the (generalized) Kostant partition function of the total weight of that monomial is $\alpha_1+2\alpha_2,\ \alpha_1+\alpha_2,\ \alpha_1$.
We store this data as well in the following list of all PBW monomials that we will deal-with in the proof.
\bigskip\\
{\bf Monomials with total weight (3,3)}\bigskip
\begin{enumerate}[1)]
\item $2111\  \Big | \ 2 \  \Big |\ 2$
\item $21\  \Big | \ 2 \  \Big |\ 211$
\item $211\  \Big | \ 2 \  \Big |\ 21$ \hspace*{3cm} $3\alpha_1+3\alpha_2$
\bigskip
\item $211\  \Big | \ 2 \  \Big |\ 2 \ \ \star\ \ 1$
\item $21\  \Big | \ 2 \  \Big |\ 21 \ \ \star\ \ 1$ \hspace*{3cm} $2\alpha_1+3\alpha_2, \ \alpha_1$\bigskip
\item $21\  \Big | \ 2 \  \Big |\ 2\ \ \star \ \ 1\ \ \star\ \ 1$\hspace*{3cm} $\alpha_1+3\alpha_2,\ \alpha_1, \ \alpha_1$\bigskip
\item $2111\ \Big | \ 2\ \ \star\ \ 2$ 
 \item $211\ \Big | \ 21\ \ \star\ \ 2$ \hspace*{3cm} $3\alpha_1+2\alpha_2,\ \alpha_2$
 \bigskip
 \item $2111\ \ \star\ \ 2\ \ \star\ \ 2 $ \hspace*{3cm} $3\alpha_1+\alpha_2,\ \alpha_2,\ \alpha_2$\bigskip
 \item $211\ \Big|\ 2\ \ \star\ \ 21$ \hspace*{3cm} $2\alpha_1+2\alpha_2,\ \alpha_1+\alpha_2$\bigskip
 \item $211\ \Big|\ 2\ \ \star\ \ 2\  \ \star\ \ 1$\hspace*{3cm} $2\alpha_1+2\alpha_2,\ \alpha_2,\ \alpha_1$\bigskip
 \item $21\ \Big|\ 2\ \ \star\ \ 211$ \hspace*{3cm} $\alpha_1+2\alpha_2,\ 2\alpha_1+\alpha_2$\bigskip
 \item $21\ \Big|\ 2\ \ \star\ \ 21\ \ \star\ \ 1 $ \hspace*{3cm} $\alpha_1+2\alpha_2,\ \alpha_1+\alpha_2,\ \alpha_1$\bigskip
 \item $21\ \Big|\ 2\ \ \star\ \ 2\ \ \star\ \ 1\ \ \star\ \ 1$ \hspace*{3cm} $\alpha_1+2\alpha_2,\ \alpha_2,\ \alpha_1,\ \alpha_1$\bigskip
 \item $211\ \ \star\ \ 21\ \ \star\ \ 2$ \hspace*{3cm} $2\alpha_1+\alpha_2,\ \alpha_1+\alpha_2,\ \alpha_2$\bigskip
 \item $211\ \ \star\ \ 2\ \ \star\ \ 2 \ \ \star\ \ 1$ \hspace*{3cm}\ $2\alpha_1+\alpha_2,\ \alpha_2,\ \alpha_2,\ \alpha_1$\bigskip
 \item $21\ \ \star \ \ 21 \ \ \star\ \ 21$ \hspace*{3cm} $\alpha_1+\alpha_2,\ \alpha_1+\alpha_2,\ \alpha_1+\alpha_2$\bigskip
 \item $21 \ \ \star\ \ 21\ \ \star\ \ 2\ \ \star\ \ 1$ \hspace*{3cm} $\alpha_1+\alpha_2,\ \alpha_1+\alpha_2,\ \alpha_2,\ \alpha_1$\bigskip
 \item $21 \ \ \star\ \ 2 \ \ \star\ \ 2\ \ \star\ \ 1 \star\ \ 1$ \hspace*{3cm} $\alpha_1+\alpha_2,\ \alpha_2,\ \alpha_2,\ \alpha_1,\ \alpha_1$\bigskip
 \item $2 \ \ \star\ \ 2 \ \ \star \ \ 2 \ \ \star\ \ 1\ \ \star\ \ 1 \star \ \ 1$ \hspace*{3cm} $\alpha_2,\ \alpha_2,\ \alpha_2,\ \alpha_1,\ \alpha_1,\ \alpha_1$
\end{enumerate}
\bigskip
{\bf Monomials with total weight (2,3)}\bigskip
\begin{enumerate}[1)]
    \item $211\ \Big| \ 2\ \Big| \ 2$
    \item $21 \  \Big| \ 2\ \Big| \ 21 $ \hspace*{3cm} $2\alpha_1+3\alpha_2$ \bigskip
\item $21\ \Big| \ 2\ \Big|\ 2\  \ \star\ \ 1$ \hspace*{3cm} $\alpha_1+3\alpha_2,\ \alpha_1$ \bigskip
\item $211\ \Big|\ 2\  \ \star\ \ 2$ \hspace*{3cm} $2\alpha_1+2\alpha_2,\ \alpha_2$\bigskip
\item $21\ \Big|\ 2\ \ \star \ \ 21$ \hspace*{3cm} $\alpha_1+2\alpha_2,\ \alpha_1+\alpha_2$ \bigskip
\item $21\ \Big|\ 2\ \ \star \ \ 2 \ \ \star\ \ 1$ \hspace*{3cm} $\alpha_1+2\alpha_2,\ \alpha_2,\ \alpha_1$ \bigskip
\item $211\ \ \star\ \ 2\ \ \star\ \ 2$ \hspace*{3cm} $2\alpha_1+\alpha_2,\ \alpha_2,\ \alpha_2$ \bigskip
\item $21 \ \ \star\ \ 21\ \ \star\ \ 2$ \hspace*{3cm} $\alpha_1+\alpha_2,\ \alpha_1+\alpha_2,\ \alpha_2$ \bigskip
\item $21\ \ \star\ \ 2 \ \ \star\ \ 2\ \ \star\ \ 1$ \hspace*{3cm} $\alpha_1+\alpha_2,\ \alpha_2,\ \alpha_2,\ \alpha_1$ \bigskip
\item $2\ \ \star \ \ 2\ \  \star\ \ 2 \ \  \star\ \ 1\ \ \star\ \ 1$ \hspace*{3cm} $\alpha_2,\ \alpha_2,\ \alpha_2,\ \alpha_1,\ \alpha_1$
\end{enumerate}\bigskip
\bigskip
{\bf Monomials with total weight (3,2)}\bigskip
\begin{enumerate}[1)]
    \item $2111\ \Big| \ 2$ 
    \item $211\ \Big|\ 21$ \hspace*{3cm} $3\alpha_1+2\alpha_2$ \bigskip
    \item $211\ \Big| \ 2 \  \ \star\ \ 1$ \hspace*{3cm} $2\alpha_1+2\alpha_2,\ \alpha_1$\bigskip
    \item $2111\ \ \star\ \ 2$ \hspace*{3cm} $3\alpha_1+\alpha_2,\ \alpha_2$\bigskip
    \item $211\  \star\ \ 21$ \hspace*{3cm} $2\alpha_1+\alpha_2,\ \alpha_1+\alpha_2$\bigskip
    \item $211 \ \ \star\ \ 2\ \ \star\ \ 1$ \hspace*{3cm} $2\alpha_1+\alpha_2,\ \alpha_2,\ \alpha_1$\bigskip
\item $21\ \Big|\ 2\ \ \star\ \ 1\ \ \star\ \ 1$ \hspace*{3cm} $\alpha_1+2\alpha_2,\ \alpha_1,\ \alpha_1$ \bigskip
\item $21\ \ \star\ \ 21\ \ \star\ \ 1$ \hspace*{3cm} $\alpha_1+\alpha_2,\ \alpha_1+\alpha_2,\ \alpha_1$ \bigskip
\item $21\ \ \star\ \ 2\ \ \star\ \ 1\ \ \star 1 $ \hspace*{3cm} $\alpha_1+\alpha_2,\ \alpha_2,\ \alpha_1,\ \alpha_1$\bigskip
\item $2\ \ \star\ \ 2\ \ \star\ \ 1\ \ \star\ \ 1\ \ \star\ \ 1$ \hspace*{3cm} $\alpha_2,\ \alpha_2,\ \alpha_1,\ \alpha_1,\ \alpha_1$.
\end{enumerate}\bigskip
We now carefully write the actions of $e_1$ and $e_2$ on the ``images of above $(3,3)$-weighted'' monomials in $M(\lambda)_{\lambda-3\alpha_1-3\alpha_2} $.
For example,
\begin{center}
$e_1$ sends $\bigg[\Big[ \Big[ \big[[f_2, f_1],\ f_1\big], \ f_1 \Big],\ f_2\Big], \ f_2\bigg]\ m_{\lambda}$ to $(-3a-3b)\bigg[\Big[ \big[[f_2, f_1],\ f_1\big] ,\ f_2\Big], \ f_2\bigg]\ m_{\lambda}$
\end{center}
(which can be verified by derivation action of $e_1$), and it is re-casted in terms of words on indices in point 1) in $e_1$-action list below. 
And also for simplicity, we omitted $m_{\lambda}$ at the right-end of each monomial in the (20+20) raising computations below, and we used the identification between PBW monomials in $U(\mathfrak{n}^-)$ and their images in $M(\lambda)$. 
For example, point 4) below denotes
\begin{center}
$\bigg[\Big[ \big[ [f_2, f_1],\ f_1 \big],\  f_2 \Big],\ f_2\bigg] \times f_1\cdot m_{\lambda}\ \  \overset{e_1}{\longrightarrow}$\\ 
$(-2a-b) \ \Big[\big[ [f_2, f_1],\  f_2 \big],\ f_2\Big] \times f_1 \cdot m_{\lambda}\ +\ \frac{-b}{2}(M_1-1) \ \bigg[\Big[ \big[ [f_2, f_1],\ f_1 \big],\  f_2 \Big],\ f_2\bigg]\cdot m_{\lambda}$.
\end{center}
\bigskip
$\boldsymbol{e_1-}$\textbf{action} \bigskip
\begin{enumerate}[1)]
\item $2111\ \Big|\ 2\ \Big|\ 2 \quad  \longrightarrow \quad (-3a-3b) \ \ 211\ \Big|\ 2\ \Big|\ 2$\bigskip
\item $21\ \Big|\ 2\ \Big|\ 211\quad \longrightarrow \quad (-2a-b)\ \ 21\ \Big|\ 2\ \Big|\ 21$ \hfill \big(as $21\ \Big|\ 21\ \Big|\ 2\ =0 $\big)\bigskip
\item $211\ \Big|\ 2\ \Big|\ 21 \quad \longrightarrow\quad (-2a-b)\ \ 21\ \Big|\ 2\ \Big|\ 21\quad + (-a)\ \ 211\ \Big|\ 2\ \Big|\ 2 $\bigskip
\item $211\ \Big|\ 2\ \Big|\ 2\ \ \star\ \ 1\quad \longrightarrow \quad  (-2a-b)\ \ 21 \ \Big|\ 2\ \Big|\ 2\ \ \star\ \ 1 \quad + \frac{-b}{2}(M_1-1)\ \ 211\ \Big|\ 2\ \Big|\ 2$ \bigskip
\item $21\ \Big|\ 2\ \Big|\ 21\ \ \star\ \ 1 \quad \longrightarrow \quad (-a) \ \ 21\ \Big|\ 2\ \Big|\ 2\ \ \star\ \ 1 \quad +\ \frac{-b}{2}(M_1-1)\ \ 21\ \Big|\ 2\ \Big|\ 21$\bigskip
\item $21 \ \Big|\ 2\ \Big|\ 2\  \ \star\ \ 1 \ \ \star\ \ 1\quad \longrightarrow \quad  \big[-b(M_1-1)+b\big]\ 21\ \Big|\ 2\ \Big|\ 2\ \ \star\ \ 1 $\bigskip
\item $211 1\ \Big|\ 2\ \ \star\ \ 2\quad \longrightarrow 
\quad (-3a-3b)\ \ 211\ \Big|\ 2 \ \ \star\ \ 2 $\bigskip
\item $211\ \Big|\ 21\ \ \star\ \ 2 \quad \longrightarrow \quad (-a)\ \ 211\ \Big|\ 2\ \ \star\ \ 2 $\bigskip
\item $2111\ \ \star\ \ 2\ \ \star\ \ 2 \quad \longrightarrow \quad (-3a-3b)\ \ 211\ \ \star\ \ 2\ \ \star\ \ 2$\bigskip
\item $211\ \Big|\ 2\ \ \star\ \ 21\quad \longrightarrow \quad (-2a-b)\ \ 21\ \Big|\ 2\ \ \star\ \ 21\quad +(-a)\ \ 211\ \Big|\ 2\ \ \star\ \ 2$ \bigskip
\item $211\ \Big|\ 2\ \ \star\  \ 2\ \ \star\ \ 1  \quad \longrightarrow \quad (-2a-b)\ \ 21\ \Big|\ 2\ \ \star\ \ 2\ \ \star\ \ 1\quad + \frac{-b}{2}(M_1-1)\ \ 211\ \Big|\ 2 \ \ \star\ \ 2$\bigskip
\item $21\ \Big|\ 2 \ \ \star\ \ 211\quad \longrightarrow\quad (-2a-b)\ \ 21\ \Big|\ 2\ \ \star\ \ 21$\bigskip
\item $21\ \Big|\ 2\ \ \star\ \ 21\ \ \star\ \ 1 \quad  \longrightarrow\quad (-a) \ \  21\ \Big|\ 2\ \ \star\ \ 2\ \ \star\ \ 1\quad + \frac{-b}{2}(M_1-1)\ \ 21\ \Big|\ 2\ \ \star\ \ 21$ \bigskip
\item $21\ \Big|\ 2 \ \ \star\ \ 2 \ \ \star\ \ 1\ \ \star\ \ 1\quad \longrightarrow \quad \big[-b(M_1-1)+b\big]\ 21\ \Big|\ 2\ \ \star\ \ 2\ \ \star\ \ 1$ \bigskip
\item $211\ \ \star\ \ 21\ \ \star\ \ 2\quad \longrightarrow \quad (-2a-b)\ \ 21\ \ \star\ \ 21\ \ \star\ \ 2 \quad + (-a)\ \ 211\ \ \star\ \ 2\ \ \star\ \ 2$\bigskip
\item $211\ \ \star\ \ 2\ \ \star\ \ 2\ \ \star\ \ 1\quad \longrightarrow\quad (-2a-b)\ \ 21\ \ \star\ \ 2\ \ \star\ \ 2\ \ \star \ \ 1\quad + \frac{-b}{2}(M_1-1)\ \ 211\ \ \star\ \ 2\ \ \star\ \ 2$ \bigskip
\item $21\ \ \star\ \ 21\ \ \star\ \ 21 \quad \longrightarrow\quad (-a)\ \big[\  2 \ \ \star\ \ 21\ \ \star\ \ 21\quad + \ \ 21\ \ \star\ \ 2\ \ \star\ \ 21\quad +\ \ 21\ \ \star\ \ 21\ \ \star\ \ 2\ \big] $\\
\hspace*{3.5cm} $=\ (-a)\ \big[\  2\ \Big| \  21\ \ \star\ \ 21\quad + \ \ (2) \big(21\ \ \star\ \ 2\ \Big| \ 21\big)\quad +\ \ (3) \big(21\ \ \star\ \ 21\ \ \star\ \ 2\big)\ \big] $\\
\hspace*{3.5cm} $=\ (-a)\ \big[\ (3) \big(2\ \Big| \  21\ \ \star\ \ 21\big)\quad + \ \ (2) \Big(21\ \Big| \ \Big(2\ \Big| \ 21\Big)\Big)\quad +\ \ (3) \big(21\ \ \star\ \ 21\ \ \star\ \ 2\big)\ \big] $\\
\hspace*{3.5cm} $=\ (3a)\ 21\ \Big| \  2\ \ \star\ \ 21\quad + (-3a) \big(21\ \ \star\ \ 21\ \ \star\ \ 2\big)\quad + \ \ (-2a)\  21\ \Big| \ 2\ \Big| \ 21 $\bigskip
\item $21 \ \ \star\ \ 21\ \ \star\ \ 2\ \ \star\ \ 1 \quad \longrightarrow \quad (-a)\big[\  2 \ \ \star\ \ 21\ \ \star\ \ 2\ \ \star\ \ 1\quad + \quad 21\ \ \star\ \ 2\ \ \star\ \ 2\ \ \star\ \ 1  \ \big]\quad +$\\
\hspace*{5.3cm} $\frac{-b}{2}(M_1-1)\ \ 21\ \ \star\ \ 21\ \ \star\ \ 2$\\
\hspace*{3cm} $= (-2a)\ 21\ \ \star\ \ 2\ \ \star\ \ 2\ \ \star\ \ 1\quad +\  \frac{-b}{2}(M_1-1)\ 21\ \ \star\ \ 21\ \ \star\ \ 2\quad  + \ (a)\ 21\ \Big|\ 2\ \ \star\ \ 2 \ \ \star\ \ 1$\bigskip
\item $21\ \ \star\ \ 2\ \ \star\ \ 2\ \ \star\ \ 1\ \ \star\ \ 1 \quad \longrightarrow \quad (-a)\ \ 2\ \ \star\ \ 2\ \ \star\ \ 2\  \ \star\ \ 1\ \ \star\ \ 1$\\
\hspace*{6.2cm} $+ \ \big[ -b(M_1-1)+b\big]\ 21\ \ \star\ \ 2\ \ \star\ \ 2\ \ \star\ \ 1$\bigskip
\item $2\ \ \star\ \ 2\ \ \star\ \ 2\ \ \star\ \ 1\ \ \star\ \ 1\ \ \star 1 \quad \longrightarrow \quad \big[\frac{-3}{2}(b)(M_1-1) + 3b\big]\ 2\ \ \star\ \ 2\ \ \star\ \ 2\  \ \star\ \ 1\ \ \star\ \ 1$
\end{enumerate}
 \bigskip
$\boldsymbol{e_2-}$\textbf{action} \bigskip
\begin{enumerate}[1)]
\item $2111\ \Big|\ 2\ \Big|\ 2 \quad  \longrightarrow \quad  \big[(-2)(d+3c)-d\big]\ 2111\ \Big|\ 2 $\bigskip
\item $21\ \Big|\ 2\ \Big|\ 211\quad \longrightarrow \quad (-c-d)\ \ 21\ \Big|\ 211 \quad +\ (c)\ \ 12\ \Big|\ 211 \quad = \quad (2c+d)\ \ 211\ \Big|\ 21 $\bigskip
\item $211\ \Big|\ 2\ \Big|\ 21 \quad \longrightarrow\quad (-2c-d)\ \ 211\ \Big|\ 21\quad + (c)\ \ 211\ \Big|\ 2\ \Big|\ 1 $\\ 
\hspace*{3cm} $= (-c-d)\ \ 211\ \Big|\ 21\quad +\  (c)\ \ 2111\ \Big|\ 2 $\bigskip
\item $211\ \Big|\ 2\ \Big|\ 2\ \ \star\ \ 1\quad \longrightarrow \quad  \big[ -2(d+2c)-d \big]\ \  211\ \Big|\ 2\ \ \star\ \ 1$ \bigskip
\item $21\ \Big|\ 2\ \Big|\ 21\ \ \star\ \ 1 \quad \longrightarrow \quad (c)\ \ 1\ \Big|\ 2\ \Big|\ 21\ \ \star\ \ 1 \quad + (c)\ \ 21\ \Big|\ 2\ \Big|\ 1\ \ \star\ \ 1\quad = \quad (c)\ \ 211\ \Big|\ 2\ \ \star\ \ 1 $\bigskip
\item $21 \ \Big|\ 2\ \Big|\ 2\  \ \star\ \ 1 \ \ \star\ \ 1\quad \longrightarrow \quad  \big[-2(d+c)-d\big]\ 21\ \Big|\ 2 \ \   \star\ \ 1\ \ \star\ \ 1\quad + \ (c)\ \  1\ \Big| \ 2\ \Big|\ 2\ \ \star\ \ 1\ \ \star \ \ 1  $\\ 
\hspace*{4.7cm} $=(-3)(d+c)\ \ 21\ \Big|\ 2\ \ \ \star\ \ 1\ \ \star\ \ 1$\bigskip
\item $211 1\ \Big|\ 2\ \ \star\ \ 2\quad \longrightarrow 
\quad (-d-3c)\ \ 2111\ \ \star\ \ 2\quad +\ \frac{-d}{2}(M_2-1)\ \ 2111\ \Big|
\ 2$\bigskip
\item $211\ \Big|\ 21\ \ \star\ \ 2 \quad \longrightarrow \quad (c)\ \ 2111\ \ \star\ \ 2\quad +\ \frac{-d}{2}(M_2-1)\ \ 211\ \Big|\ 21 $\bigskip
\item $2111\ \ \star\ \ 2\ \ \star\ \ 2 \quad \longrightarrow \quad \big[-d(M_2-1)+d\big]\ 2111\ \ \star\ \ 2 $\bigskip
\item $211\ \Big|\ 2\ \ \star\ \ 2 1\quad \longrightarrow \quad (-d-2c)\ \ 211\ \ \star\ \ 21\quad +\  (c)\ \ 211\ \Big|\ 2\ \ \star\ \ 1  $ \bigskip
\item $211\ \Big|\ 2\ \ \star\  \ 2\ \ \star\ \ 1  \quad \longrightarrow \quad (-d-2c)\ \ 211\ \ \star\ \ 2\ \ \star\ \ 1\quad +\ \big[\frac{-d}{2}(M_2-1) +c\big]\ 211\ \Big|\ 2\ \ \star\ \ 1 $\bigskip
\item $21\ \Big|\ 2 \ \ \star\ \ 211 \quad \longrightarrow \quad (-c)\ \ 21\ \ \star\ \ 211\quad + \ (-c-d)\ \ 21\ \ \star\ \ \ 211 \quad =\ \  (-2c-d)\ \ 21\ \ \star\ \ 211$ \\
\hspace*{4cm} $= \ \ (-2c-d)\ \ 211\ \ \star\ \ 21\quad +\ (2c+d)\ \ 211\ \Big|\ 21$ \bigskip
\item $21\ \Big|\ 2\ \ \star\ \ 21\ \ \star\ \ 1 \quad  \longrightarrow\quad (-c)\ \ 21\ \ \star\ \ 2 1\ \ \star\ \ 1\quad +\  (-c-d)\ \ 21\ \ \star\ \ 21\ \ \star\ \ 1 \quad + \ (c)\ \ 21\ \Big|\ 2\  \ \star\ \ 1\ \ \star\ \ 1 $\\
\hspace*{4cm} $=\ \ (-2c-d)\ \ 21\ \ \star\ \ 21\ \ \star\ \ 1 \quad + \ (c)\ \ 21\ \Big|\ 2\  \ \star\ \ 1\ \ \star\ \ 1 $ \bigskip
\item $21\ \Big|\ 2 \ \ \star\ \ 2 \ \ \star\ \ 1\ \ \star\ \ 1\quad \longrightarrow \quad (-2c-d)\ \ 21\ \ \star\ \ 2\ \ \star\ \ 1\ \ \star\ \ 1\quad + \big[\frac{-d}{2}(M_2-1)+2c \big]\ 21\ \Big|\ 2\ \ \star\ \ 1 \ \ \star\ \ 1 $ \bigskip
\item $211\ \ \star\ \ 21\ \ \star\ \ 2\quad \longrightarrow \quad (c)\ \ 211\ \ \star\ \ 1\ \ \star \ \ 2\quad + \frac{-d}{2}(M_2-1)\ \ 211\ \ \star\ \ 21  $\\
\hspace*{4cm} $=(-c)\ \ 211\ \ \star\ \ 21 \quad \ + (c)\ \ 211\ \ \star\ \ 2\ \ \star\ \ 1\quad +\ \frac{-d}{2}(M_2-1)\ \ 211\ \ \star\ \ 21 $ \\
\hspace*{4cm} $=\ \  (c)\ \ 211\ \ \star\ \ 2\ \ \star\ \ 1\quad +\ \big[\frac{-d}{2}(M_2-1)\ -c\big]\ \ 211\ \ \star\ \ 21 $ \bigskip
\item $211\ \ \star\ \ 2\ \ \star\ \ 2\ \ \star\ \ 1\quad \longrightarrow\quad  \big[ -d(M_2-1) +2c+d \big] \ 211\ \  \star\ \ 2\ \ \star\ \ 1$ \bigskip
\item $21\ \ \star\ \ 21\ \ \star\ \ 21 \quad \longrightarrow\quad (c)\big[\ 1\ \ \star\ \ 21\ \ \star\ \ 21\quad +\ 21\ \ \star\ \ 1\ \ \star\ \ 21 \quad  +\ 21\ \ \star\ \ 21\ \ \star\ \ 1  \ \big]$\\
\hspace*{2cm} $=\ (c)\ \big[\ -\ 211\ \  \star\ \ 21\quad +\ (2)\ \ 21\ \ \star\ \ 1\ \ \star\ \ 21\quad +\ 21\ \ \star\ \ 21\ \ \star\ \ 1  \ \big] $\\
\hspace*{2cm} $=\ (c)\ \big[\ (3)\ \ 21\ \ \star\ \ 21\ \ \star\ \ 1 \quad + \ \ (-2) \ \ 21 \ \ \star\ \ 211\quad - \ 211\ \ \star\ \ 21 \big] $\\
\hspace*{2cm} $=\ (-c)\ \ 211\ \ \star\ \ 21\quad \ + (3c)\ \ 21\ \ \star\  \ 21\ \ \star\ \ 1\quad + \ (2c)\ \ 211\ \Big|\ 21\quad +\ (-2c)\  \ 211 \ \ \star\ \ 21$\\
\hspace*{2cm} $=\ (-3c)\ \ 211\ \ \star\ \ 21\quad \ + (3c)\ \ 21\ \ \star\  \ 21\ \ \star\ \ 1\quad + \ (2c)\ \ 211\ \Big|\ 21 $ \bigskip
\item $21 \ \ \star\ \ 21\ \ \star\ \ 2\ \ \star\ \ 1 \quad \longrightarrow \quad (c)\big[\  1\  \ \star\ \ 21\ \ \star\ \ 2\ \  \star\ \ 1  \quad  +\ 21\ \ \star\ \ 1\ \ \star\ \ 2\ \ \star\ \ 1 \ \big]\quad +$\\
\hspace*{5.3cm} $\big[\frac{-d}{2}(M_2-1)+c\big]\ \ 21\ \ \star\ \ 21\ \ \star\ \ 1$\\
\hspace*{1.3cm} $= (-c)\ \ 211\ \ \star\ \ 2\ \ \star\ \ 1\quad +\ (-2c)\ \ 21\ \ \star\ \ 21\ \ \star\ \ 1 \quad +\ (2c)\ \ 21\ \ \star\ \ 2\ \ \star\ \ 1\ \ \star\ \ 1$ \\
\hspace*{2cm}$+ \big[ \frac{-d}{2}(M_2-1)+c\big]\ 21\ \ \star\ \ 21\ \ \star\ \ 1 $\\
\hspace*{1.3cm} $= (-c)\ \ 211\ \ \star\ \ 2\ \ \star\ \ 1\quad +\ (2c)\ \ 21\ \ \star\ \ 2\ \ \star\ \ 1\ \ \star\ \ 1\quad + \big[ \frac{-d}{2}(M_2-1)-c\big]\ 21\ \ \star\ \ 21\ \ \star\ \ 1 $
\bigskip
\item $21\ \ \star\ \ 2\ \ \star\ \ 2\ \ \star\ \ 1\ \ \star\ \ 1 \quad \longrightarrow \quad (c) \  \ 1\ \ \star\ \ 2\ \ \star\ \ 2\ \ \star\ \ 1\ \ \star\ \ 1 \quad + \big[ -d(M_2-1)+4c+d\big]\ 21\ \ \star\ \ 2\ \ \star\ \ 1\ \ \star 1 $\\
\hspace*{0.5cm} $=\ (-c)\ \ 21\ \ \star\ \ 2\ \ \star\ \ 1\ \ \star\ \ 1\quad +\ (-c)\ \ 2\  \ \star\ \ 21\ \ \star 1\  \ \star 1\quad + (c) \ \ 2\ \ \star\ \ 2 \ \ \star\ \ 1\ \ \star\ \ 1\ \ \star\ \ 1\quad +\ \big[ -d(M_2-1)+4c+d\big]\ 21\ \ \star\  \ 2 \ \ \star \ \ 1\ \ \star\ \ 1$\\
\hspace*{0.5cm} $=\ \big[ -d(M_2-1)+3c+d\big]\ 21\ \ \star\  \ 2 \ \ \star \ \ 1\ \ \star\ \ 1 \quad + \ (c)\ \ 21\ \Big|\ 2\ \ \star\ \ 1\ \ \star\ \ 1 \quad +\ (-c)\ \ 21\ \ \star\ \ 2\ \ \star\ \ 1\ \ \star\ \ 1\quad +\ (c)\ \  2\ \ \star\ \ 2 \ \ \star\ \ 1\ \ \star\ \ 1\ \ \star\ \ 1 $
\\
\hspace*{0.5cm} $=\ \big[ -d(M_2-1)+2c+d\big]\ 21\ \ \star\  \ 2 \ \ \star \ \ 1\ \ \star\ \ 1 \quad + \ (c)\ \ 21\ \Big|\ 2\ \ \star\ \ 1\ \ \star\ \ 1 \quad +\  (c)\ \  2\ \ \star\ \ 2 \ \ \star\ \ 1\ \ \star\ \ 1\ \ \star\ \ 1 $
\bigskip
\item $2\ \ \star\ \ 2\ \ \star\ \ 2\ \ \star\ \ 1\ \ \star\ \ 1\ \ \star 1 \quad \longrightarrow \quad\big[ \frac{-3}{2}(d)(M_2-1)+9c+3d \big]\ 2\ \ \star\ \ 2\ \ \star\ \ 1\ \ \star\ \ 1\ \ \star\ \ 1$
\end{enumerate}\medskip
Now we apply all of this data.
We begin by a weight vector $v= X_1  \bigg[ \Big[ \Big[ \big[ [f_2, f_1], \ f_1\big],\ f_1 \Big] ,\ f_2 \Big] ,\ f_2 \bigg]\cdot m_{\lambda}\ +\ X_2 \Big[\big[ [f_2, f_1], \ f_2\big] ,\ \big[[f_2,\ f_1], f_1\big] \Big]\cdot m_{\lambda} \ +\ \cdots\  +\  X_{19} [f_2, f_1]\ \times\ f_2^2\ \times\ f_1^2\cdot m_{\lambda}\ +\ X_{20} f_2^3 \times f_1^3\cdot m_{\lambda}$ in $M(\lambda)_{\lambda-3\alpha_1-3\alpha_2}$, which is a linear combination of the 20 PBW basis vectors from $M(\lambda)_{\lambda- 3\alpha_1-3\alpha_2}$ \big(from the list of (3,3) weighted monomials\big), with coefficients $X_1,\ldots, X_{20}\in \mathbb{C}$.
We observe as follows, the conditions for maximality of $v$.
These are the 20 conditions which should vanish for the vanishings of (10+10) PBW basis vectors $\bigg[\Big[ \big[ [f_2, f_1],\ f_1\big] ,\ f_2\Big], \ f_2\bigg]\cdot m_{\lambda},\ldots, f_2^3\times f_1^2\cdot m_{\lambda}\ \in M(\lambda)_{\lambda-2\alpha_1-3\alpha_2}$ and $\bigg[ \Big[ \big[ [f_2, f_1],\ f_1\big]  ,\ f_1\Big] ,\ f_2   \bigg]\cdot m_{\lambda}, \ldots, f_2^2\times f_1^3\cdot m_{\lambda}\ \in M(\lambda)_{\lambda-3\alpha_1-2\alpha_2}$, under $e_1$'s \& $e_2$'s action on $v$. 
\begin{align}
     -3(a+b)\ X_1\ \ -a\ X_3\ \  -\frac{b}{2}(M_1-1)\ X_4\ =0 \label{Eqn 3,3 (1)} \\
  -(2a+b)\ X_2\ \   -(2a+b)\ X_3\ \ -\frac{b}{2}(M_1-1)\ X_5\ \ -2a\ X_{17} \ =0\label{Eqn 3,3 (2)}\\
     -(2a+b)\ X_4\ \ -a\ X_5\ \ +\big(-b(M_1-1)+b\big)\ X_6\ =0\label{Eqn 3,3 (3)}\\
    -3(a+b)\ X_7\ \ -a\ X_8\ \ -a\ X_{10}\ \ -\frac{b}{2}(M_1-1)X_{11} \ =0\label{Eqn 3,3 (4)}\\
    -(2a+b)\ X_{10}\ \ -(2a+b)\ X_{12}\ \ -\frac{b}{2}(M_1-1)\ X_{13}\ \ +3a\ X_{17} \ =0\label{Eqn 3,3 (5)}\\
    -(2a+b)\ X_{11}\ \ -a\ X_{13}\ \ +\big(-b(M_1-1)+b\big)\ X_{14}\ \ +a\ X_{18} \ =0\label{Eqn 3,3 (6)}\\
    -3(a+b)\ X_{9}\ \ -a\ X_{15}\ \ -\frac{b}{2}(M_1-1)X_{16}\ = 0\label{Eqn 3,3 (7)}\\
    -(2a+b)\ X_{15}\ \ -3a\ X_{17}\ \ - \frac{b}{2}(M_1-1)\ X_{18} \ =0\label{Eqn 3,3 (8)}\\
    -(2a+b)\ X_{16}\ \ -2a\ X_{18}\ \ + \big(-b(M_1-1)+b\big)\ X_{19}\ =0 \label{Eqn 3,3 (9)}\\
    -a\ X_{19}\ \ + \Big(-3\frac{b}{2}(M_1-1)+3b \Big)\ X_{20}\ =0 \label{Eqn 3,3 (10)}
\end{align}
\bigskip
\begin{align}
    - \big(2(3c+d)+d\big)\ X_1\ \ +c\ X_3\ \ - \frac{d}{2}(M_2-1)\ X_7\ =0 \label{Eqn 3,3 (11)}\\
   (2c+d)\ X_2\ \   -(c+d)\ X_3\ \  - \frac{d}{2}(M_2-1)\ X_8\ \   +(d+2c)\ X_{12}\ \ +2c\ X_{17}\ =0 \label{Eqn 3,3 (12)} \\
    -\big(2(d+2c)+d\big)\ X_4\ \ +\ c\ X_5\ \ +c\ X_{10}\ \ -\Big(\frac{d}{2}(M_2-1)-c\Big)\ X_{11}\ =0 \label{Eqn 3,3 (13)}\\
    -(d+3c)\ X_7\ \ +c\ X_8\ \ - [d(M_2-1)-d] \ X_9\ =0 \label{Eqn 3,3 (14)}\\
    -(d+2c)\ X_{10}\ \  -(d+2c)\ X_{12}\ \ -\Big(\frac{d}{2}(M_2-1)+c\Big)\ X_{15}\ \ -3c\ X_{17}\ =0 \label{Eqn 3,3 (15)}\\
    -(d+2c)\ X_{11}\ \ +c\ X_{15}\ \ + \big(-d(M_2-1)+2c+d\big)\ X_{16}\ \ -c\ X_{18}\ =0 \label{Eqn 3,3 (16)}\\
    -3(c+d)\ X_6\ \ + c\ X_{13}\ \ - \Big( \frac{d}{2} (M_2-1)-2c \Big)\ X_{14}\ \ +c\ X_{19} \label{Eqn 3,3 (17)}\ =0\\
    -(d+2c)\ X_{13}\ \ +3c\ X_{17}\ \ -\Big(\frac{d}{2}(M_2-1)+c\Big)\ X_{18} \label{Eqn 3,3 (18)}\ =0\\
    -(d+2c)\ X_{14}\ \ +2c\ X_{18}\ \ +\ \big(-d(M_2-1)+2c+d\big)\ X_{19} \label{Eqn 3,3 (19)}\ =0\\
    c\ X_{19}\ \ + \Big(-\frac{3}{2}d(M_2-1)+9c+3d\Big)\ X_{20}\ =0\label{Eqn 3,3 (20)}
\end{align}
Observe, we obtain a $20 \times 20$ matrix of coefficients for the above system, whose null space yields the desired space of maximal vectors in $M(\lambda)_{\lambda- 3\alpha_1-3\alpha_2}$. 

We write down a method to analytically compute all the variables $X_1, \ldots, X_{20}$. 
In this we first arbitrarily fix  values for the 4-tuple $(X_{20}, \ X_{18},\ X_{15},\ X_{12})$ in $\mathbb{C}^4$.
We show that by each such choice, we can determine the values of the rest of the 16 variables as functions of $(X_{20},\ X_{18},\ X_{15},\ X_{12})$.
For this we use some subset of equations, one at time, to successively determine each of those 16 dependent variables in the previous sentence.
Now recall, Proposition \ref{Prop KK} \eqref{maxl vect count inequality} says that the solution space to \eqref{Eqn 3,3 (1)}--\eqref{Eqn 3,3 (20)} has dimensions at least $\dim \mathfrak{g}_{3\alpha_1+3\alpha_2}+\dim \mathfrak{g}_{\alpha_1+\alpha_2}= 3+1=4$.
And so the previous two lines together show that the tuple $(X_1,\ldots, X_{20})$ built using each fixed $(X_{20},\ X_{18}, X_{15}, \ X_{12})\in \mathbb{C}^4$, indeed satisfy all the equations among \eqref{Eqn 3,3 (1)}--\eqref{Eqn 3,3 (20)}.
That is, the method that we will discuss below will yield every solution to that system, and importantly, the solution space is 4-dimensional. 

We begin by the identity arising by $(3,3)$ satisfying \eqref{Eqn norm equality condition for general mu= (X,Y) in Case (N)}
\begin{equation}\label{Eqn 3,3 condition simplification}
     b (3-M_1)\ \ = \ \ -\frac{ad}{c}(3-M_2) \ - \ 6a. 
\end{equation}
By this, it can be easily seen that L.H.S. of \eqref{Eqn 3,3 (20)} multiplied by $-\frac{a}{c}$ equals the L.H.S of \eqref{Eqn 3,3 (10)}.
So there are essentially 19 equations.

In Equations \eqref{Eqn 3,3 (1)}--\eqref{Eqn 3,3 (20)}, the coefficients that involve $M_1$ or $M_2$ \big(indeed there are 20 such coefficients one in each in each equation among \eqref{Eqn 3,3 (1)}--\eqref{Eqn 3,3 (20)}\big), could possibly vanish for some suitable choices of $M_1$ and $M_2$; for instance at $M_1\in \{1,2,3\}$ or $M_2= 1+\frac{2c}{d}$ etc.
And those are all the coefficients that can vanish.
So, in the below steps, every time we compute a dependent variable (among 16 mentioned above), we use the non-vanishings of the coefficients $3(a+b), \ a, \ (2a+b), \ 2(3c+d)+d, c,\  (2c+d),\  (c+d),\ 2(d+2c)+d $ (as $b,a,c,d\in \mathbb{N}$) which we have to reciprocate.
In particular the coefficients in the previous line avoid $M_1$ and $M_2$, which the reader can easily observe at each step below.
And we omit writing these sanity checks hereafter.
The variable that is fixed or computed at a given step can be read-off from the step index/number.
\begin{enumerate}
    \item[(1-20)] We arbitrarily fix a value of $X_{20}$.
    \item[(2-19)] Using \eqref{Eqn 3,3 (10)} or \eqref{Eqn 3,3 (20)}, we can uniquely determine $X_{19} = -\frac{3b}{2a}(M_1-3) \cdot X_{20}$ in terms of the value of $X_{20}$.
    \item[(3-18)] Next we fix an arbitrary value for $X_{18}$.
    \item[(4-16)] Using \eqref{Eqn 3,3 (9)} we can uniquely find the value of $X_{16}$ in terms of known $(X_{18},\ X_{19})$.
    \item[(5-14)] \eqref{Eqn 3,3 (19)} yields the $X_{14}$ in terms of $(X_{18},\ X_{19})$.
    \item[(6-15)] Next we arbitrarily fix a value for $X_{15}$.
    \item[(7-17)] \eqref{Eqn 3,3 (8)} yields the value of $X_{17}$ in terms of $(X_{15}, \ X_{18})$.
    \item[(8-13)] \eqref{Eqn 3,3 (18)} yields the value of $X_{13}$ in terms of $(X_{17}, \ X_{18})$.
    \item[(9-11)] \eqref{Eqn 3,3 (16)} yields the value of $X_{11}$ in terms of known $(X_{15},\ X_{16},\ X_{18})$.
    \item[(10-9)] \eqref{Eqn 3,3 (7)} yields the $X_{9}$ in terms of $(X_{15}, \ X_{16})$.
    \item[(11-6)] \eqref{Eqn 3,3 (17)} yields the value of $X_{6}$ in terms of $(X_{13}, \ X_{14},\ X_{19})$.
    \item[(12-12)] Finally we arbitrarily fix the value of $X_{12}$.
    \item[(13-10)] \eqref{Eqn 3,3 (15)} yields the value of $X_{10}$ in terms of $(X_{12},\ X_{15},\ X_{17})$.
    \item[(14-7,8)] By the two equations \eqref{Eqn 3,3 (4)} and \eqref{Eqn 3,3 (14)} on unknowns $(X_7,\ X_8)$, and known $(X_{10},\ X_{11},\ X_9)$, 
    \begin{align*}
        -3(a+b) X_7 \ -\ a X_8\ \ = & \ \ a X_{10} \ + \ \frac{b}{2}(M_1-1)X_{11},\\
        -(d+3c)X_7 \ + \ cX_8\ \ =  & \ \ d(M_2-2)X_9.
    \end{align*}
$X_7$ and $X_8$ can be uniquely found from the above system, as its determinant is equal to $-3ac-3bc-ad-3ac\leq -3-3-1-3<0$ as $b,a,c,d\in \mathbb{N}$.
\item[(15-4,5)] Similarly to in Step 14 above, $X_4$ and $X_5$ can be found using \eqref{Eqn 3,3 (3)} and \eqref{Eqn 3,3 (13)}, in terms of known $(X_6,\ X_{10},\ X_{11})$.
\item[(16-2,3)] Once again similar to in Step 14, $X_2$ and $X_3$ can be found using \eqref{Eqn 3,3 (2)} and \eqref{Eqn 3,3 (12)}, in terms of unknowns $(X_5,\ X_{12},\  X_{17})$.
\item[(17-1)] We finally have two equations \eqref{Eqn 3,3 (1)} and \eqref{Eqn 3,3 (11)} that involve unknown $X_1$, in terms of known $(X_3, \ X_5,\ X_8, \ X_{12},\ X_{17})$.
Now we could use either of the two equations to find $X_1$ (in terms of all the other variables).
But what is the guarantee that the values computed in those two ways match?
\end{enumerate}
For addressing the consistency question (in our method) in (17-1), recall that the actual solution space is 4-dimensional.
So let us map each such solution $(X',\ldots, X'_{20})$ to its sub-tuple $(X'_{20},\ X'_{18},\ X'_{15},\ X'_{12})$.
This must be a surjective projection,  as the above steps showed that each choice of $(X_{20},\ X_{18},\ X_{15},\ X_{12})$ uniquely fixes the values of all the other variables (with some ambiguity about $X_1$).
Hence at the terminating step (17-1), the problem therein does not arise, completing the proof of the theorem. 
\end{proof}
\section{Proof of Theorem \ref{Theorem composition series}: On maximal vectors and characters, when Kac--Kazhdan equation has 4 solutions}
Recall when $A_{i,j}\neq 0$ for all $i,j$, \cite[Theorem II]{Naito 2} and \cite{Naito BGG 2} extend Jantzen's character sum formulas, for quotients of Verma $M(\lambda)$ by its Verma submodules that are isomorphic to $M(\lambda-\beta)$, for domestic-imaginary roots (\cite{Wakimoto}) $\beta\in \Delta^+$.
Such roots $\beta$'s are the Weyl group conjugates of the imaginary simple roots.
Also, the simplest cases of such quotients \big($M(\lambda)$ modulo $M(\lambda-\beta_1)$\big) in \cite{Kac--Kazhdan} are when \eqref{Eqn dot action by any positive root} has unique solution $\lambda-\beta_1$.
Our rank-2 setting with $\lambda\in P^{\pm}$, would perhaps be a first step along negative directions, to set-up a platform to explore the problem in the previous line and Jantzen's sum formulas, for {\it alien} (non-domestic) imaginary $\beta_1\in \Delta^+$.
Particularly, it might be interesting to explore Jantzen's filtrations and character sum formulas, for the early cases of unique solutions $(M_1,\ M_2-M_1>0)$ to \eqref{Eqn dot action by any positive root} in Theorem \ref{Theorem composition series}, by \cite[Proposition 3.1]{T-P} over $A(2)$.

    \begin{proof}[{\bf \textnormal Proof of Theorem \ref{Theorem composition series}}]
    To begin with, $(M_1,0)$ and $(0,M_2)$ satisfy \eqref{Eqn norm equality condition for general mu= (X,Y) in Case (N)} by $\lambda\in P^{\pm}$.
    Assume $\Big( M_1,\  M_2-\frac{2c}{d}M_1\ \in \mathbb{N} \Big)$ to be the unique solution to \eqref{Eqn norm equality condition for general mu= (X,Y) in Case (N)} in $\mathbb{N}\{\alpha_1\ , \ \alpha_2\}$.
  
  The maximality of $f_2^{M_2-\frac{2c}{d}M_1}f_1^{M_1}m_{\lambda}$ follows by \cite[Proposition 3.7(a)]{T-P}.
  For other maximal vectors, we consider the $\mathfrak{g}$-h.w.m. $V=\frac{M(\lambda)}{U(\mathfrak{n}^-)f_1^{M_1}m_{\lambda}}$, with a h.w. vector $v_{\lambda}\in V_{\lambda}$; note $f_2^{M_2-\frac{2c}{d}M_1}f_1^{M_1}v_{\lambda}=0$.
  By the action of Casimir element, maximal vectors in $V$ can have weights $\lambda-M_1\alpha_1-\Big(M_2-\frac{2c}{d}M_1\Big)\alpha_2$ or $\lambda-M_2\alpha_2$; note $V_{\lambda-M_2\alpha_2}$ consists of maximal vectors.

We write for convenience $f_i$ as $f_{\alpha_i}$, and work with the following:\ \  1) An ordered basis $B(\mathfrak{n}^-)$\\ $= \ \left\{ f_{\alpha_1}(1)=f_{\alpha_1}\right\}\cup\left\{ f_{\alpha_2}(1)=f_{\alpha_2}\right\}\cup \left\{f_{\alpha_1+\alpha_2}(1)=[f_{\alpha_1},f_{\alpha_2}]\}\cup \cdots \cup \{ f_{\beta}(1),\ldots, f_{\beta}(\dim(\mathfrak{g}_{\beta}))\right\} \cup \cdots $.\\
Here, for $\beta\in \Delta^+
$ and $\{f_{\beta}(1),\ldots, f_{\beta}(\dim(\mathfrak{g}_{\beta}))$ is an ordered Lyndon word basis of the root space $\mathfrak{n}^-_{-\beta}$.
In the increasing PBW order in $B(\mathfrak{n}^-)$ from left to right, the heights of $\beta$'s do not decrease.\\
2) Thereby, we have the PBW monomial basis for $U(\mathfrak{n}^-)=\mathbb{C}\Big\{f_1\ ;\  f_2\  ; \ f_2f_1\ ;\  \cdots\ ;\\
\big(\cdots \underbrace{\big[f_2,[f_2,f_1]\big]\cdots \big[f_2,[f_2,f_1]\big]}_{n_{221}}$ $ \underbrace{[f_2,f_1]\cdots [f_2,f_1]}_{n_{12}} \underbrace{f_2\cdots f_2}_{n_2}\underbrace{f_1\cdots f_1}_{n_1}\big)\ ; \ \cdots \Big|\  n_1,n_2,n_{1,2},n_{221},\ldots \in \mathbb{Z}_{\geq 0}\Big\}$.
3) Fix a PBW monomial $f_{\beta_{i_1}}(j_1)\cdots f_{\beta_{i_m}}(j_m)$, or its image $f_{\beta_{i_1}}(j_1)\cdots f_{\beta_{i_m}}(j_m)m_{\lambda}$ with $\beta_{i_1},\ldots , \beta_{i_m}$ and $j_1,\ldots, j_m$ not necessarily distinct.
For it, we associate the (ordered) positive root sequence $[\beta_{i_1},\ldots, \beta_{i_m}]$ -- so in particular the heights of roots do not increase from left to right ($i_1\rightsquigarrow i_m$) -- and consider the lexicographic ordering on the space of these root sequences induced by the ordering in $B(\mathfrak{n}^-)$.
Note $\big[2\alpha_1+\alpha_2; \alpha_1+\alpha_2; \alpha_2;\alpha_1\big]$ is an example of ordered sequence decreasing (not just in heights) from left to right. 
We need to work with root sequences to tackle PBW monomial expansions of elements in $U(\mathfrak{n}^-)$.

We begin by fixing a homogeneous element  $x=\sum_{\gamma=(\gamma_1.\ldots, \gamma_k)}c_{\gamma}f_{\gamma_1}\cdots f_{\gamma_k}\in U(\mathfrak{n}^-)_{\big(M_1\alpha_1+\big(M_2-\frac{2c}{d}M_1\big)\alpha_2\big)}$ for PBW monomials $f_{\gamma_1}\cdots f_{\gamma_k}$ \big($f_{\gamma_i}$s allowed to repeat\big) and their ordered root sequences $\gamma=\big[\gamma_1;\ldots ; \gamma_k\big]$, and $c_{\gamma}\in \mathbb{C}$. 
We explore the maximality of $xm_{\lambda}\in M(\lambda)$ or its projection $xv_{\lambda}\in V$.
%$M(\lambda)_{\lambda}$   Fix (any) weight vector $xv_{\lambda}\neq 0$ for $x\in U(\mathfrak{n}^-)$ in $\lambda-M_1\alpha_1-\Big(M_1-\frac{2c}{d}M_1\Big)\alpha_2$-weight space of $V$ with $e_2xv_{\lambda}=0$.   We will show below that $x=0$, which proves the lemma.
  
  Let us observe $e_2$'s action on each of the above summands of $xm_{\lambda}$: $e_2\cdot f_{\gamma_1}\cdots f_{\gamma_k}m_{\lambda}=\sum_{i=1}^k f_{\gamma_1}\cdots f_{\gamma_{i-1}}$ $\big[e_2, f_{\gamma_i}\big]f_{\gamma_{i+1}}\cdots f_{\gamma_k}m_{\lambda}$.
  If $\height(\gamma_i-\alpha_2)> \height(\gamma_{i+1})$ then $f_{\gamma_1}\cdots f_{\gamma_{i-1}} \big[e_2, f_{\gamma_i}\big]f_{\gamma_{i+1}}\cdots f_{\gamma_k}m_{\lambda}$ (when non-zero) is a PBW basis vector.
  The similar assertion is true when $\height(\gamma_i-\alpha_2)=\height(\gamma_{i+1})$ and $f_{\gamma_{i+1}}$ in PBW ordering is below all the basis root vectors $f_{\gamma_i-\alpha_2}(j)$'s that span $[e_2, f_{\gamma_i}]$. %\big(so that $f_{\gamma_1}\cdots f_{\gamma_{i-1}}[e_2,\gamma_i]f_{\gamma_{i+1}}\cdots f_{\gamma_k}$ is also a PBW monomial\big) .
  Next if $\height(\gamma_i-\alpha_2)=0$ then clearly $f_{\gamma_i}=f_2$ and importantly  $f_{\gamma_{i+1}}\cdots f_{\gamma_k} = \underbrace{f_2\cdots f_2}_{i_2 \geq 0} \underbrace{f_1\cdots f_1}_{i_1\geq 0}$ for $i_1+i_2=k-i\geq 0$, and so $f_{\gamma_1}\cdots f_{\gamma_{i-1}} \underset{=\alpha_2^{\vee}}{\big[e_2, f_{\gamma_i}\big]}f_{\gamma_{i+1}}\cdots f_{\gamma_k}m_{\lambda}$ is ($i_2d+i_1c$)-times the PBW basis vector $f_{\gamma_1}\cdots f_{\gamma_{i-1}}f_{\gamma_{i+1}}\cdots f_{\gamma_k}m_{\lambda}$.
  Finally suppose either $0<\height(\gamma_i-\alpha_2)< \height(\gamma_{i+1})$, or $\height(\gamma_i -\alpha_2)=\height(\gamma_{i+1})$ but not all the Lyndon basis root vectors occurring in expansion of $[e_2, f_{\gamma_i}]$ are above $f_{\gamma_{i+1}}$ in our PBW ordering.
  Then we appeal to the following observation.
\begin{claim}
As in the beginning of the proof of Theorem \ref{Theorem composition series}, we fix a BKM $\mathfrak{g}$, and a root vector basis $B(\mathfrak{n}^-)$ for $\mathfrak{n}^-$ with the order of increasing (positive) root heights and the corresponding PWB monomial basis for $U(\mathfrak{n}^-)$.
Recall the lexicographic ordering on root sequences for PBW monomials above.
%Fix an ordered basis $B(\mathfrak{n}^-)\ = \ \left\{ f_{\alpha_1}(1)=f_{\alpha_1};\ f_{\alpha_2}(1)=f_{\alpha_2};\ f_{\alpha_1+\alpha_2}(1)=[f_{\alpha_1},f_{\alpha_2}];\ \ldots; f_{\beta}(1),\ldots, f_{\beta}(\dim(\mathfrak{g}_{\beta}));\ldots \right\}$, for $\beta\in \Delta^+ $ and $f_{\beta}(1),\ldots, f_{\beta}(\dim(\mathfrak{g}_{-\beta}))$ an ordered basis of the root space $\mathfrak{n}^-_{-\beta}$, and such that heights of $\beta$s increase from left to right in $B(\mathfrak{n}^-)$; this basis for $\mathfrak{n}^-$ is written in increasing order when read from left to right. For each PBW monomial $f_{\beta_{i_1}}(j_1)\cdots f_{\beta_{i_m}}(j_m)$, we associate the (ordered) root sequence $[\beta_{i_1},\ldots, \beta_{i_m}]$ -- so in particular the (negative) heights of roots decrease from left to right -- and consider the lexicographic ordering on the space of these root sequences induced by the ordering in $B(\mathfrak{n}^-)$. 
    Fix the product of $f_{\beta_1}=f_{\beta_1}(j_1),\ldots,  f_{\beta_k}=f_{\beta_k}(j_k)\in B(\mathfrak{n}^-)$ not necessarily in the PBW order, for $\beta_1,\ldots, \beta_k\in \Delta^+$ and some $j_1,\ldots, j_k$.
Let $\overline{f_{\beta_1}\cdots f_{\beta_k}}$ denote the PBW monomial obtained by permuting $f_{\beta_1}, \ldots, f_{\beta_k}$.
   Then, the expansion of  $f_{\beta_1}\cdots f_{\beta_k}\in U(\mathfrak{n}^-)$ as a linear combination of PBW monomials, has following properties:
     i) $\overline{f_{\beta_1} \cdots f_{\beta_k}}$ occurs in the sum with coefficient 1.
   ii) The other summands occur as products of Lie words -- i.e. of iterated Lie brackets -- on some of $f_{\beta_1},\ldots, f_{\beta_k}$.
   iii) So $\overline{f_{\beta_1}\cdots f_{\beta_k}}$ is the unique PBW monomial in the sum, whose root sequence is the least in the lexicographic ordering among the other (non-zero) summands' sequences in ii).
%   be any fixed (possibly un-ordered) subset of negative root vectors in $\mathfrak{n}^-$.     Then $f_1\cdots f_k\in U(\mathfrak{n}^-)$ is a linear combination of PBW monomials of the form $F_1\cdots F_m$, $m\leq k$ where each $F_i$ is a right-normed Lie word on some subset of elements $f_1,\ldots ,f_k$. 
\end{claim} 
  \begin{proof}[{\bf \textnormal Proof of the claim}]
      We prove the claim for products $f_{\beta_1} \cdots f_{\beta_k}$ by inducting on their lengths $k\geq 1$.
      In the base step $k=1$, $f_{\beta_1}$ is already a PBW monomial by the choice of $f_{\beta_i}=f_{\beta_i}(j_i)$s in the statement of the claim.\\
      Induction step: Assume $k>1$ in the product. 
      Pick $s\in \{1,\ldots, k\}$ such that $f_{\beta_s}$ is the least among $f_{\beta_1}, \ldots, f_{\beta_k}$  (possibly $s=k$) in our ordering in $B(\mathfrak{n}^-)$.
      Suppose $1<s$, and let $\overline{f_{\beta_s}\cdots f_{\beta_k}}=F_1 f_{\beta_s}$ with $F_1$ also a PBW monomial.
    By the induction hypothesis $f_{\beta_s}\cdots f_{\beta_k}$ can be written as a desired sum $1 \overline{f_{\beta_s}\cdots f_{\beta_k}}+\sum_{t=2}^r c_tF_t$ of PBW monomials, with $c_t\neq 0$ and with the root sequences of $F_t$'s strictly above that of $\overline{f_{\beta_s}\cdots f_{\beta_k}}$ in the lexicographic ordering.
    Observe that the number of root vector factors in each $F_t$ is strictly less than $k-s+1$; one can directly see this, or prove it by adding it to the induction hypothesis.  
    In turn by the induction hypothesis applied to $f_{\beta_1}\cdots f_{\beta_{s-1}}F_t$ for $t\geq 2$, the resulting summands other than $\overline{f_{\beta_1}\cdots f_{\beta_{s-1}}F_t}$  in $f_{\beta_1}\cdots f_{\beta_{s-1}}F_t$ have root sequences strictly above that of  $\overline{f_{\beta_1}\cdots f_{\beta_{s-1}}F_t}$ along with all the desired properties.
    Observe using the transitivity property of the lexicographic ordering that $\overline{f_{\beta_1}\cdots f_{\beta_{s-1}}F_t}$ have their root sequences strictly above that of $\overline{f_{\beta_1}\cdots f_{\beta_k}}$ for all $t$.
    Moreover $\overline{f_{\beta_1}\cdots f_{\beta_{s-1}} \big( \overline{f_{\beta_s}\cdots f_{\beta_k}}\big)} = \overline{f_{\beta_1}\cdots f_{\beta_k}}$ can be seen by comparing the root sequences, and we are done. 
    Finally, we assume that $s=1$.
    Then $f_{\beta_1}(f_{\beta_2}\cdots f_{\beta_k})=(f_{\beta_2}\cdots f_{\beta_k})f_{\beta_1}+\sum_{i=2}^k f_{\beta_2}\cdots f_{\beta_{i-1}}[f_{\beta_1} ,\ f_{\beta_i}]f_{\beta_{i+1}}\cdots f_{\beta_k}$.
    Now each of $f_{\beta_2}\cdots f_{\beta_k}
    $ and $f_{\beta_2}\cdots f_{\beta_{i-1}}[f_{\beta_1}, f_{\beta_i}]f_{\beta_{i+1}}\cdots f_{\beta_k}$ (when non-zero) have the number of root vector factors $\leq k-1$,
    and so applying induction hypothesis finishes the proof; here the minimality of the root sequence for $\overline{f_{\beta_1}\cdots f_{\beta_k}}$ among those of the others is once again seen by the transitive property of the ordering.     
  \end{proof}
 We resume the proof of the theorem.
Observe for any PBW monomial basis vector $f_{\gamma_1}\cdots f_{\gamma_k}m_{\lambda}\in M(\lambda)_{\lambda-M_1\alpha_1-\big(M_2-\frac{2c}{d}M_1\big)\alpha_2}$ other than $[f_2,f_1]f_2^{M_2-\frac{2c}{d}M_1-1}f_1^{M_1-1}m_{\lambda}$ and $f_2^{M_2-\frac{2c}{d}M_1}f_1^{M_1}m_{\lambda}$, that the root sequences are strictly above \allowdisplaybreaks $\big[\alpha_1+\alpha_2 ;\underbrace{\alpha_2;\ldots; \alpha_2}_{M_2-\frac{2c}{d}M_1-1};\underbrace{\alpha_1;\ldots; \alpha_1}_{M_1-1}\big] \  \underset{\text{lexico.}}{ > } \ \big[\underbrace{\alpha_2;\ldots; \alpha_2}_{M_2-\frac{2c}{d}M_1};\underbrace{\alpha_1;\ldots; \alpha_1}_{M_1}\big]$ in our lexicograhic ordering.
Moreover for such $f_{\gamma_1}\cdots f_{\gamma_k}$, observe that the same inequality in the previous line holds true between the root sequence of $\overline{f_{\gamma_1}\cdots f_{\gamma_{i-1}} \big[e_2, f_{\gamma_i}\big]f_{\gamma_{i+1}}\cdots f_{\gamma_k}}$ -- thereby for every PBW monomial in the expansion for $f_{\gamma_1}\cdots f_{\gamma_{i-1}} \big[e_2, f_{\gamma_i}\big]f_{\gamma_{i+1}}\cdots f_{\gamma_k}$ -- and the sequence $\big[\underbrace{\alpha_2;\ldots; \alpha_2}_{M_2-\frac{2c}{d}M_1-1} ;\underbrace{\alpha_1;\ldots; \alpha_1}_{M_1}\big]$, for all $i$.
So $f_2^{M_2-\frac{2c}{d}M_1-1}f_1^{M_1}m_{\lambda}$ does not occur in the PBW-expansion of such 
$f_{\gamma_1}\cdots f_{\gamma_{i-1}} \big[e_2, f_{\gamma_i}\big]f_{\gamma_{i+1}}\cdots f_{\gamma_k}m_{\lambda}$.
However, note that $e_2\ \cdot\  [f_2,f_1]f_2^{M_2-\frac{2c}{d}M_1-1}f_1^{M_1-1}m_{\lambda}\ = \ c f_1f_2^{M_2-\frac{2c}{d}M_1-1}f_1^{M_1-1}m_{\lambda}- c\big(M_2-\frac{2c}{d}M_1-1\big)\big(2M_1+\frac{d}{2c}-1\big) [f_2,f_1]f_2^{M_2-\frac{2c}{d}M_1-2}f_1^{M_1-1}m_{\lambda}$, and further in its PBW expansion we observe $f_2^{M_2-\frac{2c}{d}M_1-1}f_1^{M_1}m_{\lambda}$ to occur with coef. $c$ by the above claim.

In parallel, we observe the action of $e_1$ on summands in $xm_{\lambda}$.
Note that $\Big(\lambda-(M_1-1)\alpha_1-\Big(M_2-\frac{2c}{d}M_1\Big)\alpha_2\Big)$-weight space of $U(\mathfrak{g})f_1^{M_1}m_{\lambda}$ is 0.
 So, $e_1 x v_{\lambda}=0\in V$ iff $e_1x m_{\lambda}=0\in M(\lambda)$.
  \\ Next, $e_1 [f_2, f_1]f_2^{M_2-\frac{2c}{d}M_1-1}f_1^{M_1-1}m_{\lambda}= -af_2^{M_2-\frac{2c}{d}M_1}f_1^{M_1-1}m_{\lambda}- \frac{b}{2}(M_1-1) [f_2,f_1]f_2^{M_2-\frac{2c}{d}M_1-1}f_1^{M_1-2}m_{\lambda}$.
 Moreover as in the $e_2$-action, observe for each 
 $i$, by comparing the lexicographic orders that the PBW-expansion of $f_{\gamma_1}\cdots f_{\gamma_{i-1}}[e_1, f_{\gamma_i}] f_{\gamma_{i+1}}\cdots f_{\gamma_k}m_{\lambda}$ involves $f_2^{M_2-\frac{2c}{d}M_1}f_1^{M_1-1}m_{\lambda}$, iff $f_{\gamma_1}\cdots f_{\gamma_k}= [f_2, f_1]f_2^{M_2-\frac{2c}{d}M_1-1} f_1^{M_1-1}$.
 So for $e_1xm_{\lambda}=0$, $x$ must be in the span of all those PBW monomials other than $[f_2, f_1]f_2^{M_2-\frac{2c}{d}M_1}f_1^{M_1-1}m_{\lambda}$.
 The maximality of $f_1^{M_1}m_{\lambda}$ says $f_2^{M_2-\frac{2c}{d}M_1}f_1^{M_1}m_{\lambda}$ is maximal.
 
By the observations in the above two paragraphs, we see for $e_1xm_{\lambda}=e_2xm_{\lambda}=0$ that $xm_{\lambda}= y +pf_2^{M_2-\frac{2c}{d}M_1}f_1^{M_1}m_{\lambda}$, for a maximal vector $y$ in the span of the PBW basis vectors other than $[f_2,f_1]f_2^{M_2-\frac{2c}{d}M_1-1} f_1^{M_1-1}m_{\lambda}$ and $f_2^{M_2-\frac{2c}{d}M_1}f_1^{M_1}m_{\lambda}$, and $p\in \mathbb{C}$.
 Hence $xv_{\lambda}$ is maximal in $V$ iff $xm_{\lambda}\in M(\lambda)$ is so.
 
Let $N\neq 0$ be a submodule of $M(\lambda)$ with trivial $\lambda-M_1\alpha_1$ and $\lambda-M_2\alpha_2$ weight spaces.
Then observe by the uniqueness of solution $\Big(M_1, M_2-\frac{2c}{d}M_1\Big)$ in $\mathbb{N}\{\alpha_1\}\oplus \mathbb{N}\{\alpha_2\}$, the following : i) $N$ is generated by the maximal weight vectors -- possibly involving the maximal vector $f_2^{M_2-\frac{2c}{d}M_1}f_1^{M_1}m_{\lambda}$ -- from $M(\lambda)_{\lambda-M_1\alpha_1-\big(M_2-\frac{2c}{d}M_1\big)\alpha_2}$ space. ii) Moreover $N$ is the direct sum of some copies of $L\Big(\lambda-M_1\alpha_1-\big(M_2-\frac{2c}{d}M_1\big)\alpha_2\Big)$; the submodule generated by $f_2^{M_2-\frac{2c}{d}M_1} f_1^{M_1}m_{\lambda}$ is isomorphic to this simple. 
iii) Thus the submodule $N\cap U(\mathfrak{g})f_1^{M_1}m_{\lambda}= N\cap U(\mathfrak{g})f_2^{M_2-\frac{2c}{d}M_1}f_1^{M_1}m_{\lambda}$ has a compliment (possibly 0) in $N$.
In parallel notice $\lambda-M_1\alpha_1- \big(M_2-\frac{2c}{d} M_1\big)\alpha_2$ is not a weight of the submodule $U(\mathfrak{n}^-)f_2^{M_2}m_{\lambda}$.
Moreover, as there is no solution to \eqref{Eqn dot action by any positive root} below this weight,  $N \supsetneqq  N\cap U(\mathfrak{n}^-)f_2^{M_2}m_{\lambda}=\{0\}$.
Note by the same reasoning $ \big(U(\mathfrak{n}^-)f_1^{M_1}m_{\lambda}\big)\cap \big( U(\mathfrak{n}^-)f_2^{M_2}m_{\lambda}\big)=0$. 
Now all the claims in the theorem can be easily seen.
\end{proof}

\bigskip

  \noindent
\address{(Souvik Pal)} \textsc{Department of Sciences and Humanities, CHRIST University, Bangalore 560 074, India.}
  \textit{E-mail address}: \email{\texttt{pal.souvik90@gmail.com, souvik.pal@christuniversity.in}} \smallskip\\
\address{(Supravat Sarkar)} \textsc{Department of Mathematics, Fine Hall, Princeton University, Princeton, NJ 700108, USA.}
  \textit{E-mail address}: \email{\texttt{ss6663@princeton.edu}}   \bigskip\\
       	 \address{(G. Krishna Teja) \textsc{Stat. Math. Unit, Indian Statistical Institute Bangalore Center,  560059,  India.}}
	 \  \  	 \textit{E-mail address}: \email{\texttt{tejag@alum.iisc.ac.in, tejag\_pd@isibang.ac.in}}

\end{document}